\documentclass[letter,11pt]{article}

\usepackage[english]{babel}
\usepackage[utf8]{inputenc}

\usepackage[margin=1.3 in, top=1.1 in, bottom= 1.2 in]{geometry}

\makeatletter
\g@addto@macro\normalsize{%
  \setlength\abovedisplayskip{7pt}
  \setlength\belowdisplayskip{7pt}
  \setlength\abovedisplayshortskip{7pt}
  \setlength\belowdisplayshortskip{7pt}
}
\makeatother

\usepackage{tocloft}

\usepackage{amsfonts}
\usepackage{mathrsfs}
\usepackage{bbm}
\usepackage{latexsym}
\usepackage{math dots}
\usepackage{amssymb}
\usepackage{mathtools}
\usepackage{relsize}
\usepackage{nicefrac}

\usepackage{todonotes}
\usepackage{enumitem}
\setlist{nolistsep} 	
\usepackage{amsthm}

\usepackage{xcolor}
\definecolor{Color1}{rgb}{0.0, 0.42, 0.47}
\definecolor{Color2}{rgb}{0.78, 0.11, 0.0}

\usepackage{titlesec}
\titleformat{\section}
  {\large\center\bfseries}
  {\thesection.}{.7em}{}
\titlespacing*{\section}{0pt}{3.5ex plus 0ex minus 0ex}{1.5ex plus 0ex}
\titleformat{\subsection}
  {\center\bfseries}
  {\thesubsection.}{.7em}{}
\titlespacing*{\subsection}{0pt}{3.5ex plus 0ex minus 0ex}{1.5ex plus 0ex}
\titleformat{\subsubsection}
  {\center\bfseries}
  {\thesubsubsection.}{.7em}{}
\titlespacing*{\subsubsection}{0pt}{3.5ex plus 0ex minus 0ex}{1.5ex plus 0ex}

\addto\captionsenglish{}

\usepackage{titling}
\usepackage[linktocpage=true]{hyperref}
\usepackage[capitalize]{cleveref}
\hypersetup{citecolor = Color2,colorlinks,
			linkcolor = black,
			urlcolor = Color2}

\newtheoremstyle{plain}{3mm}{3mm}{\slshape}{}{\bfseries}{.}{.5em}{}
\newtheoremstyle{definition}{2mm}{2mm}{}{}{\bfseries}{.}{.5em}{}
\theoremstyle{plain}
	
\newtheorem{Theorem}{Theorem}

\newtheorem{Lemma}[Theorem]{Lemma}
\newtheorem{Proposition}[Theorem]{Proposition}
\newtheorem{Corollary}[Theorem]{Corollary}
\newtheorem{Conjecture}{Conjecture}
\newtheorem{Question}{Question}
\theoremstyle{definition}
\newtheorem{Definition}[Theorem]{Definition}
\newtheorem{Remark}[Theorem]{Remark}
\newtheorem{Example}[Theorem]{Example}

\theoremstyle{plain} 
\newcounter{MainTheoremCounter}

\newtheorem{Maintheorem}[MainTheoremCounter]{Theorem}

\theoremstyle{plain}
\newtheorem*{namedthm}{\namedthmname}
\newcounter{namedthm}
\makeatletter
	\newenvironment{named}[2]
	{\def\namedthmname{#1}
	\refstepcounter{namedthm}
	\namedthm[#2]\def\@currentlabel{#1}}
	{\endnamedthm}
\makeatother

\usepackage{chngcntr}
\counterwithin{Theorem}{section}

\numberwithin{equation}{section}

\allowdisplaybreaks

\newcommand{\Cesaro}{Ces\`{a}ro}
\newcommand{\Erdos}{Erd\H{o}s}
\newcommand{\Folner}{F\o{}lner}

\newcommand{\Katai}{K\'{a}tai}
\newcommand{\Mobius}{M\"{o}bius}

\newcommand{\Halasz}{Hal\'{a}sz}

\newcommand{\Turan}{Tur{\'a}n}

\newcommand{\Oh}{{\rm O}}
\newcommand{\oh}{{\rm o}}
\newcommand{\N}{\mathbb{N}}
\newcommand{\Z}{\mathbb{Z}}
\newcommand{\R}{\mathbb{R}}
\newcommand{\C}{\mathbb{C}}
\newcommand{\Q}{\mathbb{Q}}
\newcommand{\T}{\mathbb{T}}

\newcommand{\Cont}{\mathsf{C}}

\newcommand{\define}[1]{{\itshape #1}}

\renewcommand{\epsilon}{\varepsilon}
\renewcommand{\leq}{\leqslant}
\renewcommand{\geq}{\geqslant}
\renewcommand{\setminus}{\backslash}

\renewcommand{\P}{\mathbb{P}}

\newcommand{\E}{\mathbb{E}}

\newcommand{\1}{1}

\renewcommand{\d}{~\mathsf{d}}

\newcommand{\mob}{\boldsymbol{\mu}}
\newcommand{\lio}{\boldsymbol{\lambda}}
\newcommand{\tot}{\boldsymbol{\varphi}}

\newcommand{\BEu}[1]{\underset{#1}{\mathlarger{\mathlarger{\mathbb{E}}}^{~}}\,}
\newcommand{\BEul}[1]{\underset{#1}{\mathlarger{\mathlarger{\mathbb{E}}}^\log}\,}
\newcommand{\POH}{{\mkern 0mu\times\mkern-.3mu}}

\begin{document}

\title{\bfseries Dynamical generalizations of the Prime Number Theorem and disjointness of additive and multiplicative semigroup actions}
\author{Vitaly Bergelson \and Florian K.\ Richter}

\date{\small \today}
\maketitle
\begin{abstract}
We begin by establishing two ergodic theorems which
have among their corollaries numerous classical results from multiplicative number theory, including the Prime Number Theorem, a theorem of Pillai-Selberg, a theorem of \Erdos{}-Delange, the mean value theorem of Wirsing, and special cases of the mean value theorem of \Halasz{}.
Then, by building on the ideas behind our ergodic results, we recast Sarnak's \Mobius{} disjointness conjecture in a new dynamical framework.
This naturally leads to an extension of Sarnak's conjecture that focuses on the disjointness of actions of $(\mathbb{N},+)$ and $(\mathbb{N},\cdot)$.
We substantiate this extension by providing proofs of several special cases.
\end{abstract}

\small
\tableofcontents
\thispagestyle{empty}
\normalsize


\section{Introduction}

One of the fundamental challenges in number theory is to understand the intricate way in which the additive and multiplicative structures of natural numbers intertwine. It is the purpose of this paper to offer a new dynamical perspective on this topic.

This introduction is divided into two subsections. In Subsection \ref{sec_DG_PNT} we present two new ergodic theorems, Theorems~\ref{thm_dynamical_MVT_Omega} and~\ref{thm_dynamical_MVT_fg_sue}, which can be viewed as dynamical amplifications of various classical number-theoretic results including the Prime Number Theorem.
In Subsection~\ref{sec_beyond_sarnak} we take a closer look at the independence of additive and multiplicative structures in $\N\coloneqq\{1,2,3,\ldots\}$.
In particular, we interpret some classical theorems in multiplicative number theory as manifestations of additive-multiplicative independence. This leads to a formulation of 
an extended form of Sarnak's \Mobius{} disjointness conjecture, which is supported by  Theorems~\ref{thm_ortho_fg_sue_nilsequences} and~\ref{thm_ortho_fg_sue_horocycle} in that subsection.

\subsection{Dynamical generalizations of the Prime Number Theorem}
\label{sec_DG_PNT}

Let $\Omega(n)$ denote the number of prime factors of a natural number $n\in\N$ counted with multiplicities.
One of the central themes in multiplicative number theory is the study of the asymptotic distribution of the values of $\Omega(n)$. 
It has a long and rich history and is closely related to fundamental questions about the prime numbers.

{For example, the Prime Number Theorem is equivalent to the assertion that, asymptotically, there are as many $n\in\N$ for which $\Omega(n)$ is even as there are for which $\Omega(n)$ is odd. 
This fact dates back to von Mangoldt \cite[p.\ 852]{vonMangoldt97} and Landau \cite[pp.\ 571--572,\, 620-621]{Landau09b} and can also be expressed using the classical \define{Liouville function} $\lio(n)\coloneqq(-1)^{\Omega(n)}$ as}
\begin{equation}
\label{eqn_PNT_lio}
\lim_{N\to\infty}~\frac{1}{N}\sum_{n=1}^N \lio(n)\,=\,0.
\end{equation}

Given that the values of $\Omega(n)$ equally distribute over evens and odds, it is natural to ask whether the same is true over other residue classes.
This question is answered by the Pillai-Selberg Theorem \cite{Pillai40, Selberg39}, a classical extension of the Prime Number Theorem asserting that for all $m\in\N$ and all $r\in\{0,\ldots,m-1\}$ the set $\{n\in\N: \Omega(n)\equiv {r}\bmod{m}\}$ 
has asymptotic density equal to $1/m$.

Another classical 
result in this direction states that for any irrational $\alpha$ the sequence $\Omega(n)\alpha$, $n\in\N$, is uniformly distributed mod~${1}$. This was first mentioned by \Erdos{} in \cite[p.\ 2, lines 4--5]{Erdos46} without a proof, although \Erdos{} adds that ``the proof is not easy''.
A proof was later published by Delange in \cite{Delange58} (see also \cite[Section 2.4]{Wirsing61} and \cite{Elliott71}). 
The \Erdos{}-Delange Theorem complements the above mentioned result of Pillai and Selberg, as it implies that the values of $\Omega(n)$ are evenly distributed among ``generalized arithmetic progressions'', i.e., for all $\alpha\in\R\setminus\Q$ with $\alpha\geq 1$ and $\beta\in\R$ the asymptotic density of the set of $n$ for which $\Omega(n)$ belongs to $\{\lfloor \alpha m+\beta \rfloor: m\in\N\}$ equals $1/\alpha$.

We will presently formulate our first result, which can be viewed as an ergodic theorem along the sequence $\Omega(n)$. It contains the Prime Number Theorem, the Pillai-Selberg Theorem, and the \Erdos{}-Delange Theorem as rather special cases.
Let $X$ be a compact metric space and $T\colon X \to X$ a continuous map. Since
\begin{equation}
\label{eqn_additive_action}
T^m\circ T^n\,=\, T^{m+n},\qquad \forall m,n\in\N,
\end{equation}
the transformation $T$ naturally induces an action of $(\N,+)$ on $X$. We call the pair $(X,T)$ an \define{additive topological dynamical system}.
A Borel probability measure $\mu$ on $X$ is called \define{$T$-invariant} if $\mu(T^{-1}A)=\mu(A)$ for all Borel measurable subsets $A\subset X$. By the Bogolyubov-Krylov theorem (see for example \cite[Corollary 6.9.1]{Walters82}), every additive topological dynamical system $(X,T)$ possesses at least one $T$-invariant Borel probability measure.
If $(X,T)$ admits only one such measure then the system is called \define{uniquely ergodic}.
For convenience, we will use $(X,\mu,T)$ to denote a uniquely ergodic additive topological dynamical system $(X,T)$ together with its unique invariant probability measure $\mu$.

\begin{Maintheorem}
\label{thm_dynamical_MVT_Omega}
Let $(X,\mu,T)$ be uniquely ergodic.
Then
\begin{equation}
\label{eqn_ud_of_Omega_1}
\lim_{N\to\infty}\frac{1}{N}\sum_{n=1}^N f\big(T^{\Omega(n)}x\big)
~=~\int f\d\mu
\end{equation}
for every every $x\in X$ and $f\in\Cont(X)$.
\end{Maintheorem}

One can interpret \cref{thm_dynamical_MVT_Omega} as saying that for any uniquely ergodic system $(X,\mu,T)$ and any point $x\in X$ the orbit $T^{\Omega(n)}x$ is uniformly distributed in the space $X$ with respect to $\mu$.

It is straightforward to recover the Prime Number Theorem -- in the shape of \eqref{eqn_PNT_lio} -- from \cref{thm_dynamical_MVT_Omega}. Indeed, consider the additive topological dynamical system $(X,T)$ where $X=\{0,1\}$ and $T\colon x\mapsto {x+1}\pmod{2}$; this system is often referred to as \define{rotation on two points}. With its help we can write the Liouville function $\lio(n)$ in the form $f(T^{\Omega(n)}x)$ by taking $x=0$ and defining $f\colon \{0,1\}\to \{-1,1\}$ as $f(0)=1$ and $f(1)=-1$.
Since rotation on two points is uniquely ergodic with respect to the unique invariant probability measure defined by $\mu(\{0\})=\mu(\{1\})=1/2$, and since $f$ has zero integral with respect to this measure, we see that \eqref{eqn_ud_of_Omega_1} implies \eqref{eqn_PNT_lio}. 
In a totally similar way one can recover the Pillai-Selberg Theorem from \cref{thm_dynamical_MVT_Omega} by considering a cyclic \define{rotation on $m$ points} instead of a rotation on two points.

To see that the \Erdos{}-Delange Theorem also follows from \cref{thm_dynamical_MVT_Omega}, consider the uniquely ergodic system $(X,\mu,T)$ where $X$ is the torus $\T\coloneqq \R/\Z$, $T\colon x\mapsto {x+\alpha}\pmod{1}$ for some $\alpha\in\R\setminus\Q$, and $\mu$ is the normalized Haar measure on $\T$; this system is usually called \define{rotation by $\alpha$}. 
Let $x=0$ and, for $h\in\Z\setminus\{0\}$, let $f(x)=e(hx)$, where $e(x)$ is shorthand for $e^{2\pi i x}$. Then, by \cref{thm_dynamical_MVT_Omega},
\begin{equation*}
\label{eqn_ud_Omega_torus}
\lim_{N\to\infty}\frac{1}{N}\sum_{n=1}^N e\left( h\Omega(n)\alpha \right) ~=~0, \qquad\forall h\in\Z\setminus\{0\},
\end{equation*}
which in view of Weyl's equidistribution criterion (see \cite[\S{}1, Theorem 2.1]{KN74}) is equivalent to the assertion that $\Omega(n)\alpha$, $n\in\N$, is uniformly distributed ${}\bmod{1}$. 

\cref{thm_dynamical_MVT_Omega} can also be used to derive some new results. But before we go into more details about the various number-theoretic applications of \cref{thm_dynamical_MVT_Omega}, let us make a few relevant remarks regrading its proof.

\begin{Remark}
\label{rem_methods}
The classical proofs of the Pillai-Selberg Theorem and the \Erdos{}-Delange Theorem rely on sophisticated machinery from analytic number theory.
By way of contrast, our proof of \cref{thm_dynamical_MVT_Omega} is elementary and hinges on new ideas and combinatorial tools developed in \cref{sec_distr_orbits_along_Omega}.
Admittedly, a down-side of our ``soft'' approach is that we do not obtain any noteworthy asymptotic bounds.
\end{Remark}

\begin{Remark}
Our proof of \cref{thm_dynamical_MVT_Omega} can be adapted to give a new elementary proof of the Prime Number Theorem. This is carried out in \cite{Richter21}, where it is shown that
\begin{align}
\label{eqn_sequential_Omega}
\frac{1}{N}\sum_{n=1}^N a(\Omega(n)+1)
=\frac{1}{N}\sum_{n=1}^N a(\Omega(n))+\oh_{N\to\infty}(1)
\end{align}
holds for any bounded $a\colon\N\to\C$.
Although \eqref{eqn_sequential_Omega} is more general\footnote{Here is a short proof that \eqref{eqn_sequential_Omega} implies \eqref{eqn_ud_of_Omega_1}: By applying \eqref{eqn_sequential_Omega} $K$-times to the sequence $a(n)=f(T^nx)$, it follows that
\[
\frac{1}{N}\sum_{n=1}^N \frac{1}{K}\sum_{k=1}^K f(T^{\Omega(n)+k}x)
=\frac{1}{N}\sum_{n=1}^N f(T^{\Omega(n)}x)+\oh_{N\to\infty}(1)
\]
holds for all $K\in\N$. Then \eqref{eqn_ud_of_Omega_1} follows from the fact that for every uniquely ergodic system one has $\sup_{y\in X}|\int f\d\mu - \frac{1}{K}\sum_{k=1}^K f(T^ky)|=\oh_{K\to\infty}(1)$.} than \cref{thm_dynamical_MVT_Omega}, it is the dynamical character of the latter that caught our interest and will lead naturally to our other main results, Theorems~\ref{thm_dynamical_MVT_fg_sue}, \ref{thm_ortho_fg_sue_nilsequences}, and~\ref{thm_ortho_fg_sue_horocycle} below.
\end{Remark}

\begin{Remark}
\label{rem_proof_A_gen}
Let $\pi_k(N)$ denote the cardinality of the set $\{1\leq n\leq N: \Omega(n)=k\}$ and define $w_N(k)\coloneqq \pi_k(N)/N$.
An alternative way of proving \cref{thm_dynamical_MVT_Omega}, which is significantly different from the approach that we use in \cref{sec_distr_orbits_along_Omega}, consists of writing
\[
\frac{1}{N}\sum_{n=1}^N f\big(T^{\Omega(n)}x\big)\,=\,\frac{1}{N} \sum_{k\in\N} \pi_k(N)f\big(T^{k}x\big)\,=\,
\sum_{k\in\N} w_N(k) f\big(T^{k}x\big)
\]
and showing that the weights $w_N(k)=\pi_k(N)/N$ are asymptotically shift-invariant in the sense that $\sum_{k\in\N} |w_N(k+1)-w_N(k)|=\oh_{N\to\infty}(1)$.  
With some effort, the latter can be derived from work of \Erdos{}, who established good asymptotic bounds for $\pi_k(N)$ uniformly over $k$ \cite{Erdos48a}.
It follows that $\frac{1}{N}\sum_{n=1}^N f\big(T^{\Omega(n)}x\big)$ is asymptotically $T$-invariant, which implies $\lim_{N\to\infty}\frac{1}{N}\sum_{n=1}^N f\big(T^{\Omega(n)}x\big)=\int f\d\mu$ by unique ergodicity. 
The authors thank A.~Kanigowski and M.~Radziwi\l{}\l{} for bringing this approach to our attention. 
One of the advantages of our method is that
it nicely generalizes to a proof of \cref{thm_dynamical_MVT_fg_sue} (see also \cref{rem_proof_B_gen} below). 
\end{Remark}

\begin{Remark}
Motivated by \cref{thm_dynamical_MVT_Omega} and Birkhoff's Pointwise Ergodic Theorem, it is natural to wonder whether for any ergodic measure-preserving system $(X,\mathcal{B},\mu,T)$ and any integrable function $f$ the averages  
\[
\lim_{N\to\infty}\frac{1}{N}\sum_{n=1}^N f\big(T^{\Omega(n)}x\big)
\]
converge almost everywhere.
Somewhat surprisingly, the answer is no. In forthcoming work \cite{Loyd21arXiv}, it is shown that for any non-atomic ergodic probability measure-preserving system $(X,\mathcal{B},\mu,T)$ the sequence $T^{\Omega(n)}$ has the \define{strong sweeping-out property}, meaning there exists a measurable set $A\in\mathcal{B}$ such that for almost all $x\in X$ one has
\begin{align*}
\limsup_{N\to\infty}\frac{1}{N}\sum_{n=1}^N 1_A\big(T^{\Omega(n)}x\big)
=1
\qquad\text{and}\qquad
\liminf_{N\to\infty}\frac{1}{N}\sum_{n=1}^N 1_A\big(T^{\Omega(n)}x\big)
=0.
\end{align*}
\end{Remark}

Let us now formulate some new number-theoretic corollaries that can be derived from \cref{thm_dynamical_MVT_Omega}. 
In \cite{Weyl16}, Weyl proved that a polynomial sequence $Q(n)=c_k n^k + \ldots + c_1 n + c_0$, $n\in\N$, is uniformly distributed mod~${1}$ if and only if at least one of the coefficients $c_1,\ldots,c_k$ is irrational.
Furstenberg gave a dynamical proof of Weyl's result utilizing the fact that any sequence of the form $e(Q(n))$, where $Q$ is a real polynomial, can be generated dynamically with the help of \define{unipotent affine transformations}\footnote{A \define{unipotent affine transformation} $T\colon\T^d\to\T^d$ is a transformation of the form $T(x)\coloneqq Ax+b$ where $A$ is a $d\times d$ unipotent integer matrix and $b$ is an element in $\T^d$.} on tori (see \cite[pp. 67--69]{Furstenberg81a}).
By invoking Furstenberg's method, we can derive from \cref{thm_dynamical_MVT_Omega} the following variant of Weyl's theorem, which can be viewed as a polynomial generalization of the \Erdos{}-Delange Theorem.

\begin{Corollary}
\label{cor_polynomoial_ud_Omega}
Let $Q(n)=c_k n^k + \ldots + c_1 n + c_0$. Then $Q(\Omega(n))$, $n\in\N$, is uniformly distributed ${}\bmod{1}$ if and only if at least one of the coefficients $c_1,\ldots,c_k$ is irrational.
\end{Corollary}

\cref{thm_dynamical_MVT_Omega} also implies other results similar to \cref{cor_polynomoial_ud_Omega}. For example, one can show that if $\beta$ is irrational and $\alpha\in\R$ is rationally independent from $\beta$ then the sequences $\lfloor \Omega(n)\alpha\rfloor\beta$ and $\{\Omega(n)\alpha\}\Omega(n)\beta$, $n\in\N$, are uniformly distributed ${}\bmod{1}$, where $\{x\}\coloneqq x-\lfloor x\rfloor $ denotes the fractional part of a real number $x$. Indeed, the sequences $(\lfloor n\alpha\rfloor\beta)_{n\in\N}$ and $(\{n\alpha\}n\beta)_{n\in\N}$ belong the the class of so-called \define{generalized polynomials}, which is the class of functions generated by starting with conventional real polynomials and applying in an arbitrary order the operations of taking the integer part $\lfloor.\rfloor$, addition, and multiplication.\footnote{Generalized polynomials also appear in various other contexts under the name \define{bracket polynomials}, see for instance \cite{GTZ12}.}
In particular, the following are examples of generalized polynomials: $p_1+p_2\lfloor p_3\rfloor$, $\lfloor p_1\rfloor^2 \{p_2\lfloor p_3\rfloor +p_4\}$, and $ \lfloor \lfloor p_1 \rfloor p_2+ \{p_3\}^2 p_4\rfloor+ \{p_5\}\lfloor p_6\rfloor^3$, where $p_1,\ldots, p_6$ are any real polynomials.
As was shown in \cite[Theorem A]{BL07}, any bounded generalized polynomial can be written as $f(T^nx)$ where $(X,T)$ is a \define{nilsystem}\footnote{\label{ftnt_1}Let $G$ be a nilpotent Lie group, let $\Gamma$ be a discrete and co-compact subgroup of $G$, and take $X=G/\Gamma$. Note that $G$ acts continuously on $X$ via left-multiplciation. Let $a\in G$ be a fixed group element and define $T\colon X\to X$ as $T(x)\coloneqq ax$ for all $x\in X$. The resulting additive topological dynamical system $(X,T)$ is called a \define{nilsystem}.} and $f$ is piecewise polynomial. We will show in \cref{sec_appl_thmA} that this fact, combined with \cref{thm_dynamical_MVT_Omega}, implies the following extension of \cref{cor_polynomoial_ud_Omega}.

\begin{Corollary}
\label{cor_gneralized_polynomoial_ud_Omega}
Let $Q\colon\N\to\R$ be a generalized polynomial. Then the sequence $Q(\Omega(n))$ is uniformly distributed ${}\bmod{1}$ if and only if $Q(n)$ is uniformly distributed ${}\bmod{1}$.
\end{Corollary}
For an ample supply of concrete examples of generalized polynomials that are uniformly distributed ${}\bmod{1}$ we refer the reader to \cite{Haaland92,Haaland94,BL07,BKS20}.

Another corollary of \cref{thm_dynamical_MVT_Omega} concerns an analogue of an old theorem of Gelfond. For integers $q\geq 2$ and $n\geq 1$ let $s_q(n)$ denote the sum of digits of $n$ in base $q$, that is, $s_q(n)=\sum_{k\geq 0} a_k$ where $n= \sum_{k\geq 0} a_k q^k$ for $a_0,a_1,\ldots\in\{0,1,\ldots,q-1\}$.
Gelfond \cite{Gelfond68} showed that if $m$ and $q-1$ are coprime then for all $r\in\{0,1,\ldots,m-1\}$ the set of $n$ for which $s_q(n)\equiv {r}\bmod{m}$ has asymptotic density $1/m$. 
In \cref{sec_appl_thmA} we explain how one can combine \cref{thm_dynamical_MVT_Omega} with well-known results regarding the unique ergodicity of certain substitution systems to obtain the following result. 

\begin{Corollary}
\label{cor_odious_evil_Omega}
If $m$ and $q-1$ are coprime then for all $r\in\{0,1,\ldots,m-1\}$ the set of $n$ for which $s_q(\Omega(n))\equiv {r}\bmod{m}$ has asymptotic density $1/m$.
\end{Corollary}

In terminology introduced in \cite{BCG01}, an \define{odious} number is a non-negative integer with an odd number of $1$s in its binary expansion, whereas an \define{evil} number has an even number of $1$s. The special case of \cref{cor_odious_evil_Omega} where $m=q=2$ asserts that, asymptotically, $\Omega(n)$ is as often an odious number as it is an evil number.

Our next application of \cref{thm_dynamical_MVT_Omega} connects to the classical \define{\Mobius{} function} $\mob\colon\N\to\{-1,0,1\}$, which is defined as $\mob(n)=\lio(n)$ if $n$ is squarefree and $\mob(n)=0$ otherwise. 
In analogy to \eqref{eqn_PNT_lio}, one has
\begin{equation}
\label{eqn_PNT_mob}
\lim_{N\to\infty}\frac{1}{N}\sum_{n=1}^N \mob(n)~=~\lim_{N\to\infty}\frac{1}{N}\hspace{-.6 em}\sum_{1\leq n\leq N\atop n\, \mathrm{squarefree}}\hspace{-.6 em}\lio(n)~=~0,
\end{equation}
which is yet another well-known equivalent form\footnote{\label{ftn_PNT_equivalence}Equation \eqref{eqn_PNT_mob} was first observed by von Mangoldt \cite[pp.\ 849--851]{vonMangoldt97}; its equivalence to the Prime Number Theorem goes back to Landau \cite{Landau11,Landau12}, but can also be found in most textbooks on number theory (see for example \cite[\S 4.4]{TM00} or \cite[\S 4.9]{Apostol76}); the equivalence between \eqref{eqn_PNT_lio} and \eqref{eqn_PNT_mob} can be proved using the two basic formulas
$
\frac{1}{N}\sum_{n=1}^N\mob(n)= \sum_{d=1}^N \frac{\mob(d)}{d^2}( \frac{1}{N/d^2}\sum_{n\leq N/d^2}\lio(n))
$
and
$
\frac{1}{N}\sum_{n=1}^N \lio(n)=\sum_{d=1}^N \frac{1}{d^2}(\frac{1}{N/d^2}\sum_{n\leq N/d^2}\mob(n))
$; see \cite[pp.\ 631--632]{Landau09b}, \cite[p.\ 582]{Diamond82} for the former and \cite[pp.\ 620--621]{Landau09b}, \cite[Eq.\ (16)]{Axer10} for the latter.} of the Prime Number Theorem.
The next corollary provides a dynamical generalization of \eqref{eqn_PNT_mob}. 

\begin{Corollary}
\label{cor_dynamical_MVT_squarefree}
Let $(X,\mu,T)$ be uniquely ergodic.
Then
\begin{equation}
\label{eqn_ud_of_squarefree}
\lim_{N\to\infty}\frac{1}{N}\hspace{-.6 em}\sum_{1\leq n\leq N\atop n\, \mathrm{squarefree}} \hspace{-.6 em} f\big(T^{\Omega(n)}x\big)
~=~\frac{6}{\pi^2} \left(\int f\d\mu\right).
\end{equation}
for every $x\in X$ and $f\in\Cont(X)$.
\end{Corollary}

Note that \cref{cor_dynamical_MVT_squarefree} applied to rotation on two points yields \eqref{eqn_PNT_mob} in analogy to the way that \cref{thm_dynamical_MVT_Omega} applied to rotation on two points yielded  \eqref{eqn_PNT_lio}.
One can also derive from \cref{thm_dynamical_MVT_Omega} a generalization of \eqref{eqn_ud_of_squarefree} asserting that
\begin{equation}
\label{eqn_ud_of_k-free}
\lim_{N\to\infty}\frac{1}{N}\hspace{-.6 em}\sum_{1\leq n\leq N\atop n\,\mathrm{is}\, k\text{-free}} \hspace{-.6 em} f\big(T^{\Omega(n)}x\big)
~=~\frac{1}{\zeta(k)} \left(\int f\d\mu\right),
\end{equation}
where $\zeta$ is \define{Riemann's zeta function} and a \define{$k$-free integer} is an integer that is not divisible by a $k$-th power. 
This is a special case of a more general result presented in the next subsection, see \cref{cor_ortho_Omega_Besicovitch_ap_fctns}.

Another classical number-theoretic function, akin to $\Omega(n)$, is  $\omega\colon\N\to\N\cup\{0\}$, which counts the number of prime factors without multiplicity. 
Results concerning $\Omega(n)$ often possess a counterpart where $\Omega(n)$ is replaced by $\omega(n)$. The next result shows that \cref{thm_dynamical_MVT_Omega} is no exception to this rule.

\begin{Corollary}
\label{cor_dynamical_MVT_omega}
Let $(X,\mu,T)$ be uniquely ergodic.
Then
\begin{equation}
\label{eqn_ud_of_omega}
\lim_{N\to\infty}\frac{1}{N}\sum_{n=1}^N f\big(T^{\omega(n)}x\big)
~=~\int f\d\mu
\end{equation}
for every $f\in\Cont(X)$ and every $x\in X$.
\end{Corollary}

\begin{Remark}
Variants of Corollaries \ref{cor_polynomoial_ud_Omega}, \ref{cor_gneralized_polynomoial_ud_Omega}, and \ref{cor_odious_evil_Omega} where $\Omega(n)$ is replaced by $\omega(n)$ also hold and can be derived from \cref{cor_dynamical_MVT_omega} similarly to the way Corollaries \ref{cor_polynomoial_ud_Omega}, \ref{cor_gneralized_polynomoial_ud_Omega}, and \ref{cor_odious_evil_Omega} are derived from \cref{thm_dynamical_MVT_Omega}.
\end{Remark}

Our next main result,  \cref{thm_dynamical_MVT_fg_sue} below, shows that \cref{thm_dynamical_MVT_Omega} is merely a special case of a more general dynamical phenomenon involving actions of the multiplicative semigroup $(\N,\cdot)$.
To illustrate the connection between \cref{thm_dynamical_MVT_Omega} and actions of $(\N,\cdot)$, let us take a closer look at the expression $T^{\Omega(n)}$.
The function $\Omega\colon\N\to\N\cup\{0\}$ is \define{completely additive}, meaning that $\Omega(n_1n_2)=\Omega(n_1)+\Omega(n_2)$ for all $n_1,n_2\in\N$.
For that reason, $\Omega$ turns any action of $(\N,+)$ into an action of $(\N,\cdot)$:
\begin{equation}
\label{eqn_multipliactive_action}
T^{\Omega(nm)}=T^{\Omega(n)+\Omega(m)}=T^{\Omega(n)}\circ T^{\Omega(m)},\qquad\forall n,m\in\N.
\end{equation}
Thus, given an additive topological dynamical system $(X,T)$, we can view $X$ equipped with the multiplicative action induced by $(T^{\Omega(n)})_{n\in\N}$ as a new ``multiplciative'' dynamical system.
This motivates the following definition.

\begin{Definition}
A \define{multiplicative topological dynamical system} is a pair $(Y,S)$ where $Y$ is a compact metric space and $S=(S_n)_{n\in\N}$ is an action of $(\N,\cdot)$ by continuous maps on $Y$ (i.e.\ $ S_{nm}=S_n\circ S_m$ for all $m,n\in\N$).
\end{Definition}

\begin{Remark}
\label{ex_2}
The following are two natural approaches for constructing examples of multiplicative topological dynamical systems.
\begin{enumerate}	
[label=(\roman{enumi}),ref=(\roman{enumi}),leftmargin=*]

\item
\label{remark_itm_1}
The first approach utilizes the additivity of $\Omega(n)$ to turn actions of $(\N,+)$ into actions of $(\N,\cdot)$ as mentioned above. 
Indeed, it follows from \eqref{eqn_multipliactive_action} that 
for any additive topological dynamical system $(X,T)$ the pair $(X, T^\Omega)$ is a multiplicative topological dynamical system, where we use $T^\Omega$ to denote $(T^{\Omega(n)})_{n\in\N}$.
Since many dynamical properties of $(X,T)$ are inherited\footnote{It is straightforward to check that if $(X,T)$ is an \define{equicontinuous} additive topological dynamical system then $(X,T^\Omega)$ is an \define{equicontinuous} multiplicative topological dynamical system. Analogous statements hold when `equicontinuous' is replaced with either \define{minimal}, \define{transitive}, \define{distal}, \define{topologically weak mixing}, or \define{uniquely ergodic} (see \cite[pp.\ 14, 18, and 23]{Glasner03} for definitions). We remark that entropy is not preserved via this construction. More precisely, $(X,T^\Omega)$ has zero topological entropy regardless of the topological entropy of $(X,T)$, see \cref{prop_fg_implies_zero_entropy} below.} by $(X,T^\Omega)$, this construction yields a diverse class of systems with a wide range of different behaviors.
For the special case when $(X,T)$ is a rotation on two points, i.e.\ $X=\{0,1\}$ and $T(x)={x+1}\pmod{2}$, we call the corresponding multiplicative system $(X,T^\Omega)$  \define{multiplicative rotation on two points}. Although not mentioned explicitly, this system played a central role in our derivation of the Prime Number Theorem from \cref{thm_dynamical_MVT_Omega}.

\item
\label{remark_itm_2}
Another way of constructing examples of multiplicative topological dynamical systems is with the help of completely multiplicative functions. A function $b\colon\N\to \C$ is called \define{multiplicative} if $b(n_1n_2)=b(n_1)b(n_2)$ for all coprime $n_1,n_2\in\N$, and it is called \define{completely multiplicative} if $b(n_1n_2)=b(n_1)b(n_2)$ for all $n_1,n_2\in\N$.
Any completely multiplicative function $b$ taking values in the unit circle $ S^1\coloneqq \{z\in\C: |z|=1\}$ induces a natural action $S=(S_n)_{n\in\N}$ of $(\N,\cdot)$ on $ S^1$ via $S_n(z)=b(n)z$ for all $n\in\N$ and $z\in  S^1$. Let $Y$ denote the closure of the image of $b$ in $S^1$. We call the multiplicative topological dynamical system $(Y, S)$ \define{multiplicative rotation by $b$}. We remark that multiplicative rotation by the Liouville function $\lio$ and multiplicative rotation on two points (as defined in the previous paragraph) are isomorphic as topological dynamical systems. 
\end{enumerate}
\end{Remark}

A version of the Bogolyubov-Krylov theorem for actions of $(\N,\cdot)$ implies that every multiplicative topological dynamical system $(Y,S)$ possesses an $S$-invariant Borel probability measure $\nu$. If this measure is unique then $(Y,S)$ is called \define{uniquely ergodic}.
Motivated by \cref{thm_dynamical_MVT_Omega}, it is tempting to conjecture that for any uniquely ergodic multiplicative topological dynamical system $(Y,\nu,S)$ one has
\begin{equation}
\label{eqn_ud_of_mtds}
\lim_{N\to\infty}~\frac{1}{N}\sum_{n=1}^N g(S_n y)~=~\int g\d\nu
\end{equation}
for all $y\in Y$ and all $g\in\Cont(Y)$.
In general, this is false (see\ \cref{example_1} below). 
However, we will show that \eqref{eqn_ud_of_mtds} holds for a large class of multiplicative topological dynamical systems which contains systems of the form $(X,T^\Omega)$ as a rather special subclass.

\begin{Definition}
\label{def_fg_sue}
Let $(Y,S)$ be a multiplicative topological dynamical system. 
\begin{itemize}
\item
The action $S$ on $Y$ is called \define{finitely generated} if there exists a finite collection of continuous maps $R_1,\ldots,R_d\colon Y\to Y$ such that $S_p\in\{R_1,\ldots,R_d\}$ for all primes $p\in \P\coloneqq\{2,3,5,7,11,\ldots\}$. In this case we call $R_1,\ldots,R_d$ the \define{generators} of $S$.
\item
Let $\nu$ be a Borel probability measure on $Y$. 
Appropriating language from \cite{GSdraft}, we say that $\nu$ \define{pretends to be invariant under $S$} if there exists a set of primes $P\subset \P$ with $\sum_{p\notin P} 1/p<\infty$ such that $\nu$ is invariant under $S_p$ for all $p\in P$.
We call $(Y,S)$ \define{strongly uniquely ergodic} if there is only one Borel probability measure on $Y$ that pretends to be invariant under $S$. 
\end{itemize}
\end{Definition}

Note that strong unique ergodicity implies unique ergodicity, but not the other way around (see \cref{example_1}).

\begin{Maintheorem}
\label{thm_dynamical_MVT_fg_sue}
Let $(Y,\nu,S)$ be finitely generated and strongly uniquely ergodic. Then
\begin{equation*}
\lim_{N\to\infty}~\frac{1}{N}\sum_{n=1}^N g(S_n y)~=~\int g\d\nu
\end{equation*}
for every $y\in Y$ and $g\in\Cont(Y)$.
\end{Maintheorem}

Observe that for any uniquely ergodic additive topological dynamical system $(X,T)$ the corresponding multiplicative topological dynamical system $(X, T^\Omega)$ is both finitely generated and strongly uniquely ergodic. Therefore, \cref{thm_dynamical_MVT_fg_sue} contains \cref{thm_dynamical_MVT_Omega} as a special case.

\begin{Remark}
\label{rem_proof_B_gen}
If the generators $R_1,\ldots,R_d$ of the multiplicative action $S=(S_n)_{n\in\N}$ are powers of the same transformation, as for example is the case in \cref{cor_dynamical_MVT_additive_functions} below, then \cref{thm_dynamical_MVT_fg_sue} can also be proved by purely number theoretic methods using \Erdos{}-type estimates obtained by \Halasz{} in \cite{Halasz71}. This is in analogy the alternative proof of \cref{thm_dynamical_MVT_Omega} outlined in \cref{rem_proof_A_gen} above. However, in the general case when there is no apparent relation between the generators $R_1,\ldots,R_d$, we do not see how existing results from number theory can be used to give a proof of \cref{thm_dynamical_MVT_fg_sue} different to ours.
\end{Remark}

\begin{Remark}
It is natural to ask whether \cref{thm_dynamical_MVT_fg_sue} can be extended to 
multiplicative systems that are not finitely generated. In general, the answer is no. For example, the non-finitely generated multiplicative system $(Y,\nu,S)$ where $Y=\T=\R/\Z$, $\nu$ is the Haar measure on $\T$, and $S_n(x)={x+\log(n)}\pmod{1}$, does not satisfy the conclusion of \cref{thm_dynamical_MVT_fg_sue}, because $(\log(n))_{n\in\N}$ is not uniformly distributed mod~$1$ (see \cite[p.~8, Example 2.4]{KN74}). On the other hand, it is not hard to see that there are some non-finitely generated systems that do satisfy the conclusion of \cref{thm_dynamical_MVT_fg_sue}. For example, if $b\colon\N\to S^1$ is a non-finitely generated completely multiplicative function such that $(b(n))_{n\in\N}$ is uniformly distributed in the unit circle $S^1$, then the corresponding multiplicative rotation by $b$ (as described in \cref{ex_2}, part~\ref{remark_itm_2}) is a non-finitely generated multiplicative system with the desired property. A concrete example of such a function is the one uniquely determined by
\[
b(p)=
\begin{cases}
e(\alpha), &\text{if}~p\equiv 1\mod 4,
\\
e(1/p), &\text{otherwise},
\end{cases}
\]
where $p$ ranges over all prime numbers and $\alpha$ is some fixed irrational number.
At the end of this section we pose two questions (see Questions~\ref{question_1} and~\ref{question_2}) which further address this subject matter.   
\end{Remark}

We will discuss now some applications of \cref{thm_dynamical_MVT_fg_sue}. 
We call a multiplicative function $b\colon\N\to\C$ \define{finitely generated} if the set $\{b(p): p\in\P\}$ is finite.
Also, we say a multiplicative function $b$ has a \define{mean value} if its \Cesaro{} averages converge, i.e.,
$$
M(b)\coloneqq \lim_{N\to\infty}\,\frac{1}{N}\sum_{n=1}^N b(n) ~~\text{exists.}
$$
A renowned result of Wirsing \cite{Wirsing61, Wirsing67}, which confirmed a longstanding conjecture (see\ \cite[Problem 9 on pp.\ 293--294]{Erdos57}), states that any multiplicative function taking only the values $-1$ and $1$ has a mean value (cf.\ also \cite{Hildebrand86c}).
This result is often referred to as Wirsing's mean value theorem. 
A natural extension thereof, which follows from the more general mean value theorem of \Halasz{} \cite{Halasz68}, asserts that actually any bounded and finitely generated multiplicative function has a mean value.
This result can be recovered from \cref{thm_dynamical_MVT_fg_sue} by considering multiplicative rotations $z\mapsto b(n)z$ defined in \cref{ex_2}, part \ref{remark_itm_2}. Although the derivation is not too complicated, we defer the details to \cref{sec_applications_of_thm_B} (see \cref{prop_fg_mult_fctn}).

Another application of \cref{thm_dynamical_MVT_fg_sue} is the following generalization of \cref{thm_dynamical_MVT_Omega} along arithmetic progressions:

\begin{Corollary}
\label{cor_ortho_Omega_rational_phases}
Let $(X,\mu,T)$ be uniquely ergodic.
Then
\begin{equation}
\lim_{N\to\infty}\frac{1}{N}\sum_{n=1}^N f\big(T^{\Omega(mn+r)}x\big)
~=~\int f\d\mu
\end{equation}
for every $x\in X$, $f\in\Cont(X)$, $m\in\N$, and $r\in\{0,1,\ldots,m-1\}$.
\end{Corollary}

Note that \cref{cor_ortho_Omega_rational_phases} contains the Prime Number Theorem along arithmetic progressions as a special case. Indeed, by applying \cref{cor_ortho_Omega_rational_phases} to rotation on two points we obtain
\begin{equation}
\label{eqn_DPNT_lio}
\lim_{N\to\infty}~\frac{1}{N}\sum_{n=1}^N \lio(mn+r)\,=\,0,\qquad\forall m\in\N,~\forall  r\in\{0,1,\ldots,m-1\},
\end{equation}
which is a well-known equivalent form of the Prime Number Theorem along arithmetic progressions
(see \cite{Shapiro49}). 

\begin{Remark}
While \cref{cor_ortho_Omega_rational_phases} implies \eqref{eqn_DPNT_lio}, our proof is not entirely self-contained because it uses the fact that
\begin{equation}
\label{eqn_dirichlet_weak}
\sum_{ p\equiv{r}\bmod{m}}\frac{1}{p}\,=\,\infty
\end{equation}
for all $m,r\in\N$ with $\gcd(m,r)=1$, which is a variant of Dirichlet's celebrated theorem on primes in arithmetic progression.
It would be interesting to see whether \cref{cor_ortho_Omega_rational_phases} can be proved without relying on \eqref{eqn_dirichlet_weak}. 
\end{Remark}

In view of \cref{thm_dynamical_MVT_Omega} and \cref{cor_dynamical_MVT_omega}, it is natural to ask whether there are other number-theoretic functions, besides $\Omega(n)$ and $\omega(n)$, along which one can formulate and prove an ergodic theorem.
The following corollary of \cref{thm_dynamical_MVT_fg_sue} shows that one can replace $\Omega(n)$ in \eqref{eqn_ud_of_Omega_1} with a wider range of additive functions.

\begin{Corollary}
\label{cor_dynamical_MVT_additive_functions}
Let $a\colon\N\to\N\cup\{0\}$ be an additive function (i.e.\ $a(nm)=a(n)+a(m)$ for all coprime $n,m\in\N$) and assume the set $\{a(p): p~\text{prime}\}$ is finite.
Let $P_0\coloneqq \{p\in\P: a(p)\neq 0\}$ and suppose $\sum_{p\in P_0}1/p=\infty$. If $(X,\mu,T^{a(p)})$ is uniquely ergodic for all $p\in P_0$ then
\begin{equation}
\label{eqn_ud_additive_functions_2}
\lim_{N\to\infty}\frac{1}{N}\sum_{n=1}^N f\big(T^{a(n)}x\big)
~=~\int f\d\mu
\end{equation}
for every $f\in\Cont(X)$ and every $x\in X$.
\end{Corollary}

Among other things, \cref{cor_dynamical_MVT_additive_functions} implies yet another classical result of Wirsing. Let $\mathcal{Q}\subset \P$ be a set of primes with relative density $\tau>0$, i.e.\
$
|\mathcal{Q}\cap[1,n]| \sim {\tau n}/{\log(n)}.
$
Let $\Omega_\mathcal{Q}(n)$ denote the number of prime factors of $n$ that belong to $\mathcal{Q}$ (counted with multiplicities), and consider the following variant of the Liouville function:
\[
\lio_{\mathcal{Q}}(n)\, =\, (-1)^{\Omega_{\mathcal{Q}}(n)},\qquad\forall n\in\N.
\]
It was shown in \cite[\S\,2.4.2]{Wirsing61} that
\begin{equation}
\label{eqn_wirsings_lio_PNT}
\lim_{N\to\infty}\frac{1}{N}\sum_{n=1}^N \lio_{\mathcal{Q}}(n) \,=\, 0.
\end{equation}
It follows from \cite{Halasz68} that the condition $|\mathcal{Q}\cap[1,n]| \sim {\tau n}/{\log(n)}$ for some $\tau>0$ that appears in Wirsing's work can actually be weakened to $\sum_{p\in\mathcal{Q}}1/p=\infty$. Equation \eqref{eqn_wirsings_lio_PNT}, under this weaker condition, can rather easily be derived from \cref{cor_dynamical_MVT_additive_functions}. 
Indeed, if $a(n)=\Omega_{\mathcal{Q}}(n)$, $(X,T)$ is rotation on two points, $x=0$, and $f\colon \{0,1\}\to \C$ is the function $f(0)=1$ and $f(1)=-1$, then \eqref{eqn_wirsings_lio_PNT} follows immediately from \eqref{eqn_ud_additive_functions_2}.
It is also straightforward to derive from \cref{cor_dynamical_MVT_additive_functions} analogues of Corollaries \ref{cor_polynomoial_ud_Omega}, \ref{cor_gneralized_polynomoial_ud_Omega}, \ref{cor_odious_evil_Omega}, and \ref{cor_dynamical_MVT_squarefree} with $\Omega(n)$ replaced by either $\Omega_{\mathcal{Q}}(n)$ or $\omega_{\mathcal{Q}}(n)$, where $\omega_{\mathcal{Q}}(n)$ is the number of prime factors of $n$ belonging to $\mathcal{Q}$ counted without multiplicity.

\cref{thm_dynamical_MVT_fg_sue} also implies a variant of \cref{cor_dynamical_MVT_additive_functions} for several commuting transformations.

\begin{Corollary}
\label{cor_dynamical_MVT_several_additive_functions}
Let $a_1,\ldots,a_k\colon\N\to\N\cup\{0\}$ be completely additive functions such that
\begin{enumerate}	
[label=(\roman{enumi}),ref=(\roman{enumi}),leftmargin=*]
\item
\label{itm_several_add_fctns_i}
$\{a_i(p): p~\text{prime}\}$ is finite for each $i\in\{1,\ldots,k\}$;
\item
\label{itm_several_add_fctns_ii}
for any subgroup $\Gamma\subset \Z^k$ with $\mathrm{rank}(\Gamma)< k$ the set $$P_{\Gamma}=\{p\in\P: (a_1(p),\ldots,a_k(p))\notin\Gamma\}$$ satisfies $\sum_{p\in P_\Gamma} 1/p=\infty$.
\end{enumerate}
Then for any totally uniquely ergodic\footnote{Let $T_1,\ldots,T_k$ be commuting continuous maps on a compact metric space $X$. We call $(X,\mu,T_1,\ldots,T_k)$ totally uniquely ergodic if for all $m_1,\ldots,m_k\in\N$ the measure $\mu$ is the only Borel probability measure on $X$ that is invariant under $T_{i}^{m_i}$ for all $i=1,\ldots,k$.} $(X,\mu,T_1,\ldots,T_k)$ one has
\begin{equation}
\lim_{N\to\infty}\frac{1}{N}\sum_{n=1}^N f\big(T_1^{a_1(n)}\cdots T_k^{a_k(n)} x\big)
~=~\int f\d\mu
\end{equation}
for every $f\in\Cont(X)$ and $x\in X$.
\end{Corollary}

\begin{Example}
Let $\mathcal{Q}_1=\{p\in\P: p\equiv 1\mod 4\}$, $\mathcal{Q}_2=\{p\in\P: p\equiv 3\mod 4\}$, $a_1(n)=\Omega_{\mathcal{Q}_1}(n)$, and $a_2(n)=\Omega_{\mathcal{Q}_2}(n)$. Let $X=\T^3$, $\alpha\in\R\setminus\Q$, $T_1(x,y,z)=(x+\alpha, y,z+y)$, and $T_2(x,y,z)=(x,y+\alpha,z+x)$.
It is not hard to check that the system $(X,T_1,T_2)$ is totally uniquely ergodic with respect to the normalized Haar measure on $\T^3$. Also, note that $T_1^nT_2^m(x,y,z)=(x+n\alpha, y+m\alpha,z+mx+ny+nm\alpha)$ for all $n,m\in\N$.
Applying \cref{cor_dynamical_MVT_several_additive_functions} for $x=(0,0,0)$ and $f_h(x,y,z)=e(hz)$, $h\in\Z\setminus\{0\}$, we thus obtain
\begin{align}
\label{eqn_quad_Omega_dist_1}
\lim_{N\to\infty}\frac{1}{N}\sum_{n=1}^N f_h\big(T_1^{a_1(n)}T_2^{a_2(n)} x\big) \,=\, \lim_{N\to\infty}\frac{1}{N}\sum_{n=1}^N e(h\Omega_{\mathcal{Q}_1}(n)\Omega_{\mathcal{Q}_2}(n)\alpha)=0.
\end{align}
In view of Weyl's equidistribution criterion, this proves that for any irrational $\alpha$ the sequence $\Omega_{\mathcal{Q}_1}(n)\Omega_{\mathcal{Q}_2}(n)\alpha$, $n\in\N$, is uniformly distributed mod~$1$.
Similarly, one can show that for any irrational $\alpha$ and any disjoint $\mathcal{Q}_1,\ldots,\mathcal{Q}_k\subset\P$ with $\sum_{p\in\mathcal{Q}_i}1/p=\infty$ the sequence $\Omega_{\mathcal{Q}_1}(n)\cdot\ldots\cdot \Omega_{\mathcal{Q}_k}(n)\alpha$, $n\in\N$, is uniformly distributed mod~$1$. We leave the details to the interested reader.
\end{Example}

\subsection{A wider framework for Sarnak's conjecture}
\label{sec_beyond_sarnak}

A rule of thumb in number theory says that the additive and multiplicative structures in the integers behave ``independently'', where the respective notion of independence may vary based on the context.
In this subsection, we explore this phenomenon from a dynamical perspective through an analysis of the interplay between additive and multiplicative semigroup actions.  
To begin with, let us introduce a way of capturing the absence of correlation between arithmetic functions. 

\begin{Definition}
\label{def_indepndence_of_sequences}
We call two bounded arithmetic functions $a,b\colon\N\to\C$ \define{asymptotically uncorrelated} if
\begin{equation}
\label{eq_asymp_indep}
\lim_{N\to\infty}\left(\frac{1}{N}\sum_{n=1}^N a(n)\overline{b(n)}-
\left(\frac{1}{N}\sum_{n=1}^N a(n)\right)
\left(\frac{1}{N}\sum_{n=1}^N \overline{b(n)}\right)\right)=0.
\end{equation}
\end{Definition}

We are interested in instances of \eqref{eq_asymp_indep} where $a(n)$ arises from an additive system and $b(n)$ arises from a multiplicative system. 
\begin{Definition}
\label{def_disjoitness}
Let $(X,T)$ be an additive topological dynamical system and $(Y,S)$ a multiplicative topological dynamical system. We call $(X,T)$ and $(Y,S)$ \define{disjoint}\footnote{Using \define{disjoint} in this context is in line with the use of this term in connection to Sarnark's \Mobius{} disjointness conjecture. It should not be confused with Furstenberg's notion of disjointness introduced in \cite{Furstenberg67}.} if for all $x\in X$, $f\in \Cont(x)$, $y\in Y$, and $g\in\Cont(Y)$ the sequences $a(n)=f(T^n x)$ and $b(n)=g(S_ny)$ are asymptotically uncorrelated.
\end{Definition}

Numerous classical results and conjectures in number theory can be interpreted as the disjointness between certain classes of additive and multiplicative systems. 
For example, a well-known result by Davenport \cite{Davenport37} asserts that
\begin{equation}
\label{eqn_davenport_lio}
\lim_{N\to\infty}\frac{1}{N}\sum_{n=1}^N  e(n\alpha)\, \lio(n)\, =\, 0,\qquad\forall\alpha\in \R,
\end{equation}
where $\lio$ is the Liouville function.
Note that \eqref{eqn_davenport_lio} is equivalent to the assertion that $e(n\alpha)$ and $\lio(n)$ are asymptotically uncorrelated.
This allows us to recast \eqref{eqn_davenport_lio} as a dynamical statement as follows:
\begin{named}{Davenport's Theorem Reformulated}{}
\label{davenport_reformulated}
Multiplicative rotation on two points (defined in \cref{ex_2}, part \ref{remark_itm_1}) is disjoint from all additive rotations $x\mapsto {x+\alpha}~(\text{mod}\, {1})$ on the torus $\T$.
\end{named}

To see why the above statement implies \eqref{eqn_davenport_lio}, it suffices to observe that $e(n\alpha)$ is of the form $f(T^nx)$, where $T$ is rotation by $\alpha$, and the Liouville function $\lio(n)$ is of the form $g(S_n y)$, where $S=(S_n)_{n\in\N}$ is multiplicative rotation on two points.
For the reverse implication, notice that any sequence $a(n)=f(T^nx)$ arising from rotation by $\alpha$ can be approximated uniformly by finite linear combinations of sequences of the form $e(hn\alpha)$, $h\in\Z$, 
whereas any sequence $b(n)= g(S_n y)$ arising from multiplicative rotation on two points is equal to $c_1\lio(n)+c_2$ for some $c_1,c_2\in\C$. Therefore, the desired conclusion that $a(n)=f(T^nx)$ and $b(n)=g(S_n y)$ are  asymptotically uncorrelated follows from \eqref{eqn_davenport_lio}.

Davenport's result is complemented by a theorem of Daboussi (see \cite{Daboussi75,DD74, DD82}), which asserts that for all $\alpha\in\R\setminus\Q$ and all completely multiplicative\footnote{Daboussi's result actually holds for all bounded multiplicative functions and not just for completely multiplicative functions taking values in the unit circle $S^1$. In fact, using standard tools from number theory one can show that these two assertions are equivalent (cf.\ \cref{lem_mult_fctns_transferrence} below).} functions $b\colon\N\to S^1$ one has
\begin{equation}
\label{eqn_daboussi}
\lim_{N\to\infty}\frac{1}{N}\sum_{n=1}^N  e(n\alpha)\, b(n)\, =\, 0.
\end{equation}
Similarly to Davenport's theorem, Daboussi's result can be reformulated as the disjointness between certain classes of additive and multiplicative systems.
\begin{named}{Daboussi's Theorem Reformulated}{}
\label{daboussi_reformulated}
Let $\alpha\in\R\setminus\Q$ and assume $b\colon\N\to S^1$ is completely multiplicative. Then the multiplicative rotation $z\mapsto b(n)z$ on the unit circle $ S^1$ (defined in \cref{ex_2}, part \ref{remark_itm_2}) is disjoint from the additive rotation $x\mapsto {x+\alpha}~(\text{mod}\, {1})$ on the torus $\T$.
\end{named}

It is perhaps instructive to point out that \eqref{eqn_davenport_lio} holds for all $\alpha$, whereas \eqref{eqn_daboussi} holds only for irrational $\alpha$. 
This is because for any rational $\alpha$ there exists a \define{Dirichlet character}\footnote{\label{ftn_dc}An arithmetic function $\chi\colon\N\to\C$ is called a \define{Dirichlet character} if  there exists a number $d\in\N$, called the \define{modulus} of $\chi$, such that
\begin{itemize}
\item
$\chi(n+d)=\chi(n)$ for all $n\in\N$;
\item
$\chi(n)=0$ whenever $\gcd(n,d)\neq 1$, and $\chi(n)$ is a $\tot(d)$-th root of unity if $\gcd(n,d)= 1$, where $\tot$ denotes Euler's totient function.
\item
$\chi(nm)=\chi(n)\chi(m)$ for all $n,m\in\N$.
\end{itemize}
A Dirichlet character $\chi$ is \define{principal} if $\chi(n)=1$ for all $n\in\N$ with $\gcd(n,d)=1$.} $\chi$ such that $e(n\alpha)$ and $\chi$ are not asymptotically uncorrelated (if $\alpha=r/m$ for $r$ and $m$ coprime then taking $\chi$ to be any primitive Dirichlet character of modulus $m$ works, because in this case $\lim_{N\to\infty}\frac{1}{N}\sum_{n=1}^N e(n\alpha)\chi(n)$ equals the Gauss sum $\frac{1}{m}\sum_{j=0}^m e(jr/m)\chi(j)$ which is non-zero).
Phenomena of this type are often referred to as \define{local obstructions}.

A conjecture of Sarnak, which represents a far-reaching dynamical generalization of Davenport's Theorem, emphasizes even more strongly that additive systems are often disjoint from multiplicative systems. 
The formulation of \ref{conj_sarnak} involves the notion of \define{(topological) entropy}.
Entropy is a dynamical invariant which measures the complexity of a dynamical system. We give the precise definition in \cref{sec_entropy}. A function $a\colon\N\to\C$ is called \define{deterministic} if there exists a zero entropy additive topological dynamical system $(X,T)$, a point $x\in X$, and a function $f\in\Cont(X)$ such that $a(n)=f(T^n x)$ for all $n\in\N$.

\begin{named}{Sarnak's Conjecture}{\cite{Sarnak11,Sarnak12}\footnote{Sarnak originally formulated his conjecture using the \Mobius{} function $\mob$ instead of the Liouville function $\lio$. For the equivalence between these two formulations see \cite[Corollary 3.8]{FKL18}.}}
\label{conj_sarnak}
For any deterministic $a\colon\N\to\C$ one has
\begin{equation}
\label{eqn_sarnak_lio}
\lim_{N\to\infty}\frac{1}{N}\sum_{n=1}^N  a(n)\, \lio(n)~=~0.
\end{equation} 
\end{named}

We can reformulate \ref{conj_sarnak} as follows.

\begin{named}{Sarnak's Conjecture Reformulated}{}
\label{conj_sarnak_dynamical}
Multiplicative rotation on two points is disjoint from any zero entropy additive topological dynamical system. 
\end{named}

The following heuristic postulate is an attempt to put forward a principle that, on the one hand, encompasses the above reformulations of  Davenport's and Daboussi's theorems and \ref{conj_sarnak} and, on the other hand, serves as a guide for new developments.

\vspace{-0.8em}
\begin{equation}
  \tag{H}\label{eqn_heuristic}
  \parbox{\dimexpr\linewidth-6em}{%
    \strut
    \emph{
Let $(X,T)$ be a zero entropy additive topological dynamical system and $(Y,S)$ a ``low complexity'' multiplicative topological dynamical system. If there are ``no local obstructions'' then $(X,T)$ and $(Y,S)$ are disjoint.}
    \strut
  }
\end{equation}

The notion of ``low complexity'' which appears in the formulation of \eqref{eqn_heuristic} is admittedly (and somewhat intentionally) not well defined.
While for additive topological dynamical systems the notion of zero topological entropy is just a precise form of low complexity, the situation with multiplicative topological dynamical systems is drastically different due to the fact that $(\N,\cdot)$ has an infinite number of generators. Although it is certainly tempting to try to replace ``low complexity'' in \eqref{eqn_heuristic} with zero entropy, this does not work! For example, consider the $(\N,\cdot)$-action on the torus $\T$ given by $x\mapsto {nx}\pmod{1}$, $n\in\N$.
This action has zero topological entropy, but it violates \eqref{eqn_heuristic} because it can easily be used to generate deterministic sequences such as $e(n\alpha)$, $n\in\N$. 
This example indicates that the notion of low complexity for $(Y,S)$ in \eqref{eqn_heuristic} needs to be more restrictive than just zero entropy. One such possibility, which leads to interesting new developments including 
Theorems \ref{thm_ortho_fg_sue_nilsequences} and \ref{thm_ortho_fg_sue_horocycle} below, is to assume that $(Y,S)$ belongs to a subclass of zero entropy systems which we introduced in the previous subsection under the name \define{finitely generated} (see \cref{def_fg_sue}).
For the proof that finitely generated multiplicative systems have zero entropy see \cref{prop_fg_implies_zero_entropy}.
Yet another non-trivial example of low complexity is provided by actions $S=(S_n)_{n\in\N}$ of $(\N,\cdot)$ for which every generator $S_p$, $p\in\P$, has zero entropy.
Special cases of this option are implemented by our reformulations of Davenport's and Daboussi's theorems and Sarnak's conjecture above. (See also Questions \ref{question_1} and \ref{question_2} at the end of this section.)

As for the stipulation ``no local obstructions'' in \eqref{eqn_heuristic}, we believe it is captured by the notion of \define{aperiodicity} which we will presently introduce.
We call an arithmetic function $P\colon\N\to\C$ \define{periodic} if there exists $m\in\N$ such that $P(n+m)=P(n)$ holds for all $n\in\N$. 

\begin{Definition}
\label{def_apeiodic_systems}
\
\begin{itemize}
\item We call an additive topological dynamical system $(X,T)$ \define{aperiodic} if for any periodic $P\colon\N\to\C$, any $f\in\Cont(X)$, and any $x\in X$ the sequences $P(n)$ and $a(n)=f(T^nx)$ are asymptotically uncorrelated. Equivalently, for all $f\in\Cont(X)$ and $x\in X$, if $a_N(n)\coloneqq f(T^nx)-\frac{1}{N}\sum_{m=1}^N f(T^mx)$, then
\begin{equation}
\label{eqn_additive_aperiodicity}
\lim_{N\to\infty}\,\frac{1}{N}\sum_{n=1}^N e(n\alpha)\, a_N(n)~=~0,\qquad\forall\alpha\in\Q.
\end{equation}
\item Similarly, we call a multiplicative topological dynamical system $(Y,S)$ \define{aperiodic} if for any periodic $P\colon\N\to\C$, any $g\in\Cont(Y)$, and any $y\in Y$ the sequences $P(n)$ and $b(n)=g(S_n y)$ are asymptotically uncorrelated. Equivalently, for all $g\in\Cont(Y)$ and $y\in Y$, if $b_N(n)\coloneqq g(S_n y)-\frac{1}{N}\sum_{n=1}^N g(S_ny)$, then
\begin{equation}
\label{eqn_multiplicative_aperiodicity}
\lim_{N\to\infty}\,\frac{1}{N}\sum_{n=1}^N e(n\alpha)\,b_N(n)~=~0,\qquad\forall\alpha\in\Q.
\end{equation}
\end{itemize}
\end{Definition}

\begin{Remark}
\
\begin{enumerate}	
[label=(\roman{enumi}),ref=(\roman{enumi}),leftmargin=*]
\item
An additive topological dynamical system $(X,T)$ is aperiodic if and only if any \define{ergodic}\footnote{A $T$-invariant Borel probability measures $\mu$ on $(X,T)$ is \define{ergodic} if for any Borel measurable set $A\subset X$ with $\mu(A\triangle T^{-1}A)=0$ satisfies $\mu(A)\in\{0,1\}$.} $T$-invariant Borel probability measure on $(X,T)$ is \define{totally ergodic}\footnote{A $T$-invariant Borel probability measures $\mu$ on $(X,T)$ is \define{totally ergodic} if $(X,T^m)$ is ergodic for all $m\in\N$.}.\footnote{The fact that the aperiodicity of $(X,T)$ implies that all ergodic measures are totally ergodic follows rather quickly from the definition. For the proof of the other direction, assume $(X,T)$ is not aperiodic, i.e., there exist $x\in X$, $m\in\N$, $f\in C(X)$, and a periodic $P\colon\N\to \C$ such that $f(T^nx)$ and $P(n)$ are not asymptoically uncorrelated. This means there exists a sequence $N_1<N_2<\ldots\in\N$ such that
\begin{align}
\label{eqn_aper_erg_toterg}
\lim_{i\to\infty}\frac{1}{N_i}\sum_{n=1}^{N_i} f(T^nx) P(n) \neq \left(\lim_{i\to\infty}\frac{1}{N_i}\sum_{n=1}^{N_i}  f(T^nx)\right)\left(\lim_{i\to\infty}\frac{1}{N_i}\sum_{n=1}^{N_i} P(n)\right).
\end{align}
Let $m\in\N$ be the period of $P$ and let $\Z/m\Z=\{0,1,\ldots,m-1\}$ denote the cyclic group of order $m$.
Define
$
\nu_i\coloneqq \frac{1}{N_i}\sum_{n=1}^{N_i} \delta_{(T^nx,{n}\bmod{m})}
$
where $\delta_{(T^nx,{n}\bmod{m})}$ is the point-mass at $(T^nx,{n}\bmod{m})\in X\times \Z/m\Z$. By passing to a subsequence of $N_1,N_2,\ldots$, we can assume without loss of generality that $\nu_i$ converges in the weak-$^*$ topology to a measure $\nu$.
Let $\nu=\int \nu_w\d\rho(w)$ be the ergodic decomposition of $\nu$ so that $\nu_w$ is an ergodic Borel probability measure on $X\times \Z/m\Z$ for every $w$. Let $\mu_w$ be the projection of $\nu_w$ onto the first coordinate and observe that $\mu_w$ is ergodic.
By assumption, this means that $\mu_w$ is also totally ergodic.
Let $m_{\Z/m\Z}$ denote the normalized counting measure on $\Z/m\Z$, which coincides with the projection of $\nu_w$ onto the second coordinate. Since $\mu_w$ and $m_{\Z/m\Z}$ are the marginals of $\nu_w$, and since $\mu_w$ is totally ergodic and $m_{\Z/m\Z}$ is purely atomic, we must have $\nu_w=\mu_w\otimes m_{\Z/m\Z}$. It follows that $\nu=\mu\otimes m_{\Z/m\Z}$, where $\mu=\int \mu_w\d\rho(w)$. Hence
\begin{align*}
\lim_{i\to\infty}\frac{1}{N_i}\sum_{n=1}^{N_i} f(T^nx) P(n)
&=
\int f\otimes P \d\nu
\\
&=
\left(\int f \d\mu\right)
\left(\int P \d m_{\Z/m\Z}\right)
\\
&=
\left(\lim_{i\to\infty}\frac{1}{N_i}\sum_{n=1}^{N_i}  f(T^nx)\right)\left(\lim_{i\to\infty}\frac{1}{N_i}\sum_{n=1}^{N_i} P(n)\right),
\end{align*}
a contradiction to \eqref{eqn_aper_erg_toterg}.}
\item
It is straightforward to show that \eqref{eqn_multiplicative_aperiodicity} is equivalent to the assertion that for all non-principal Dirichlet characters $\chi$ one has $\lim_{N\to\infty}\frac{1}{N}\sum_{n=1}^N g(S_{n}y)\chi(n)=0$.
\item
It follows from \eqref{eqn_davenport_lio} that a multiplicative rotation on two points is an aperiodic multiplicative topological dynamical system. More generally, we show in \cref{lem_Omega_derived_systems_are_aperiodic} below that for \define{any} additive topological dynamical system $(X,T)$ the corresponding multiplicative topological dynamical system $(X,T^\Omega)$ is aperiodic. 
\end{enumerate}
\end{Remark}

Note that if a multiplicative topological dynamical system $(Y,S)$ is not aperiodic then there exists an addditive topological dynamical system $(X,T)$ (namely a cyclic rotation on finitely many points) such that $(X,T)$ and $(Y,S)$ are not disjoint.\footnote{If $(Y,S)$ is not aperiodic then for some $y\in Y$, $g\in C(Y)$, and $m\in\N$ there exists a periodic sequence $P(n)$ of period $m$ such that $b(n)=g(S_ny)$ and $P(n)$ are not asymptotically uncorrelated. It follows that rotation on $m$ points and $(Y,S)$ are not disjoint.} Conversely, if $(X,T)$ is not aperiodic then there exists $(Y,S)$ such that $(X,T)$ and $(Y,S)$ are not disjoint. In this sense, aperiodicity is a necessary condition for disjointness between additive and multiplicative systems.
The following conjecture, which is a generalization of \ref{conj_sarnak} and one of our main illustrations of heuristic \eqref{eqn_heuristic}, asserts that if $(Y,S)$ is finitely generated then aperiodicity is not just a necessary but also sufficient condition.

\begin{Conjecture}
\label{conj_AM}
Let $(X,T)$ be an additive topological dynamical system of zero entropy and let $(Y,S)$ be a finitely generated multiplicative topological dynamical system.
If either $(X,T)$ is aperiodic or $(Y,S)$ is aperiodic then $(X,T)$ and $(Y,S)$ are disjoint.
\end{Conjecture}

Observe that \ref{conj_sarnak} corresponds to the special case of \cref{conj_AM} where $(Y,S)$ is a multiplicative rotation on two points.

Assuming that $(X,T)$ is uniquely ergodic and $(Y,S)$ is strongly uniquely ergodic, we have the following aesthetically appealing variant of \cref{conj_AM}.

\begin{Conjecture}
\label{conj_AM_ue}
Let $(X,\mu,T)$ be a uniquely ergodic additive topological dynamical system of zero entropy and let $(Y,\nu,S)$ be a finitely generated and strongly uniquely ergodic multiplicative topological dynamical system.
If either $(X,T)$ is aperiodic or $(Y,S)$ is aperiodic then
\begin{equation}
\label{eqn_AM_ue_sue_1}
\lim_{N\to\infty}\,\frac{1}{N}\sum_{n=1}^N f(T^n x) g(S_n y)~=~\left(\int f\d\mu\right)
\left(\int g\d\nu\right)
\end{equation}
for all $x\in X$, $f\in \Cont(X)$, $y\in Y$, and $g\in\Cont(Y)$.
\end{Conjecture}
Note that \cref{thm_dynamical_MVT_fg_sue} corresponds to the special case of \cref{conj_AM_ue} where $(X,T)$ is the trivial system. Moreover, due to \cref{thm_dynamical_MVT_fg_sue}, we see that \cref{conj_AM} actually implies \cref{conj_AM_ue}.

An extension of \eqref{eqn_davenport_lio}, which constitutes a special case of \ref{conj_sarnak}, was established in \cite{GT12b} (see also \cite{FFKPL19}) and asserts that
\begin{equation}
\label{eqn_GT_lio}
\lim_{N\to\infty}\frac{1}{N}\sum_{n=1}^N f(T^nx)\, \lio(n)\,=\,0,
\end{equation}
for all nilsystems $(X,T)$, $x\in X$, and $f\in\Cont(X)$ (see \cref{ftnt_1} for the definition of a nilsystem). We have the following extension of \eqref{eqn_GT_lio}, which establishes a special case of \cref{conj_AM}.

\begin{Maintheorem}
\label{thm_ortho_fg_sue_nilsequences}
Let $(Y,S)$ be a finitely generated multiplicative topological dynamical system, and let $(X,T)$ be a nilsystem. If either $(X,T)$ is aperiodic or $(Y,S)$ is aperiodic then $(X,T)$ and $(Y,S)$ are disjoint.
\end{Maintheorem}

The following corollary of \cref{thm_ortho_fg_sue_nilsequences} generalizes \cref{cor_ortho_Omega_rational_phases} from the previous subsection.

\begin{Corollary}
\label{cor_ortho_Omega_linear_phases}
Let $(X,T)$ be a uniquely ergodic additive topological dynamical system and let $\mu$ denote the corresponding unique $T$-invariant Borel probability measure on $X$.
Then
\begin{equation*}
\lim_{N\to\infty}\frac{1}{N}\sum_{n=1}^N e(n\alpha)\, f\big(T^{\Omega(n)}x\big)
~=~
\begin{cases}
\int f\d\mu,&\text{if}~\alpha\in\Z
\\
0,&\text{if}~\alpha\in\R\setminus\Z
\end{cases}
\end{equation*}
for every $f\in\Cont(X)$ and every $x\in X$.
\end{Corollary}

Another result related to \cref{cor_ortho_Omega_linear_phases} is the following.

\begin{Corollary}
\label{cor_ortho_Omega_Besicovitch_ap_fctns}
Let $(X,T)$ be a uniquely ergodic additive topological dynamical system and let $\mu$ denote the corresponding unique $T$-invariant Borel probability measure on $X$. Let $a\colon\N\to\C$ be a Besicovitch almost periodic function\footnote{A bounded arithmetic function $a\colon\N\to\C$ is called \define{Besicovitch almost periodic} if for every $\epsilon>0$ there exists a trigonometric polynomial $P(n)\coloneqq c_1 e(n\alpha_1)+\ldots+c_L e(n\alpha_L)$, where $c_1,\ldots,c_L\in\C$ and $\alpha_1,\ldots,\alpha_L\in \R$, such that $\limsup_{N\to\infty}\frac{1}{N}\sum_{n=1}^N |a(n)-P(n)|\,\leq\,\epsilon$.} and denote by $M(a)\coloneqq\lim_{N\to\infty}\frac{1}{N}\sum_{n=1}^N a(n)$ its mean value.
Then
\begin{equation}
\label{eqn_ortho_Omega_Bes_ap}
\lim_{N\to\infty}\frac{1}{N}\sum_{n=1}^N a(n)\,f\big(T^{\Omega(n)}x\big)
~=~ M(a)\left(
\int f\d\mu\right)
\end{equation}
for every $f\in\Cont(X)$ and every $x\in X$.
\end{Corollary}

We remark that the indicator function of the squarfree numbers $\1_{\square\text{-free}}$ is Besicovitch almost periodic and its mean value equals $6/\pi^2$. Therefore, choosing $a(n)=\1_{\square\text{-free}}(n)$ in \eqref{eqn_ortho_Omega_Bes_ap} recovers \eqref{eqn_ud_of_squarefree}. More generally, for every $k\geq 2$ the indicator function for the set of $k$-free numbers is Besicovitch almost periodic with mean value $1/\zeta(k)$ and hence we can actually get \eqref{eqn_ud_of_k-free} from \eqref{eqn_ortho_Omega_Bes_ap}.

From \cref{thm_ortho_fg_sue_nilsequences} we can also derive a generalization of \cref{cor_polynomoial_ud_Omega}.

\begin{Corollary}
\label{cor_polynomoial_joint_ud}
Let $p(x)=c_k x^k + \ldots + c_1 x + c_0$ and $q(x)=d_\ell x^\ell+\ldots+d_1 x+d_0$ be two polynomials with real coefficients and suppose at least one of the coefficients $c_1,\ldots,c_k$ is irrational and at least one of the coefficients $d_1,\ldots,d_\ell$ is irrational. Then $\big(p(n),q(\Omega(n))\big)_{n\in\N}$ is uniformly distributed in the two-dimensional torus $\T^2$. 
\end{Corollary}

Any sequence of the form $f(T^nx)$, $n\in\N$, where $(X,T)$ is a nilsystem, $x\in X$, and $f\in\Cont(X)$ is called a \define{nilsequence}.
Nilsequences naturally generalize sequences of the form $e(Q(n))$, where $Q$ is a polynomial, and play an important role in additive combinatorics. Note that \eqref{eqn_GT_lio} implies that $\lio(n)$ and any nilsequence are asymptotically uncorrelated. Using \cref{thm_ortho_fg_sue_nilsequences} we can further generalize this result. 

\begin{Corollary}
\label{cor_ortho_fg_sue_nilsequences}
Let $(Y,\nu,S)$ be an aperiodic, finitely generated, and strongly uniquely ergodic multiplicative topological dynamical system. Let $\eta\colon\N\to\C$ be a nilsequence and denote by $M(\eta)\coloneqq\lim_{N\to\infty}\frac{1}{N}\sum_{n=1}^N \eta(n)$ its mean value.
Then
\begin{equation*}
\lim_{N\to\infty}\frac{1}{N}\sum_{n=1}^N \eta(n)\,g\big(S_n y\big)
~=~
M(\eta)\left(\int g\d\nu\right)
\end{equation*}
for every $y\in Y$, and $g\in\Cont(Y)$.
\end{Corollary}

In \cite{BSZ13} it was shown that \ref{conj_sarnak} holds for all \define{horocycle flows}\footnote{Let $G= SL_2(\R)$, let $\Gamma$ be a lattice in $G$ (i.e.\ a discrete subgroup of $G$ with co-finite volume) and let $u= \left[\begin{smallmatrix} 1&1\\ 0&1 \end{smallmatrix}\right]$. The additive topological dynamical system $(X,T)$ where $X=G/\Gamma$ and $T$ is given by $T(g\Gamma )=(ug)\Gamma$ is called a \define{horocycle flow}.}. We have the following generalization, which verifies yet another instance of \cref{conj_AM}.

\begin{Maintheorem}
\label{thm_ortho_fg_sue_horocycle}
Let $(Y,S)$ be a finitely generated and aperiodic multiplicative topological dynamical system, and let $(X,T)$ be a horocycle flow. Then $(X,T)$ and $(Y,S)$ are disjoint.
\end{Maintheorem}

Our next goal is to discuss analogues of \cref{conj_AM} for multiplicative topological dynamical systems that are not necessarily finitely generated. When dealing with non-finitely generated actions of $(\N,\cdot)$, new local obstructions can arise.
In particular, there exist non-finitely generated multiplicative systems that are aperiodic but not disjoint from all zero entropy additive systems (for example $S_nz=n^i$, $z\in S^1$).
To meet this challenge, we need a strengthening of the notion of aperiodicity introduced in \cref{def_apeiodic_systems}.
Let us call a bounded arithmetic function $P\colon\N\to\C$ \define{locally periodic} if there exists $m\in\N$ such that for all $\epsilon>0$ the set
$
\left\{n\in\N: |P(n+m)-P(n)|\geq \epsilon\right\}
$
has zero asymptotic density. Roughly speaking, this means that for all $H\in\N$ and for ``almost all'' $n\in\N$ the function $P$ looks like a periodic function in a window $[n-H,n+H]$ around $n$. Surely every periodic function is locally periodic. A natural class of arithmetic functions that are locally periodic but not periodic are functions of the form $\chi(n) n^{it}$, where $\chi$ is a Dirichlet character and $n^{it}$, $t\in\R$, is an \define{Archimedean character}.

\begin{Definition}
\label{def_loc_apeiodic_systems}
\
\begin{itemize}
\item We call an additive topological dynamical system $(X,T)$ \define{locally aperiodic} if for any locally periodic $P\colon\N\to\C$, any $f\in\Cont(X)$, and any $x\in X$ the sequences $P(n)$ and $a(n)=f(T^nx)$ are asymptotically uncorrelated.
Equivalently, for all $f\in\Cont(X)$ and $x\in X$, if $a_N(n)\coloneqq f(T^nx)-\frac{1}{N}\sum_{n=1}^N f(T^nx)$, then
\begin{equation}
\label{eqn_additive_local_aperiodicity}
\lim_{H\to\infty}\lim_{N\to\infty}\,\frac{1}{N}\sum_{n=1}^N \left|\frac{1}{H}\sum_{h=n}^{n+H} e(h\alpha)\, a_N(n)\right|\,=\,0, \qquad\forall\alpha\in\Q.
\end{equation}
\item We call a multiplicative topological dynamical system $(Y,S)$ \define{locally aperiodic} if for any locally periodic $P\colon\N\to\C$, any $g\in\Cont(Y)$, and any $y\in Y$ the sequences $P(n)$ and $b(n)=g(S_n y)$ are asymptotically uncorrelated. Equivalently, for all $g\in\Cont(Y)$ and $y\in Y$, if $b_N(n)\coloneqq g(S_n x)-\frac{1}{N}\sum_{n=1}^N g(S_nx)$, then
\begin{equation}
\label{eqn_additive_local_aperiodicity_2}
\lim_{H\to\infty}\lim_{N\to\infty}\,\frac{1}{N}\sum_{n=1}^N \left|\frac{1}{H}\sum_{h=n}^{n+H} e(h\alpha)\, b_N(n)\right|\,=\,0, \qquad\forall\alpha\in\Q.
\end{equation}
\end{itemize}
\end{Definition}

\begin{Remark}
It was shown in \cite{MRT15} that the Liouville function $\lio(n)$ satisfies 
$$
\lim_{H\to\infty}\lim_{N\to\infty}\,\frac{1}{N}\sum_{n=1}^N \left|\frac{1}{H}\sum_{h=n}^{n+H} e(h\alpha)\, \lio(n)\right|\,=\,0, \qquad\forall\alpha\in\R.
$$
This implies that multiplicative rotation on two points is locally aperiodic.
\end{Remark}

When considering analogues of \cref{conj_AM} for systems $(Y,S)$ that are not necessarily finitely generated, we propose to replace aperiodicity with local aperiodicity. 
This is in line with Matom{\"a}ki-Radziwi{\l}{\l}-Tao's ``corrected Elliott conjecture'' which emanated from their work in \cite{MRT15}.

We conclude this introduction with two questions.

\begin{Question}
\label{question_1}
Let $(X,T)$ be an additive topological dynamical system of zero entropy and let $(Y,S)$ be a \define{distal}\footnote{A multiplicative topological dynamical system $(Y,S)$ is \define{distal} if for all $y_1,y_2\in Y$ with $y_1\neq y_2$ we have $\inf_{n\in\N} d(S_ny_1, S_ny_2)>0$.} multiplicative topological dynamical system.
Is it true that if either $(X,T)$ is locally aperiodic or $(Y,S)$ is locally aperiodic then $(X,T)$ and $(Y,S)$ are disjoint?
\end{Question}

A harder question, which includes \cref{question_1} as a special case, is the following.

\begin{Question}
\label{question_2}
Let $(X,T)$ be an additive topological dynamical system of zero entropy and let $(Y,S)$ be a multiplicative topological dynamical system with the property that for every $p\in\P$ the map $S_p\colon Y\to Y$ has zero entropy.
Is it true that if either $(X,T)$ is locally aperiodic or $(Y,S)$ is locally aperiodic then $(X,T)$ and $(Y,S)$ are disjoint?
\end{Question}

Even the case of \cref{question_2} when the additive system $(X,T)$ is an irrational rotation seems to be an interesting open question.

\paragraph{\textbf{Structure of the paper}.}
In \cref{sec_distr_orbits_along_Omega} we present a proof of \cref{thm_dynamical_MVT_Omega}.
As was mentioned above, \cref{thm_dynamical_MVT_Omega} is a corollary of \cref{thm_dynamical_MVT_fg_sue}.
Since the proof of \cref{thm_dynamical_MVT_Omega} contains the essential ideas in embryonic form and is much shorter and less technical than the proof of \cref{thm_dynamical_MVT_fg_sue}, we believe that it is beneficial to the reader to see first the proof of \cref{thm_dynamical_MVT_Omega}.

\cref{sec_fg_sue_mtds} is dedicated to the proof \cref{thm_dynamical_MVT_fg_sue}.

In Sections~\ref{sec_appl_thmA} and~\ref{sec_applications_of_thm_B} we give the proofs of Corollaries \ref{cor_polynomoial_ud_Omega}, \ref{cor_gneralized_polynomoial_ud_Omega}, \ref{cor_odious_evil_Omega}, \ref{cor_dynamical_MVT_squarefree}, \ref{cor_dynamical_MVT_omega}, \ref{cor_ortho_Omega_rational_phases},  \ref{cor_dynamical_MVT_additive_functions}, and \ref{cor_dynamical_MVT_several_additive_functions}.

\cref{sec_disjointness_add_mult_actions} contains the proofs of \cref{thm_ortho_fg_sue_nilsequences} (in \cref{sec_proof_thm_C}),  \cref{thm_ortho_fg_sue_horocycle} (in \cref{sec_horocycle}), as well as the proofs of Corollaries \ref{cor_ortho_Omega_linear_phases}, \ref{cor_ortho_Omega_Besicovitch_ap_fctns}, \ref{cor_polynomoial_joint_ud}, and \ref{cor_ortho_fg_sue_nilsequences} (in \cref{sec_applications_thm_C}).
Finally, in \cref{sec_entropy}, we discuss topological entropy for additive and multiplicative systems and prove that finitely generated multiplicative topological dynamical systems have zero entropy.

\paragraph{\textbf{Acknowledgements}.} We thank Tomasz Downarowicz and Alexander Leibman for providing useful references and the anonymous referees for their helpful comments. 
The second author is supported by the National Science Foundation under grant number DMS~1901453.

\section{Distribution of orbits along $\Omega(n)$}
\label{sec_distr_orbits_along_Omega}

This section is dedicated to the proof of \cref{thm_dynamical_MVT_Omega} and is divided into two subsections.
In \cref{sec_proof_distr_orbits_along_Omega} we give a proof of \cref{thm_dynamical_MVT_Omega} conditional on three technical results, namely Propositions \ref{prop_coprimality_criterion} and \ref{lem_coprimaility_measures_1} and \cref{lem_hilfslemma_1}, whose proofs are presented afterwards in \cref{sec_aux_thm_A}.

\subsection{Proof of \cref{thm_dynamical_MVT_Omega}}
\label{sec_proof_distr_orbits_along_Omega}
As was mentioned in \cref{rem_methods}, our proof of \cref{thm_dynamical_MVT_Omega} does not rely on technology from analytic number theory.
Instead, our methods are combinatorial in nature and involve special types of averages over \define{almost primes} (defined below). 
To motivate our approach, we begin with a brief discussion of a well-known corollary of the \Turan{}-Kubilius inequality.

Recall that $\P\subset\N$ denotes the set of prime numbers and write $[N]$ for the set $\{1,\ldots,N\}$.
For a finite and non-empty set $B\subset \N$ and a function $a\colon B\to\C$ we denote the \define{\Cesaro{} average} of $a$ over $B$ and the \define{logarithmic average} of $a$ over $B$ respectively by
$$
\BEu{n\in B} a(n) \,\coloneqq\, \frac{1}{|B|}\sum_{n\in B} a(n)\qquad\text{and }\qquad
\BEul{n\in B} a(n) \,\coloneqq\, \frac{\sum_{n\in B} {a(n)}/{n}}{\sum_{n\in B}{1}/{n}}.
$$
As was already observed by Daboussi \cite[Lemma 1]{Daboussi75} and \Katai{} \cite[Eq.\ (3.1)]{Katai86}, it follows from the \Turan{}-Kubilius inequality (see for instance \cite[Lemma 4.1]{Elliott71}) that
\begin{equation}
\label{eqn_TK}
\lim_{s\to\infty}\,\lim_{N\to\infty}\,\BEu{n\in [N]} \left|\BEul{p\in \P\cap [s]} \big(1- p \1_{p\mid n}\big)\right| \,=\,0,
\end{equation}
where
$$
\1_{p\mid n}=
\begin{cases}
1,& \text{if $p$ divides n,} 
\\
0,&\text{otherwise}.
\end{cases}
$$
One way of interpreting \eqref{eqn_TK} is to say that for ``large'' $s$ and for ``almost all'' $n\in\N$ the number of primes in the interval $[s]$ that divide $n$ is approximately equal to $\sum_{p\leq s}1/p$.
Even though \eqref{eqn_TK} is commonly viewed as a corollary of the \Turan{}-Kubilius inequality, we remark that its proof is significantly shorter and easier (it follows by choosing $B=\P\cap [s]$ in \cref{prop_coprimality_criterion} below).

An equivalent form of \eqref{eqn_TK}, which will be particularly useful for our purposes, states that for all $\epsilon>0$ there exists $s_0\in\N$ such that for all arithmetic functions $a\colon\N\to\C$ bounded in absolute value by $1$ and all $s\geq s_0$ one has
\begin{equation}
\label{eqn_functional_TK}
\limsup_{N\to\infty}\,\left|\BEu{n\in [N]} a(n)\,-\,  \BEul{p\in \P\cap [s]}\,  \BEu{n\in [\nicefrac{N}{p}]} a(pn) \right| \,\leq\,\epsilon.
\end{equation}
We note that \eqref{eqn_TK} is a special case of the so-called \define{dual form} of the \Turan{}-Kubilius inequality, see \cite[Lemma 4.7]{Elliott71}, and generalizations thereof have recently found numerous fruitful applications (cf.\ \cite{FH18a,Tao16,TT18}). 
An important role in our proof of \cref{thm_dynamical_MVT_Omega} will be played by a variant of \eqref{eqn_functional_TK}, asserting that
\begin{equation}
\label{eqn_coprimality_measure}
\limsup_{N\to\infty}\,\left|\BEu{n\in [N]} a(n)\,-\,  \BEul{m\in B}\, \BEu{n\in [\nicefrac{N}{m}]} a(mn) \right| \,\leq\,\epsilon,
\end{equation}
for some special types of finite and non-empty subsets $B\subset\N$.
To clarify which choices of $B$ work,
besides $B=\P\cap[s]$ as in \eqref{eqn_functional_TK}, we will provide an easy to check criterion.
Roughly speaking, our criterion says that $B$ is good for \eqref{eqn_coprimality_measure} if two integers $n$ and $m$ chosen at random from $B$ 
have a ``high chance'' of being coprime.
The precise statement is as follows.

\begin{Proposition}
\label{prop_coprimality_criterion}
Let $B\subset\N$ be finite and non-empty. For any arithmetic function $a\colon\N\to\C$ bounded in modulus by $1$ we have
\begin{equation}
\label{eqn_coprimality_criterion}
\limsup_{N\to\infty}\,\left|\BEu{n\in [N]} a(n)\,-\,  \BEul{m\in B}\, \BEu{n\in [\nicefrac{N}{m}]} a(mn) \right| \,\leq\,\left( \BEul{m\in B}\BEul{n\in B}\Phi(n,m) \right)^{1/2},
\end{equation}
where $\Phi\colon \N\POH\N\to \N\cup\{0\}$ is the function $\Phi(m,n)\coloneqq\gcd(m,n)-1$.
\end{Proposition}

A proof of \cref{prop_coprimality_criterion} can be found in \cref{sec_aux_thm_A}.

The usefulness of \cref{prop_coprimality_criterion} is that it reduces the task of finding sets for which \eqref{eqn_coprimality_measure} holds to the easier task of exhibiting sets for which $\mathbb{E}^{\log}_{m\in B}\mathbb{E}^{\log}_{n\in B} \Phi(m,n)$ is very small. For example, the initial segment of the set of \define{$k$-almost primes},  $\P_k\coloneqq \{n\in\N: \Omega(n)=k\}$, has this property.
Indeed, one can verify that for any $k\in\N$ and any $\epsilon>0$ there exists $s_0\in\N$ such that for all $s\geq s_0$ one has $\mathbb{E}^{\log}_{m\in \P_k\cap [s]}\mathbb{E}^{\log}_{n\in \P_k\cap [s]} \Phi(m,n) \leq \epsilon^2$. In light of \cref{prop_coprimality_criterion}, this implies
\begin{equation}
\label{eqn_functional_TK_almostprimes}
\limsup_{N\to\infty}\,\left|\BEu{n\in [N]} a(n)\,-\,  \BEul{m\in \P_k\cap [s]}\,  \BEu{n\in [\nicefrac{N}{m}]} a(mn) \right| \,\leq\,\epsilon,
\end{equation}
which is a natural generalization of \eqref{eqn_functional_TK} and perhaps of independent interest.

It is also interesting to observe that $\mathbb{E}_{m\in \P_k\cap [s]}\mathbb{E}_{n\in \P_k\cap [s]} \Phi(m,n)$ goes to $\infty$ as $s\to\infty$. In particular, \eqref{eqn_TK}, \eqref{eqn_functional_TK}, and \eqref{eqn_functional_TK_almostprimes} fail severely if one tries to replace logarithmic averages with \Cesaro{} averages.

One of the main technical ingredients in our proof of \cref{thm_dynamical_MVT_Omega} is \cref{lem_coprimaility_measures_1} below.
It guarantees the existence of two finite sets $B_1$ and $B_2$ with a number of useful properties and with its help we will be able to finish the proof of \cref{thm_dynamical_MVT_Omega} rather quickly.

\begin{Lemma}
\label{lem_coprimaility_measures_1}
For all $\epsilon\in(0,1)$ and $\rho\in(1,1+\epsilon]$ there exist finite and non-empty sets $B_1,B_2\subset\N$ with the following properties: 
\begin{enumerate}	
[label=(\alph{enumi}),ref=(\alph{enumi}),leftmargin=*]
\item\label{itm_a}
$B_1\subset \P$ and $B_2\subset \P_2$;
\item\label{itm_b}
$|B_1\cap [\rho^j,\rho^{j+1})|=|B_2\cap [\rho^j,\rho^{j+1})|$ for all $j\in\N\cup\{0\}$;
\item\label{itm_c}
$\mathbb{E}^{\log}_{m\in B_1}\mathbb{E}^{\log}_{n\in B_1} \Phi(m,n)\leq \epsilon$ as well as $\mathbb{E}^{\log}_{m\in B_2}\mathbb{E}^{\log}_{n\in B_2} \Phi(m,n)\leq \epsilon$, where $\Phi$ is as in \cref{prop_coprimality_criterion}.
\end{enumerate}
\end{Lemma}

The proof of \cref{lem_coprimaility_measures_1} is given in \cref{sec_aux_thm_A}.
Before we embark on the proof of \cref{thm_dynamical_MVT_Omega}, we need one final technical lemma whose proof is also delayed until \cref{sec_aux_thm_A}.

\begin{Lemma}
\label{lem_hilfslemma_1}
Fix $\epsilon\in(0,1)$ and $\rho\in(1,1+\epsilon]$. 
Let $B_1$ and $B_2$ be finite non-empty subsets of $\N$ with the property that $|B_1\cap [\rho^j,\rho^{j+1})|=|B_2\cap [\rho^j,\rho^{j+1})|$ for all $j\in\N\cup\{0\}$. Then for any $a\colon\N\to \C$ with $|a|\leq 1$ we have \begin{equation}
\label{eqn_dbl_ave_0}
\left|\,\BEul{p\in B_{1}} \, \BEu{n\in [\nicefrac{N}{p}]} a(n) ~-~ \BEul{q\in B_{2}} \, \BEu{n\in [\nicefrac{N}{q}]} a(n)\,\right| ~\leq~ 5\epsilon.
\end{equation}
\end{Lemma}

We are now ready to prove \cref{thm_dynamical_MVT_Omega}.

\begin{proof}[Proof of \cref{thm_dynamical_MVT_Omega} assuming \cref{prop_coprimality_criterion} and Lemmas \ref{lem_coprimaility_measures_1} and \ref{lem_hilfslemma_1}]
Let $x\in X$ be arbitrary.
Our goal is to show
\begin{equation}
\label{eqn_ud_of_Omega_2}
\lim_{N\to\infty}\, \BEu{n\in[N]} f\big(T^{\Omega(n)}x\big)
~=~\int f\d\mu,\qquad\forall f\in\Cont(X).
\end{equation}
For $N\in\N$ denote by $\mu_N$ the Borel probability measure on $X$ uniquely determined by $\int f\d\mu_N= \mathbb{E}_{n\in[N]} f(T^{\Omega(n)}x)$ for all $f\in\Cont(X)$.
Then \eqref{eqn_ud_of_Omega_2} is equivalent to the assertion that $\mu_N\to \mu$ as $N\to\infty$ in the weak-$^*$ topology on $X$.
Since $(X,\mu,T)$ is uniquely ergodic, to prove $\mu_N\to\mu$ it suffices to show that any limit point of $\{\mu_N: N\in\N\}$ is $T$-invariant, because then all limit points of $\{\mu_N: N\in\N\}$ equal $\mu$ and hence the limit exists and equals $\mu$.
The $T$-invariance of any limit point of $\{\mu_N: N\in\N\}$ follows from
\begin{equation}
\label{eqn_ud_of_Omega_3n}
\lim_{N\to\infty}\left|\BEu{n\in[N]} f\big(T^{\Omega(n)}x\big)~-~\BEu{n\in[N]} f\big(T^{\Omega(n)+1}x\big)\right|~=~0
\end{equation}
for all $f\in\Cont(X)$.

For the proof of \eqref{eqn_ud_of_Omega_3n}, fix $f\in\Cont(X)$.
We can assume without loss of generality that $f$ is bounded in modulus by $1$.
Let $\epsilon\in (0,1)$ and $\rho\in (1,1+\epsilon]$ be arbitrary and, as guaranteed by \cref{lem_coprimaility_measures_1}, find two finite sets $B_1,B_2\subset\N$ satisfying conditions \ref{itm_a}, \ref{itm_b}, and \ref{itm_c}.
Combining \eqref{eqn_coprimality_criterion} with property \ref{itm_c} gives
$$
\BEu{n\in [N]} f\big(T^{\Omega(n)+1}x\big)~=~  \BEul{p\in B_{1}} \, \BEu{n\in [\nicefrac{N}{p}]} f\big(T^{\Omega(pn)+1}x\big) \,+\,\Oh( \epsilon^{1/2})\,+\,\oh_{N\to\infty}(1)
$$
as well as
$$
\BEu{n\in [N]} f\big(T^{\Omega(n)}x\big)~=~  \BEul{q\in B_{2}} \, \BEu{n\in [\nicefrac{N}{q}]} f\big(T^{\Omega(qn)}x\big) \,+\,\Oh( \epsilon^{1/2})\,+\,\oh_{N\to\infty}(1).
$$
Since $B_1$ is comprised only of primes, we have $T^{\Omega(pn)+1}x= T^{\Omega(n)+2}x$ for all $p\in B_1$.
Similarly we have $T^{\Omega(qn)}x= T^{\Omega(n)+2}x$ for all $q\in B_2$, because $B_2$ is comprised only of $2$-almost primes.
We conclude that
\begin{equation}
\label{eqn_ud_of_Omega_3n_2}
\begin{split}
\bigg| & \BEu{n\in[N]} f\big(T^{\Omega(n)}x\big)~-~\BEu{n\in[N]} f\big(T^{\Omega(n)+1}x\big)\bigg|
\\
&~=~\left|
 \BEul{p\in B_{1}} \, \BEu{n\in [\nicefrac{N}{p}]} f\big(T^{\Omega(n)+2}x\big) ~-~
\BEul{q\in B_{2}} \, \BEu{n\in [\nicefrac{N}{q}]} f\big(T^{\Omega(n)+2}x\big) 
\right|\,+\,\Oh( \epsilon^{1/2})\,+\,\oh_{N\to\infty}(1)
\end{split}
\end{equation}
Finally, combining \eqref{eqn_ud_of_Omega_3n_2} with \eqref{eqn_dbl_ave_0} from \cref{lem_hilfslemma_1} yields
\[
\bigg| \BEu{n\in[N]} f\big(T^{\Omega(n)}x\big)~-~\BEu{n\in[N]} f\big(T^{\Omega(n)+1}x\big)\bigg|~=~\Oh( \epsilon^{1/2})\,+\,\oh_{N\to\infty}(1).
\]
Letting $\epsilon$ tend to $0$ finishes the proof of \eqref{eqn_ud_of_Omega_3n}.
\end{proof}

\subsection{Proofs of \cref{prop_coprimality_criterion} and Lemmas \ref{lem_coprimaility_measures_1} and \ref{lem_hilfslemma_1}}
\label{sec_aux_thm_A}

We begin with the proof of \cref{lem_hilfslemma_1}.

\begin{proof}[Proof of \cref{lem_hilfslemma_1}]
By splitting up the logarithmic averages over $B_1$ and $B_2$ into ``$\rho$-adic'' intervals $[\rho^j,\rho^{j+1})$ and using the triangle inequality, we obtain
\begin{equation}
\label{eqn_dbl_ave_1}
\begin{split}
&\left|\,\BEul{p\in B_{1}} \, \BEu{n\in [\nicefrac{N}{p}]} a(n) ~-~ \BEul{q\in B_{2}} \, \BEu{n\in [\nicefrac{N}{q}]} a(n)\,\right|
\\
&\qquad= ~ \left|\,\sum_{j=0}^\infty\left(
\BEul{p\in B_{1}} 1_{[\rho^j,\rho^{j+1})}(p) \BEu{n\in [\nicefrac{N}{p}]} a(n) ~-~ \BEul{q\in B_{2}} 1_{[\rho^j,\rho^{j+1})}(q) \BEu{n\in [\nicefrac{N}{q}]} a(n)\,\right)\right|
\\
&\qquad\leq ~ \sum_{j=0}^\infty~
\left|\,\BEul{p\in B_{1}} 1_{[\rho^j,\rho^{j+1})}(p) \BEu{n\in [\nicefrac{N}{p}]} a(n) ~-~ \BEul{q\in B_{2}} 1_{[\rho^j,\rho^{j+1})}(q) \BEu{n\in [\nicefrac{N}{q}]} a(n)\,\right|.
\end{split}
\end{equation}
Observe that for any $p,q\in [\rho^j,\rho^{j+1})$ we have
\begin{equation}
\label{eqn_dbl_ave_2}
\left|\BEu{n\in [\nicefrac{N}{p}]} a(n) \,-\,\BEu{n\in [\nicefrac{N}{\rho^j}]} a(n)\right|~\leq~ \epsilon
\quad\text{and}\quad
\left|\BEu{n\in [\nicefrac{N}{q}]} a(n) \,-\,\BEu{n\in [\nicefrac{N}{\rho^j}]} a(n)\right|~\leq~ \epsilon,
\end{equation}
because $\rho\leq 1+\epsilon$.
Since $B_1\cap[\rho^j,\rho^{j+1})$ has the same cardinality as $B_2\cap[\rho^j,\rho^{j+1})$, we also have
\begin{align*}
\BEul{p\in B_{1}} 1_{[\rho^j,\rho^{j+1})}(p)
&\leq \frac{\frac{1}{\rho^j}|B_1\cap [\rho^j,\rho^{j+1})|}{\sum_{p\in B_1}1/p}
\\
&= \frac{\frac{1}{\rho^j}|B_2\cap [\rho^j,\rho^{j+1})|}{\sum_{p\in B_1}1/p}
\\
&\leq  (1+\epsilon)~ \frac{\frac{1}{\rho^{j}}|B_2\cap [\rho^j,\rho^{j+1})|}{\sum_{p\in B_2}1/p}
\\
&\leq (1+\epsilon)^2~ \BEul{q\in B_{2}} 1_{[\rho^j,\rho^{j+1})}(q).
\end{align*}
Similarly, one can show that
$\E_{p\in B_{1}}^{\log{}} 1_{[\rho^j,\rho^{j+1})}(p) \geq 
\big(1+\epsilon\big)^{-2} \, \E_{q\in B_{2}}^{\log{}} 1_{[\rho^j,\rho^{j+1})}(q)$.
Since $\epsilon<1$, we can replace $(1+\epsilon)^{2}$ with $(1+3\epsilon)$ and $(1-\epsilon)^{-2}$ with $(1-3\epsilon)$ and obtain
\begin{equation}
\label{eqn_dbl_ave_3}
\big(1-3\epsilon\big) ~\BEul{q\in B_{2}} 1_{[\rho^j,\rho^{j+1})}(q)~\leq~ \BEul{p\in B_{1}} 1_{[\rho^j,\rho^{j+1})}(p) ~\leq~ \big(1+3\epsilon\big) ~\BEul{q\in B_{2}} 1_{[\rho^j,\rho^{j+1})}(q).
\end{equation}
Combining \eqref{eqn_dbl_ave_1}, \eqref{eqn_dbl_ave_2} and \eqref{eqn_dbl_ave_3} proves \eqref{eqn_dbl_ave_0}.
\end{proof}

Next we state and prove a lemma which will be useful for the proof of \cref{prop_coprimality_criterion}.

\begin{Lemma}
\label{lem_vTK}
Let $B$ be a finite and non-empty subset of $\N$ and recall that $\Phi(m,n)=\gcd(m,n)-1$.
Then
\begin{equation}
\label{eqn_Turan-Kubilius-2}
\lim_{N\to\infty} \BEu{n\in [N]} \left|\BEul{m\in B} \big(1- m \1_{m\mid n}\big)\right|^2 ~=~ \BEul{l\in B}\BEul{m\in B} \Phi(l,m).
\end{equation}
\end{Lemma}

\begin{proof}
By expanding the square on the left hand side of \eqref{eqn_Turan-Kubilius-2} we get
\begin{equation}
\label{eqn_vTK-1}
\BEu{n\in [N]} \left|\BEul{m\in B} \big(1- m \1_{m\mid n}\big)\right|^2 ~=~ 1 -2 \Sigma_1+\Sigma_2,
\end{equation}
where $\Sigma_1\coloneqq \mathbb{E}_{n\in [N]} \mathbb{E}^{\log}_{m\in B} m 1_{m\mid n}$ and $\Sigma_2\coloneqq \mathbb{E}_{n\in [N]}\mathbb{E}^{\log}_{l,m\in B} (l 1_{l\mid n})(m1_{m\mid n})$.

Note that $\mathbb{E}_{n\in [N]} m\1_{m\mid n}=1+\Oh\left(1/N\right)$ and therefore
\begin{equation}
\label{eqn_vTK-3}
\Sigma_1 = 1+ \Oh\left(\tfrac{1}{N}\right).
\end{equation}
Similarly, since $\mathbb{E}_{n\in [N]} lm \1_{l\mid n}\1_{m\mid n} = \gcd(l,m) + \Oh\left(1/ N\right)$, we have
\begin{equation}
\label{eqn_vTK-2}
\Sigma_2~=~\BEul{l\in B}\BEul{m\in B} \gcd(l,m)\,+\,\Oh\left(\tfrac{1}{N}\right)
~=~
1+\BEul{l\in B}\BEul{m\in B} \Phi(l,m) \,+\, \Oh\left(\tfrac{1}{N}\right).
\end{equation}
Putting together \eqref{eqn_vTK-1}, \eqref{eqn_vTK-3} and \eqref{eqn_vTK-2} completes the proof of \eqref{eqn_Turan-Kubilius-2}.
\end{proof}

\begin{proof}[Proof of \cref{prop_coprimality_criterion}]
It follows from \cref{lem_vTK} and the Cauchy-Schwarz inequality that
\begin{equation*}
\label{eqn_vTK-4}
\lim_{N\to\infty} \BEu{n\in [N]} \left|\BEul{m\in B} \big(1- m \1_{m\mid n}\big)\right| ~\leq~ \left(\BEul{l\in B}\BEul{m\in B} \Phi(l,m)\right)^{1/2}.
\end{equation*}
Thus, we have
\begin{eqnarray*}
\limsup_{N\to\infty}\left|\BEu{n\in [N]} a(n)\,-\,  \BEul{m\in B}\, \BEu{n\in [\nicefrac{N}{m}]} a(mn) \right|
&=& \limsup_{N\to\infty}\left|\BEu{n\in [N]} \BEul{m\in B} a(n)\,\big(1- m \1_{m\mid n}\big)\right|
\\
&\leq& \limsup_{N\to\infty}\BEu{n\in [N]} \left|\BEul{m\in B} \big(1- m \1_{m\mid n}\big)\right|.
\end{eqnarray*}
This proves \eqref{eqn_coprimality_criterion}.
\end{proof}

Now that we have finished the proofs of \cref{lem_hilfslemma_1} and \cref{prop_coprimality_criterion}, the remainder of this subsection is dedicated to the proof of \cref{lem_coprimaility_measures_1}.
We begin with two helpful lemmas.

\begin{Lemma}
\label{lem_coprimaility_measurement_primes}
Fix $\epsilon\in (0,1)$ and let $B\subset \P$ be a finite set of primes satisfying $\sum_{m\in B} 1/m \geq \frac{1}{\epsilon}$. Then $$
\BEul{m\in B}\BEul{n\in B} \Phi(m,n)\,\leq\, \epsilon.
$$
\end{Lemma}
\begin{proof}
Since all elements in $B$ are prime numbers, for $m,n\in B$ the quantity $\Phi(m,n)=\gcd(m,n)-1$ is non-zero if and only if $m=n$. Hence
\begin{eqnarray*}
\BEul{m\in B}\BEul{n\in B} \Phi(m,n)
&=& \frac{\sum_{m,n\in B} \frac{\gcd(m,n)-1}{mn} }{\left(\sum_{m\in B} \frac{1}{m}\right)^{2}}
\\
&=& \frac{\sum_{m\in B} \frac{m-1}{m^2} }{\left(\sum_{m\in B} \frac{1}{m}\right)^{2}}
\\
&\leq& \frac{\sum_{m\in B} \frac{1}{m}} {\left(\sum_{m\in B} \frac{1}{m}\right)^{2}}.
\\
&\leq& \epsilon.
\end{eqnarray*}
This finishes the proof.
\end{proof}

\begin{Lemma}
\label{lem_coprimaility_measurement_2almostprimes}
Fix $\epsilon\in(0,1)$. Let $P_1,P_2\subset \P$ be finite sets of primes satisfying $\sum_{p\in P_1} 1/p \geq \frac{3}{\epsilon}$ and $\sum_{q\in P_2} 1/q \geq \frac{3}{\epsilon}$. Define $B\coloneqq \{pq: p\in P_1,~ q\in P_2\}$. Then
$$
\BEul{m\in B}\BEul{n\in B} \Phi(m,n)\,\leq\, \epsilon.
$$
\end{Lemma}

\begin{proof}
First we calculate that
\begin{equation*}
\BEul{m\in B}\BEul{n\in B} \Phi(m,n)
~=~ \frac{\sum_{m,n\in B} \frac{\gcd(m,n)-1}{mn} }{\left(\sum_{m\in B} \frac{1}{m}\right)^{2}}.
\end{equation*}
Note that for $m,n\in B$ the number $\gcd(m,n)$ is either $1$, or and element of $P_1$, or an element of $P_2$, or an element of $B$. If it is $1$ then $\gcd(m,n)-1=0$ and so this term does not contribute to $\mathbb{E}^{\log}_{m\in B}\mathbb{E}^{\log}_{n\in B} \Phi(m,n)$ at all.
The case where $\gcd(m,n)$ belongs to $B$ can only happen if $m=n$. We can therefore write
$$
\sum_{m,n\in B} \tfrac{\gcd(m,n)-1}{mn}
\,=\,
\underbrace{\sum_{m,n\in B\atop \gcd(m,n)\in P_1} \tfrac{\gcd(m,n)-1}{mn}}_{[1]}
\,+\,
\underbrace{\sum_{m,n\in B\atop \gcd(m,n)\in P_2} \tfrac{\gcd(m,n)-1}{mn}}_{[2]}
\,+\,\underbrace{\sum_{m,n\in B\atop m=n} \tfrac{\gcd(m,n)-1}{mn}}_{[3]}.
$$
The third term can be bounded from above as follows:
$$
[3]~\leq~ \sum_{m\in B}\tfrac{1}{m}
~=~\Bigg(\sum_{p\in P_1}\tfrac{1}{p}\Bigg)\Bigg(\sum_{q\in P_2}\tfrac{1}{q}\Bigg).
$$
To estimate $[1]$, note that if $m,n\in B$ with $\gcd(m,n)\in P_1$ then there exists $p\in P_1$ and $q_1,q_2\in P_2$ such that $m=pq_1$ and $n=pq_2$. In this case, we have
$$
\frac{\gcd(m,n)-1}{mn}~\leq~ \frac{1}{pq_1 q_2}.
$$
This gives us
$$
[1]~\leq~\sum_{p\in P_1}\sum_{q_1,q_2\in P_2}\tfrac{1}{p q_1 q_2} ~=~\Bigg(\sum_{p\in P_1}\tfrac{1}{p}\Bigg)\Bigg(\sum_{q\in P_2}\tfrac{1}{q}\Bigg)^2.
$$
By symmetry we have
$$
[2]~\leq~\Bigg(\sum_{p\in P_1}\frac{1}{p}\Bigg)^2 \Bigg(\sum_{q\in P_2}\tfrac{1}{q}\Bigg).
$$
We conclude that
\begin{eqnarray*}
\frac{\sum_{m,n\in B} \frac{\gcd(m,n)-1}{mn} }{\left(\sum_{m\in B} \frac{1}{m}\right)^{2}}
&=&
\frac{[1]\,+\,[2]\,+\,[3]}{\left(\sum_{m\in B} \frac{1}{m}\right)^{2}}
\\
&=&
\frac{[1]\,+\,[2]\,+\,[3]}{\big(\sum_{p\in P_1}\tfrac{1}{p}\big)^2\big(\sum_{q\in P_2}\tfrac{1}{q}\big)^2}
\\
&\leq &
\frac{1}{\sum_{p\in P_1}\tfrac{1}{p}}\,+\,\frac{1}{\sum_{q\in P_2}\tfrac{1}{q}}\,+\,\frac{1}{\big(\sum_{p\in P_1}\tfrac{1}{p}\big)\big(\sum_{q\in P_2}\tfrac{1}{q}\big)}
\\
&\leq & \frac{\epsilon}{3}\,+\,\frac{\epsilon}{3}\,+\,\left(\frac{\epsilon}{3}\right)^2
\\
&< &\epsilon.
\end{eqnarray*}
This finishes the proof.
\end{proof}

\begin{proof}[Proof of \cref{lem_coprimaility_measures_1}]
Fix $\epsilon\in(0,1)$ and $\rho\in(1,1+\epsilon]$. It is a consequence of the Prime Number Theorem that there exists $j_0\in\N$ and $C>0$ such that
$$
\big|\P\cap \big[\rho^{j},\rho^{j+1/2}\big)\big|~\geq~
\frac{C\rho^j}{j},\qquad\forall j\geq j_0.
$$
Pick $s\in\N$ sufficiently large such that $\sum_{j_0\leq l< s} C/l\geq 3/\epsilon$, and define $$
P_{1,l}~\coloneqq~ \P\cap \big[\rho^{l},\rho^{l+1/2}\big)\qquad\text{and}\qquad
P_1~\coloneqq~\bigcup_{j_0\leq l< s} P_{1,l}.
$$
Observe that
\begin{equation}
\label{eqn_coprimaility_measures_1_1}
\sum_{p\in P_1}\frac{1}{p} ~=~ \sum_{j_0\leq l< s} ~\sum_{p\in P_{1,l}}\frac{1}{p}~\geq~ \sum_{j_0\leq l< s} \frac{|P_{1,l}|}{\rho^l} ~\geq~ \sum_{j_0\leq l<s}\frac{C}{l}~\geq~\frac{3}{\epsilon}.
\end{equation}
Next, choose $t\in\N$ sufficiently large such that $\sum_{j\in[t]} C/(2|P_1|sj)\geq 3/\epsilon$. Then, for every $j\in[t]$ let $P_{2,j}$ be a subset of $\P\cap [\rho^{sj},\rho^{sj+1/2})$ satisfying
$$
\frac{C \rho^{sj}}{2|P_1|sj}~\leq~ |P_{2,j}| ~\leq ~\frac{C \rho^{sj}}{|P_1|sj},
$$
and define
$$
P_2~\coloneqq~\bigcup_{j\in [t]} P_{2,j}.
$$ 
A similar calculation to \eqref{eqn_coprimaility_measures_1_1} shows that
\begin{equation}
\label{eqn_coprimaility_measures_1_2}
\sum_{p\in P_2}\frac{1}{p}~\geq~\frac{3}{\epsilon}.
\end{equation}
Define $B_2\coloneqq P_1 \cdot P_2$. Certainly, $B_2\subset \P_2$. Moreover, combining \cref{lem_coprimaility_measurement_2almostprimes} with \eqref{eqn_coprimaility_measures_1_1} and \eqref{eqn_coprimaility_measures_1_2} shows that $\mathbb{E}^{\log}_{m\in B_2}\mathbb{E}^{\log}_{n\in B_2} \Phi(m,n)\leq \epsilon$. Note that
$$
B_2\cap \big[\rho^{sj+l}, \rho^{sj+l+1}\big) ~=~ P_{1,l} \cdot P_{2,j}
$$
for all $j_0 \leq l<s$ and $j\in[t]$. Therefore
$$
\big|B_2\cap \big[\rho^{sj+l}, \rho^{sj+l+1}\big)\big| \leq \frac{C \rho^{sj}}{sj}~\leq ~ \big|\P\cap \big[\rho^{sj+l}, \rho^{sj+l+1}\big)\big|,
$$
which allows us to find for every $j_0 \leq l<s$ and $j\in[t]$ a set $Q_{l,j}\subset \P\cap [\rho^{sj+l}, \rho^{sj+l+1})$ with $|Q_{l,j}|= |B_2\cap \big[\rho^{sj+l}, \rho^{sj+l+1}\big)|$. Now define $B_1\coloneqq \bigcup_{j_0\leq l<s}\bigcup_{j\in[t]} Q_{l,j}$. By construction, we have $|B_1\cap [\rho^j,\rho^{j+1})|=|B_2\cap [\rho^j,\rho^{j+1})|$ for all $j\in\N\cup\{0\}$. Moreover, using $|B_1\cap [\rho^{j},\rho^{j+1})|=|B_2\cap [\rho^{j+i},\rho^{j+i+1})|$, we see that
$$
\sum_{m\in B_1}\frac{1}{m}~\geq~\frac{1}{\rho} \sum_{m\in B_2}\frac{1}{m} ~=~\frac{1}{\rho}\left(\sum_{p\in P_1}\frac{1}{p}\right)\left(\sum_{q\in P_2}\frac{1}{q}\right)~\geq~ \frac{9}{\rho \epsilon^2}~\geq~ \frac{1}{\epsilon}.
$$
In light of \cref{lem_coprimaility_measurement_primes}, this shows that $\mathbb{E}^{\log}_{m\in B_1}\mathbb{E}^{\log}_{n\in B_1}\Phi(m,n)\leq \epsilon$. This finishes the proof that $B_1$ and $B_2$ satisfy \ref{itm_a}, \ref{itm_b}, and \ref{itm_c}.
\end{proof}

\begin{Remark}
While the Prime Number Theorem was used in the proof of \cref{lem_coprimaility_measures_1} to streamline its exposition, it is possible to avoid using it altogether. We refer the reader to \cite{Richter21} for more details.
\end{Remark}

\section{Finitely generated and strongly uniquley ergodic systems}
\label{sec_fg_sue_mtds}

In this section we deal with \cref{thm_dynamical_MVT_fg_sue}.
In \cref{sec_fg_sue_mtds_cntr-xmpl} we give an example which illustrates that the assumption of strong unique ergodicity imposed on the system $(Y,S)$ in \cref{thm_dynamical_MVT_fg_sue} cannot be relaxed.
In \cref{sec_proof_thmB} we present a proof of \cref{thm_dynamical_MVT_fg_sue} assuming the validity of a technical proposition, \cref{prop_idempotency}, whose proof is given in \cref{sec_aux_thm_B}.

\subsection{A counterexample}
\label{sec_fg_sue_mtds_cntr-xmpl}

The following example describes a multiplicative topological dynamical system that is finitely generated and uniquely ergodic but not strongly uniquely ergodic. We will show that for this system there exist a function $g\in\Cont(Y)$ and a point $y\in Y$ such that \eqref{eqn_ud_of_mtds} fails.

\begin{Example}
\label{example_1}
Fix an arbitrary irrational number $\alpha$ and denote by $\nu_2(n)\coloneqq \max\{e\in\Z: 2^e\mid n\}$ the $2$-adic valuation of a positive integer $n$. Using $\nu_2$, we can define a multiplicative topological dynamical system on the torus in the following way: Define for all $n\in\N$ the map $S_n\colon \T\to\T$ via
$$
S_n(x)= {x+\nu_2(n)\alpha}\bmod{1}.
$$
Since $\nu_2(nm)=\nu_2(n)+\nu_2(m)$ for all $m,n\in\N$, we have $S_{nm}=S_n\circ S_m$ for all $n,m\in\N$. In particular, $S=(S_n)_{n\in\N}$ is an action of $(\N,\cdot)$ and $(\T,S)$ is a multiplicative topological dynamical system.

Since $S_p=\mathrm{id_\T}$ for all primes $p> 2$, the system $(\T,S)$ is finitely generated.
Moreover, since $S_2$ is rotation by $\alpha$, any $S$-invariant measure on $\T$ must, in particular, be invariant under rotation by $\alpha$.
Since $\alpha$ is irrational, the normalized Lebesgue measure is the only Borel probability measure with this property.
We conclude that $(\T,S)$ is uniquely ergodic. But $(\T,S)$ is not strongly uniquely ergodic, because $S_p=\mathrm{id_\T}$ for all primes $p> 2$ and so any Borel probability measure on $\T$ pretends to be invariant under $S$ (see \cref{def_fg_sue}).
To summarize, $(\T,S)$ is a finitely generated multiplicative topological dynamical system that is uniquely ergodic but not strongly uniquely ergodic. We also have that for all $y\in Y$ and $g\in\Cont(Y)$
\begin{align*}
\lim_{N\to\infty}\,\BEu{n\in[N]} g(S_n y) &\,=\,
\lim_{N\to\infty}~\frac{1}{N}\sum_{n=1}^N g(y+\nu_2(n)\alpha)
\\
&\,=\,\lim_{N\to\infty}~\sum_{0\leq i\leq \frac{\log(N)}{\log(2)}}\tfrac{|\{n\in[N]: \nu_2(n)=i\}|}{N}\, g(y+i\alpha)
\\
&\,=\,\tfrac{1}{2} g(y)+\tfrac{1}{4} g(y+\alpha)+\tfrac{1}{8} g( y+2\alpha)+\tfrac{1}{16} g(y+3\alpha)+\ldots,
\end{align*}
 which shows that there exists no Borel probability measure $\nu$ on $\T$ such that \eqref{eqn_ud_of_mtds} holds for the system $(\T,S)$ and all $y\in Y$ and $g\in\Cont(Y)$.
\end{Example}

\subsection{Proof of \cref{thm_dynamical_MVT_fg_sue}}
\label{sec_proof_thmB}

This subsection is dedicated to the proof of \cref{thm_dynamical_MVT_fg_sue}. The main ingredient in our proof is the following technical result, which is proved in \cref{sec_aux_thm_B}.

\begin{Proposition}
\label{prop_idempotency}
 Let $(Y,S)$ be a finitely generated multiplicative topological dynamical system and let $R_1,\ldots, R_d$ denote the generators of $S$. Then for every $y\in Y$, $g\in\Cont(Y)$, and $e\in [d]$ with $\sum_{p\in \P,\, S_p=R_e}1/p=\infty$ we have
\begin{equation}
\label{eqn_idempotency_1}
\lim_{N\to\infty}\,\left|\BEu{n\in[N]} g(R_{e}S_n y)~-~\BEu{n\in[N]} g(R_{e}^2S_n y) \right| ~=~0.
\end{equation}
\end{Proposition}


\begin{proof}[Proof of \cref{thm_dynamical_MVT_fg_sue} assuming \cref{prop_idempotency}]
Let $(Y,S)$ be a finitely generated multiplicative topological dynamical system and let $R_1,\ldots, R_d$ denote the generators of $S$.
We also assume that $(Y,S)$ is strongly uniquely ergodic, i.e., there exists only one Borel probability measure $\nu$ on $Y$ that pretends to be invariant under $S$ (see \cref{def_fg_sue}).
For $e\in[d]$ define $P_e\coloneqq \{p\in\P: S_p=R_e\}$ and let $y\in Y$ be fixed.
Our goal is to show
\begin{equation}
\label{eqn_ud_fg_sue_2}
\lim_{N\to\infty}\, \BEu{n\in[N]} g\big(S_n y\big)
~=~\int g\d\nu
\end{equation}
for every $g\in\Cont(X)$.

Let $\nu_N$ be the Borel probability measure on $Y$ that is uniquely determined by $\int g\d\nu_N= \mathbb{E}_{n\in[N]} g(S_n y)$ for all $g\in\Cont(Y)$.
Certainly, \eqref{eqn_ud_fg_sue_2} is equivalent to the assertion that $\nu_N\to \nu$ as $N\to\infty$.
Since $(Y,S)$ is strongly uniquely ergodic, to prove that $\nu_N\to \nu$ it suffices to show that any limit point of $\{\nu_N: N\in\N\}$ pretends to be invariant under $S$, or in other words, any limit point of $\{\nu_N: N\in\N\}$ is invariant under $R_e$ for all $e\in[d]$ for which $\sum_{p\in P_e} 1/p=\infty$.
This, in turn, follows from
\begin{equation}
\label{eqn_ud_fg_sue_3}
\lim_{N\to\infty}\,\left|\BEu{n\in[N]} g(S_n y)~-~\BEu{n\in[N]} g(R_{e}S_n y) \right| ~=~0.
\end{equation}

It remains to verify \eqref{eqn_ud_fg_sue_3}.
Fix $\epsilon\in (0,1)$ and $e\in[d]$ for which $\sum_{p\in P_e} 1/p=\infty$.
Applying \cref{prop_coprimality_criterion} once with $a(n)=g(S_ny)$ and once with $a(n)=g(R_e S_n y)$, we obtain
\[
\limsup_{N\to\infty}\,\left|\BEu{n\in[N]} g(S_n y)~-~\BEul{p\in B}\BEu{n\in[\nicefrac{N}{p}]} g(R_eS_{n} y) \right| \leq \left(\BEul{m\in B}\BEul{n\in B} \Phi(m,n)\right)^{1/2},
\]
and
\[
\limsup_{N\to\infty}\,\left|\BEu{n\in[N]} g(R_e S_n y)~-~\BEul{p\in B}\BEu{n\in[\nicefrac{N}{p}]} g(R_{e}^2S_{n} y) \right| \leq\left(\BEul{m\in B}\BEul{n\in B} \Phi(m,n)\right)^{1/2}.
\]
In view of \cref{lem_coprimaility_measurement_primes}, if $s$ is sufficiently large and $B\coloneqq P_e\cap [s]$, then
\begin{eqnarray}
\label{eqn_ud_fg_sue_4}
\limsup_{N\to\infty}\,\left|\BEu{n\in[N]} g(S_n y)~-~\BEul{p\in B}\BEu{n\in[\nicefrac{N}{p}]} g(R_eS_{n} y) \right| &\leq&\epsilon,
\\
\label{eqn_ud_fg_sue_5}
\limsup_{N\to\infty}\,\left|\BEu{n\in[N]} g(R_e S_n y)~-~\BEul{p\in B}\BEu{n\in[\nicefrac{N}{p}]} g(R_{e}^2S_{n} y) \right| &\leq&\epsilon.
\end{eqnarray}
From \cref{prop_idempotency} we get
\begin{equation}
\label{eqn_ud_fg_sue_6}
\lim_{N\to\infty}\,\left|
\BEu{n\in[\nicefrac{N}{p}]} g(R_{e}^2S_{n} y)~-~
\BEu{n\in[\nicefrac{N}{p}]} g(R_eS_{n} y)
\right| ~=~0.
\end{equation}
Putting together \eqref{eqn_ud_fg_sue_6} with \eqref{eqn_ud_fg_sue_4} and \eqref{eqn_ud_fg_sue_5} gives
\[
\limsup_{N\to\infty}\,\left|\BEu{n\in[N]} g(S_n y)~-~\BEu{n\in[N]} g(R_{e}S_n y) \right| ~\leq~2\epsilon.
\]
Since $\epsilon$ can be made arbitrarily small, we get \eqref{eqn_ud_fg_sue_3}.
\end{proof}

\subsection{Proof of \cref{prop_idempotency}}
\label{sec_aux_thm_B}

Let us say $A\subset \N$ is a \define{DSR (Divergent Sum of Reciprocals)} set if
\[
\sum_{n\in A} \frac{1}{n}~=~\infty.
\]
Note that DSR sets are \define{partition regular}, meaning that whenever one is partitioned into finitely many pieces, at least one of the pieces is itself DSR.

Suppose we are given a finite partition of $\N$, i.e.\ $\N=I_1 \cup\ldots \cup I_d$.
We define a relation $\sim$ on the set $[d]$ of indices of this partition via
\begin{equation}
\label{eqn_relation}
e\sim e' ~\iff~ 
\begin{array}{l}
\text{at least one of the sets $I_e\cap I_{e'}$, $I_e\cap (I_{e'}-1)$,}\\
\text{ or $I_e\cap (I_{e'}+1)$ is a DSR set.}
\end{array}
\end{equation}
The relation $\sim$ is reflexive and symmetric, but it might not be transitive. Its transitive closure is the relation $\approx$ defined as
\begin{equation}
\label{eqn_equiv_relation}
e\approx e' ~\iff~ 
\begin{array}{l}
\text{there exist}~r\in[d]~\text{and}~e_1,\ldots,e_r\in [d]~\text{such that}~e_1=e,\\
e_r=e',~\text{and}~e_{i}\sim e_{i+1}~\text{for all}~i\in[r-1]. 
\end{array}
\end{equation}
It is not hard to see that $\approx$ is an equivalence relation on $[d]$.

The next lemma generalizes \cref{lem_coprimaility_measures_1} and is important for our proof of \cref{prop_idempotency}.
Recall that $\Phi(m,n)=\gcd(m,n)-1$.

\begin{Lemma}
\label{lem_coprimaility_measures_2}
Let $\epsilon\in(0,1)$ and $\rho\in (1,1+\epsilon]$. Let $(Y,S)$ be a finitely generated multiplicative topological dynamical system and let $R_1,\ldots, R_d$ denote the generators of $S$. Define, for all $e\in[d]$,
\[
P_e\,\coloneqq\, \{p\in\P: S_p=R_e\}~~\text{and}~~
I_e \,\coloneqq\, \left\{j\in\N: \big|P_e\cap \big[\rho^j,\rho^{j+1}\big)\big|\,\geq\, \frac{1}{d}\big|\P\cap \big[\rho^j,\rho^{j+1}\big)\big|\right\}.
\] 
\begin{enumerate}	
[label=(\roman{enumi})~,ref=(\roman{enumi})]
\item\label{itm_c_m_1}
For every $e_1\in[d]$ with $\sum_{p\in P_{e_1}}1/p=\infty$ there exists $e_2\in[d]$ for which $I_{e_2}$ is a DSR set, and there exist finite and non-empty sets $B_1,B_2\subset\N$ with the following properties:
\begin{enumerate}	
[label=~(i-\alph{enumii}), ref=(i-\alph{enumii}),leftmargin=*]
\item\label{itm_c_m_1_a}
$B_1\subset P_{e_1}$ and $B_2\subset P_{e_2}$;
\item\label{itm_c_m_1_b}
$|B_1\cap [\rho^{j},\rho^{j+1})|=|B_2\cap [\rho^{j},\rho^{j+1})|$ for all $j\in\N$;
\item\label{itm_c_m_1_c}
$\mathbb{E}^{\log}_{m\in B_1}\mathbb{E}^{\log}_{n\in B_1} \Phi(m,n)\leq \epsilon$ and $\mathbb{E}^{\log}_{m\in B_2}\mathbb{E}^{\log}_{n\in B_2} \Phi(m,n)\leq \epsilon$.
\end{enumerate} 
\item\label{itm_c_m_2}
For all $e_1,e_2\in[d]$ with $e_1\sim e_2$ there exist $i\in\{-1,0,1\}$ and $B_1,B_2\subset\N$ with the following properties:
\begin{enumerate}	
[label=~(ii-\alph{enumii}), ref=(ii-\alph{enumii}),leftmargin=*]
\item\label{itm_c_m_2_a}
$B_1\subset P_{e_1}$ and $B_2\subset P_{e_2}$;
\item\label{itm_c_m_2_b}
$|B_1\cap [\rho^{j},\rho^{j+1})|=|B_2\cap [\rho^{j+i},\rho^{j+i+1})|$ for all $j\in\N$;
\item\label{itm_c_m_2_c}
$\mathbb{E}^{\log}_{m\in B_1}\mathbb{E}^{\log}_{n\in B_1} \Phi(m,n)\leq \epsilon$ and $\mathbb{E}^{\log}_{m\in B_2}\mathbb{E}^{\log}_{n\in B_2} \Phi(m,n)\leq \epsilon$.
\end{enumerate}
\item\label{itm_c_m_3}
For every $e_1\in[d]$ for which $I_{e_1}$ is DSR there exist $e_2\in[d]$ with $e_1\approx e_2$, $i\in\{0,1\}$, and finite and non-empty sets $B_1,B_2\subset\N$ with the following properties:
\begin{enumerate}	
[label=~(iii-\alph{enumii}),ref=(iii-\alph{enumii}),leftmargin=*]
\item\label{itm_c_m_3_a}
$B_1\subset \{pq: p,q\in P_{e_1}\}$ and $B_2\subset P_{e_2}$;
\item\label{itm_c_m_3_b}
$|B_1\cap [\rho^{j+i},\rho^{j+i+1})|=|B_2\cap [\rho^{j},\rho^{j+1})|$ for all $j\in\N$;
\item\label{itm_c_m_3_c}
$\mathbb{E}^{\log}_{m\in B_1}\mathbb{E}^{\log}_{n\in B_1} \Phi(m,n)\leq \epsilon$ and $\mathbb{E}^{\log}_{m\in B_2}\mathbb{E}^{\log}_{n\in B_2} \Phi(m,n)\leq \epsilon$.
\end{enumerate}
\end{enumerate}
\end{Lemma}

\begin{proof}[Proof of \cref{lem_coprimaility_measures_2}, part \ref{itm_c_m_1}]
Fix $e_1\in[d]$ with $\sum_{p\in P_{e_1}}1/p=\infty$. Let $J$ denote the set of all $j\in\N$ for which $P_{e_1}\cap [\rho^j,\rho^{j+1})\neq \emptyset$. Since $\N=\bigcup_{e\in[d]} I_e$, for every $j\in J$ there exists $e^{(j)}\in[d]$ such that $j\in I_{e^{(j)}}$. Define, for $e\in[d]$, the set
\begin{equation}
\label{eqn_coprimaility_measures_2_1}
Q_e\coloneqq \bigcup_{j\in J\atop e^{(j)}=e} P_{e_1}\cap \big[\rho^j,\rho^{j+1}\big).
\end{equation}
Observe that $Q_1,\ldots,Q_d$ is a partition of $P_{e_1}$. Since $P_{e_1}$ is a DSR set and since DSR sets are partition regular, there exists $e_2\in[d]$ such that $Q_{e_2}$ is DSR.

We claim that $I_{e_2}$ is a DSR set. By \eqref{eqn_coprimaility_measures_2_1} we have $Q_{e_2}\cap [\rho^j,\rho^{j+1})=\emptyset$ unless $j\in I_{e_2}$. Therefore 
$$
\sum_{p\in Q_{e_2}}\frac{1}{p} ~=~ \sum_{j\in I_{e_2}}~ \sum_{p\in Q_{e_2}\cap [\rho^j,\rho^{j+1})}\frac{1}{p}.  
$$
Note that
$$
\sum_{p\in Q_{e_2}\cap [\rho^j,\rho^{j+1})}\frac{1}{p}~\leq~\sum_{p\in \P\cap [\rho^j,\rho^{j+1})}\frac{1}{p}~\leq~ \frac{C}{j}
$$
for some positive constant $C$.
It follows that
$$
\sum_{p\in Q_{e_2}}\frac{1}{p}~\leq~ C \sum_{j\in I_{e_2}}\frac{1}{j}.
$$
Since $Q_{e_2}$ is DSR, this proves that $I_{e_2}$ is also DSR.

Next, for every $j\in \N$ let $B_{1,j}$ be a subset of $Q_{e_2}\cap [\rho^j,\rho^{j+1})$ of size $\min\{|Q_{e_2}\cap [\rho^j,\rho^{j+1})|,|P_{e_2}\cap [\rho^j,\rho^{j+1})|\}$ and let $B_{2,j}$ be a subset of $P_{e_2}\cap [\rho^j,\rho^{j+1})$ of size $\min\{|Q_{e_2}\cap [\rho^j,\rho^{j+1})|,|P_{e_2}\cap [\rho^j,\rho^{j+1})|\}$. Since both $Q_{e_2}$ and $P_{e_2}$ are DSR sets, it follows that $\bigcup_{j\in\N} B_{1,j}$ and $\bigcup_{j\in\N} B_{2,j}$ are DSR. Therefore, there exists $s\in\N$ such that the sets $B_1\coloneqq \bigcup_{j\in[s]} B_{1,j}$ and $B_2\coloneqq \bigcup_{j\in[s]} B_{2,j}$ have the properties
$$
\sum_{m\in B_1}\frac{1}{m}\,\geq \,\frac{1}{\epsilon}\qquad\text{and}\qquad \sum_{m\in B_2}\frac{1}{m}\,\geq \,\frac{1}{\epsilon}.
$$
Therefore, in light of \cref{lem_coprimaility_measurement_primes}, the sets $B_1$ and $B_2$ satisfy \ref{itm_c_m_1_c}. Since $B_1\subset Q_{e_2}\subset P_{e_1}$ and $B_2\subset P_{e_2}$, we see that $B_1$ and $B_2$ also satisfy \ref{itm_c_m_1_a}. Finally, by construction, $|B_1\cap [\rho^{j},\rho^{j+1})|=|B_2\cap [\rho^{j},\rho^{j+1})|$ for all $j\in\N$, which implies \ref{itm_c_m_1_b}.
\end{proof}

\begin{proof}[Proof of \cref{lem_coprimaility_measures_2}, part \ref{itm_c_m_2}]
Let  $e_1,e_2\in[d]$ with $e_1\sim e_2$. This means that there exists $i\in\{-1,0,1\}$ such that $I_{e_1}\cap (I_{e_2}-i)$ is a DSR set.
Let us now show how to construct sets $B_1$ and $B_2$ satisfying \ref{itm_c_m_2_a}, \ref{itm_c_m_2_b}, and \ref{itm_c_m_2_c}. 
For every $j\in\N$ define
$$
r_j\,\coloneqq\, \min\{|P_{e_1}\cap [\rho^{j},\rho^{j+1})|,|P_{e_2}\cap [\rho^{j+i},\rho^{j+i+1})|\}.
$$
Note that for every $j\in I_{e_1}\cap (I_{e_2}-i)$ we have
\begin{equation}
\label{eqn_r_j_estimate}
r_j~\geq~ \min\left\{\frac{1}{d}\big|\P\cap [\rho^{j},\rho^{j+1})\big|,\, \frac{1}{d}\big|\P\cap [\rho^{j+i},\rho^{j+i+1})\big|\right\}.
\end{equation}
For every $j\in\N$, let $B_{1,j}$ be a subset of $P_{e_1}\cap [\rho^{j},\rho^{j+1})$ of size $r_j$ and let $B_{2,j}$ be a subset of $P_{e_2}\cap [\rho^{j+i},\rho^{j+i+1})$ also of size $r_j$.
We have
\begin{eqnarray*}
\sum_{j\in\N} ~\sum_{p\in B_{1,j}}~\frac{1}{p}
&\geq& \sum_{j\in \N} ~\sum_{p\in B_{1,j}}~\frac{1}{\rho^{j}}
\\
&\geq& \sum_{j\in\N} \frac{r_j}{ \rho^{j}}
\end{eqnarray*}
By \eqref{eqn_r_j_estimate} and the Prime Number Theorem we have that $r_j \geq C \frac{\rho^j}{j}$ for all $j\in I_{e_1}\cap (I_{e_2}-i)$ and some constant $C>0$.
This, combined with the fact that $I_{e_1}\cap (I_{e_2}-i)$ is DSR, proves that the set $\bigcup_{i\in\N} B_{1,i}$ is DSR. An analugous argument shows that $\bigcup_{i\in\N} B_{2,j}$ is DSR. 
Therefore, there exists $s\in\N$ such that $B_1\coloneqq \bigcup_{j\in[s]} B_{1,j}$ and $B_2\coloneqq \bigcup_{j\in[s]} B_{2,j}$ have the properties
$$
\sum_{m\in B_1}\frac{1}{m}\,\geq \,\frac{1}{\epsilon}\qquad\text{and}\qquad \sum_{m\in B_2}\frac{1}{m}\,\geq \,\frac{1}{\epsilon}.
$$
In light of \cref{lem_coprimaility_measurement_primes}, this means that the sets $B_1$ and $B_2$ satisfy \ref{itm_c_m_1_c}.
Moreover, by construction, we have $B_1\subset  P_{e_1}$ and $B_2\subset P_{e_2}$, and also $|B_1\cap [\rho^{j},\rho^{j+1})|=|B_2\cap [\rho^{j+i},\rho^{j+i+1})|$ for all $j\in\N$. Therefore, $B_1$ and $B_2$ also satisfy \ref{itm_c_m_1_a} and \ref{itm_c_m_1_b}.
\end{proof}

For the proof of part \ref{itm_c_m_3} of \cref{lem_coprimaility_measures_2} we need another lemma.

\begin{Lemma}
\label{lem_equiv_classes}
Suppose $\N=I_1 \cup\ldots \cup I_d$ is a partition of $\N$ and fix $e\in[d]$ for which $I_e$ is a DSR set.
Let $\mathcal{C}(e)\coloneqq \{e'\in[d]: e'\approx e\}$ denote the $\approx$-equivalence class of an element $e\in[d]$ and define
$$
I~\coloneqq~\bigcup_{e'\in \mathcal{C}(e)} I_{e'}.
$$
Then for every $\ell\in \N$ the set $\{n\in I: n+\ell\notin I\}$ is not a DSR set.
\end{Lemma}

\begin{proof}
We will show that $\{n\in I: n+1\notin I\}$ is not DSR; from this the claim follows for $1$ replaced by arbitrary $\ell\in\N$.
By way of contradiction, assume the set $\{n\in I: n+1\notin I\}$ is a DSR set. 
Since $\N\setminus I \subset \bigcup_{e'\notin \mathcal{C}(e)} I_{e'}$, it follows from the partition regularity that for some $e'\notin \mathcal{C}(e)$ the set $\{n\in I: n+1\in I_{e'}\}$ is DSR. But this implies $e'\approx e$, which contradicts $e'\notin \mathcal{C}(e)$. 
\end{proof}

\begin{proof}[Proof of \cref{lem_coprimaility_measures_2}, part \ref{itm_c_m_3}]
Fix $e_1\in[d]$ for which $I_{e_1}$ is a DSR set. Since $I_{e_1}$ is DSR and $|P_{e_1}\cap [\rho^j,\rho^{j+1})|\geq 1/d\, |\P\cap [\rho^j,\rho^{j+1})|$ for all $j\in I_{e_1}$, we have that
$$
\sum_{p\in P_{e_1}}\frac{1}{p}~=~\infty.
$$
Pick $s\in\N$ sufficiently large such that
\begin{equation}
\label{eqn_c_m_p2_1}
\sum_{p\in P_{e_1}\atop p< \rho^{s+1}}\frac{1}{p}~\geq~\frac{12 d^2}{\epsilon}.
\end{equation}
Define
$
I\coloneqq\bigcup_{e\in \mathcal{C}(e_1)} I_{e}
$
and consider the set
$
I'\coloneqq\{n\in I: n+1\in I,~ n+2\in I,\ldots,~n+s\in I\}.
$
In view of \cref{lem_equiv_classes}, the set
$$
I\setminus I' ~=~\bigcup_{\ell\in[s]}\{n\in I: n+\ell\notin I\}
$$ is not a DSR set. In particular, this means $I_{e_1}\setminus I'$ is not DSR, because $I_{e_1}\subset I$. It follows that $E\coloneqq I_{e_1}\cap I'$ must be a DSR set.

By definition, every element $n\in E$ has the property that for all $l\in[s]$ the number $n+l$ belongs to $I_e$ for some $e$ in the equivalence class of $e_1$. In other words, if we define the set $K_{n,e}\subset [s]$ via
$$
l\in K_{n,e}~\iff~ n+l\in I_e,
$$ 
then, as $e$ runs through $\mathcal{C}(e_1)$, the sets $K_{n,e}$ exhaust $[s]$, i.e.,
$$
[s]\,=\,\bigcup_{e\in \mathcal{C}(e_1)} K_{n,e},\qquad\text{for all}~n\in E.
$$
This allows us to write
\[
\sum_{p\in P_{e_1}\atop p< \rho^{s+1}}\frac{1}{p}~=~
\sum_{e\in \mathcal{C}(e_1)} ~\sum_{l\in K_{n,e}} ~ \sum_{p\in P_{e_1}\cap [\rho^l,\rho^{l+1})} \frac{1}{p}.
\]
Note that $\mathcal{C}(e_1)$ contains at most $d$-many elements. Therefore, by the pigeonhole principle, it follows from \eqref{eqn_c_m_p2_1} that for some $e^{(n)}\in \mathcal{C}(e_1)$ we have
\begin{equation}
\label{eqn_c_m_p2_3_0}
\sum_{l\in K_{n,e^{(n)}}} ~ \sum_{p\in P_{e_1}\cap [\rho^l,\rho^{l+1})} \frac{1}{p}~\geq~\frac{12 d }{\epsilon}.
\end{equation}
In summary, we have found for every $n\in E$ an element $e^{(n)}\in \mathcal{C}(e_1)$ and a subset $K_{n,e^{(n)}}\subset[s]$ such that $n+l\in I_{e^{(n)}}$ for every $l\in K_{n,e^{(n)}}$ and \eqref{eqn_c_m_p2_3_0} holds.
There are only finitely many choices for both $e^{(n)}$ and $K_{n,e^{(n)}}$.
Therefore there exists $e_2\in \mathcal{C}(e_1)$, $K\subset [s]$, and $E'\subset E$ such that $E'$ is still a DSR set and $e^{(n)}=e$ and $K_{n,e^{(n)}}=K$ for all $n\in E'$.

From \eqref{eqn_c_m_p2_3_0} it follows that
\[
\sum_{l\in K} ~ \sum_{p\in P_{e_1}\cap [\rho^l,\rho^{l+1})} \frac{1}{p}~\geq~\frac{12 d }{\epsilon}.
\]
We can write the left hand side of the above inequality as
$$
\sum_{l\in K} ~ \sum_{p\in P_{e_1}\cap [\rho^l,\rho^{l+1})} \frac{1}{p}~=~\sum_{i\in\{0,1,\ldots,4d-1\}}~ \sum_{l\in K} ~ \sum_{p\in P_{e_1}\cap [\rho^{l+i/(4d)},\rho^{l+(i+1)/(4d)})} \frac{1}{p}.
$$
It follows that there exists $i_1\in\{0,1, \ldots,4d-1\}$ such that
\begin{equation*}
\sum_{l\in K} ~ \sum_{p\in P_{e_1}\cap [\rho^{l+i_1/(4d)},\rho^{l+(i+1)/(4d)})} \frac{1}{p}~\geq~\frac{3 }{\epsilon}.
\end{equation*}
Define $P'\coloneqq \bigcup_{l\in K} P_{e_1}\cap [\rho^{l+i_1/(4d)},\rho^{l+(i_1+1)/(4d)})$. Then $P'\subset P_{e_1}$ and
\begin{equation}
\label{eqn_c_m_p2_3}
\sum_{p\in P'} \frac{1}{p}~\geq~\frac{3 }{\epsilon}.
\end{equation}

Next, observe that
$j+l\in I_{e_2}$ for all $j\in E'$ and $l\in K$. Since $E'\subset I_{e_1}$, we have $|P_{e_1}\cap [\rho^j,\rho^{j+1})|\geq 1/d\, |\P\cap [\rho^j,\rho^{j+1})|$ for all $j\in E'$. Moreover, by the Prime Number Theorem, for all but finitely many $j\in E'$ we have
$$
\big|P_{e_1}\cap \big[\rho^{j+i/(4d)},\rho^{j+(i+1)/(4d)}\big)\big|~\leq~\frac{1}{2d}\,\big|\P\cap \big[\rho^{j},\rho^{j+1}\big)\big|.
$$
Therefore, for all but finitely many $j\in E'$, if we split the set $P_{e_1}\cap [\rho^j,\rho^{j+1})$ into $4d$ many pieces,  
$$
P_{e_1}\cap \big[\rho^j,\rho^{j+1}\big)~=~\bigcup_{i\in\{0,1,\ldots,4d-1\}} P_{e_1}\cap \big[\rho^{j+i/(4d)},\rho^{j+(i+1)/(4d)}\big),
$$
then at least two of the pieces will have size at least $ \frac{1}{8d}|\P\cap [\rho^{j},\rho^{j+1})|$, that is, there exist $u_j,v_j\in \{0,1,\ldots,4d-1\}$ such that
$$
\big|P_{e_1}\cap \big[\rho^{j+u_j/(4d)},\rho^{j+(u_j+1)/(4d)}\big)\big| ~\geq~\frac{1}{8d}\,\big|\P\cap \big[\rho^{j},\rho^{j+1}\big)\big|
$$
and
$$
\big|P_{e_1}\cap \big[\rho^{j+v_j/(4d)},\rho^{j+(v_j+1)/(4d)}\big)\big| ~\geq~\frac{1}{8d}\,\big|\P\cap \big[\rho^{j},\rho^{j+1}\big)\big|.
$$
Since there are only finitely many choices for $u_j$ and $v_j$, we can pass to a subset $E''\subset E'$ which is still DSR and such that $v_j=v$ and $u_j=u$ for all $j\in E''$, where $u,v\in\{0,1,\ldots,4d-1\}$ are fixed. By further refining $E''$ if necessary, we can also assume without loss of generality that $E''$ intersects any interval of length $s$ in at most one point and, additionally,
\begin{equation}
\label{eqn_c_m_p2_7}
\big|\P\cap \big[\rho^{j+l},\rho^{j+l+1}\big)\big|~\geq~ \frac{1}{8}\,\big|\P\cap \big[\rho^{j},\rho^{j+1}\big)\big|,\qquad\forall l\in K,~\forall j\in E''.
\end{equation}

Since $u$ and $v$ are distinct, we either have $u+i_1\neq 4d-1$ or $v+i_1\neq 4d-1$. If the former holds, then define $i_2\coloneqq u$, otherwise define $i_2\coloneqq v$. Either way, we have
$$
\big|P_{e_1}\cap \big[\rho^{j+i_2/(4d)},\rho^{j+(i_2+1)/(4d)}\big)\big| ~\geq~\frac{1}{8d}\,\big|\P\cap \big[\rho^{j},\rho^{j+1}\big)\big|
$$
for all $j\in E''$. This allows us to find $Q_j\subset P_{e_1}\cap [\rho^{j+i_2/(4d)},\rho^{j+(i_2+1)/(4d)})$ of size
\begin{equation}
\label{eqn_c_m_p2_8}
\frac{\big|\P\cap \big[\rho^{j},\rho^{j+1}\big)\big|}{16d |P'|}~\leq~|Q_j|~\leq~\frac{\big|\P\cap \big[\rho^{j},\rho^{j+1}\big)\big|}{8d |P'|}.
\end{equation}
Since $E''$ is a DSR set, it follows from \eqref{eqn_c_m_p2_8}, combined with he prime number theorem, that the set $\bigcup_{j\in E''} Q_j$ is DSR. This implies that there exists $t\in\N$ such that the set $P''\coloneqq \bigcup_{j\in E''\cap[t]} Q_j$ satisfies
\begin{equation}
\label{eqn_c_m_p2_4}
\sum_{p\in P''} \frac{1}{p}~\geq~\frac{3 }{\epsilon}.
\end{equation}

Define $i\coloneqq \lfloor (i_1+i_2)/4d \rfloor$. It is straightforward to calculate that
$$
P'\cdot P''~\subset ~ \bigcup_{j\in E''}\bigcup_{l\in K} \big[\rho^{j+l+(i_1+i_2)/(4d)},\rho^{j+l+(i_1+i_2+2)/(4d)}\big).
$$
Since $i_1+i_2\neq 4d-1$, we either have $(i_1+i_2+2)/(4d)\leq 1$ or $(i_1+i_2)/(4d)\geq 1$. Therefore
$$
\big[\rho^{j+l+(i_1+i_2)/(4d)},\rho^{j+l+(i_1+i_2+2)/(4d)}\big)~\subset ~\big[\rho^{j+l+i},\rho^{j+l+i+1}\big),
$$
which implies
$$
P'\cdot P''~\subset ~ \bigcup_{j\in E''}\bigcup_{l\in K} \big[\rho^{j+l+i},\rho^{j+l+i+1}\big).
$$
Moreover, since $E''$ intersects all shifts of $K$ in at most one point, we have
\begin{equation}
\label{eqn_c_m_p2_9}
(P'\cdot P'')\cap \big[\rho^{j+l+i},\rho^{j+l+i+1}\big)~= ~ Q_j \cdot\left(P_{e_1}\cap \big[\rho^{l+i_1/(4d)},\rho^{l+(i_1+1)/(4d)}\big)\right).
\end{equation}

We are now ready to construct the sets $B_1$ and $B_2$ satisfying \ref{itm_c_m_3_a}, \ref{itm_c_m_3_b}, and \ref{itm_c_m_3_c}. Take $B_1\coloneqq P'\cdot P''$. Since $P'\subset P_{e_1}$ and $P''\subset P_{e_1}$, we have that $B_1\subset \{pq: p,q\in P_{e_1}\}$. Moreover, using 
\cref{lem_coprimaility_measurement_2almostprimes} together with \eqref{eqn_c_m_p2_3} and \eqref{eqn_c_m_p2_4} proves that $\mathbb{E}^{\log}_{m\in B_1}\mathbb{E}^{\log}_{n\in B_1}\Phi(m,n)\leq \epsilon$.

Next, note that from \eqref{eqn_c_m_p2_7}, \eqref{eqn_c_m_p2_8}, and \eqref{eqn_c_m_p2_9} it follows that for every $j\in E''\cap [t]$ and every $l\in K$ we have
$$
B_1\cap \big[\rho^{j+l+i},\rho^{j+l+i+1}\big)~\leq~ \frac{1}{d}\big|\P\cap \big[\rho^{j+l},\rho^{j+l+1}\big)\big|.
$$
Moreover, since $j+l\in I_{e_2}$, we have
$$
P_{e_2}\cap \big[\rho^{j+l},\rho^{j+l+1}\big)~\geq~ \frac{1}{d}\big|\P\cap \big[\rho^{j+l},\rho^{j+l+1}\big)\big|.
$$
Therefore, for all $j\in E''\cap [t]$ and $l\in K$, we can find $P_{2,j,l}\subset P_{e_2}\cap \big[\rho^{j+l},\rho^{j+l+1}\big)$ with
$$
|P_{2,j,l}|~=~\big|B_1\cap \big[\rho^{j+l+i},\rho^{j+l+i+1}\big)\big|.
$$
Define $P_2\coloneqq \bigcup_{j\in E''\cap[t]}\bigcup_{l\in K} P_{2,j,l}$. Then $B_2\subset P_{e_2}$ and also $|B_1\cap [\rho^{j},\rho^{j+1})|=|B_2\cap [\rho^{j+i},\rho^{j+i+1})|$ for all $j\in\N$. Moreover, using $|B_1\cap [\rho^{j},\rho^{j+1})|=|B_2\cap [\rho^{j+i},\rho^{j+i+1})|$, we see that
$$
\sum_{m\in B_2}\frac{1}{m}~\geq~\frac{1}{\rho} \sum_{m\in B_1}\frac{1}{m} ~=~\frac{1}{\rho}\left(\sum_{p\in P'}\frac{1}{p}\right)\left(\sum_{q\in P''}\frac{1}{q}\right)~\geq~ \frac{9}{\rho \epsilon^2}~\geq~ \frac{1}{\epsilon}.
$$
In light of \cref{lem_coprimaility_measurement_primes}, this shows that $\mathbb{E}^{\log}_{m\in B_2}\mathbb{E}^{\log}_{n\in B_2}\Phi(m,n)\leq \epsilon$. This finishes the proof that $B_1$ and $B_2$ satisfy \ref{itm_c_m_3_a}, \ref{itm_c_m_3_b}, and \ref{itm_c_m_3_c}.
\end{proof}

Here is another lemma which we use in our proof of \cref{prop_idempotency}.

\begin{Lemma}
\label{lem_coprimaility_measures_3}
Fix $\epsilon\in (0,1)$ and $\rho\in (1,1+\epsilon]$ and $i\in\{-1,0,1\}$. Let $(Y,S)$ be a finitely generated multiplicative topological dynamical system, let $R_1,\ldots, R_d$ denote the generators of $S$, and define $N_e^k\coloneqq \{n\in\N: S_n=R_e^k\}$. Suppose there exist $k_1,k_2\in \{1,2\}$,  $e_1,e_2\in[d]$, and finite and non-empty sets $B_1,B_2\subset\N$ such that
\begin{enumerate}	
[label=~(\alph{enumi}), ref=(\alph{enumi}),leftmargin=*]
\item\label{itm_c_m_a}
$B_1\subset N_{e_1}^{k_1}$ and $B_2\subset N_{e_2}^{k_2}$;
\item\label{itm_c_m_b}
$|B_1\cap [\rho^{j},\rho^{j+1})|=|B_2\cap [\rho^{j+i},\rho^{j+i+1})|$ for all $j\in\N$;
\item\label{itm_c_m_c}
$\mathbb{E}^{\log}_{m\in B_1}\mathbb{E}^{\log}_{n\in B_1} \Phi(m,n)\leq \epsilon$ and $\mathbb{E}^{\log}_{m\in B_2}\mathbb{E}^{\log}_{n\in B_2} \Phi(m,n)\leq \epsilon$.
\end{enumerate} 
Then for every $y\in Y$ and every $g\in\Cont(Y)$ with $\sup_{y\in Y}|g(y)|\leq 1$ we have
$$
\limsup_{N\to\infty}\,\left|\BEu{n\in[N]} g(R_{e_1}^{k_1}S_n y)~-~\BEu{n\in[N]} g(R_{e_2}^{k_2}S_n y) \right| ~\leq~17\epsilon^{1/2}.
$$
\end{Lemma}

For the proof of \cref{lem_coprimaility_measures_3} we need a variant of \cref{lem_hilfslemma_1}.
\begin{Lemma}
\label{lem_hilfslemma_2}
Fix $\epsilon\in(0,1)$, $\rho\in(1,1+\epsilon]$, and $i\in\{-1,0,1\}$. 
Let $B_1$ and $B_2$ be finite non-empty subsets of $\N$ with the property that $|B_1\cap [\rho^j,\rho^{j+1})|=|B_2\cap [\rho^{j+i},\rho^{j+i+1})|$ for all $j\in\N\cup\{0\}$. Then for any $a\colon\N\to \C$ with $|a|\leq 1$ we have 
\begin{equation*}
\left|\,\BEul{p\in B_{1}} \, \BEu{n\in [\nicefrac{N}{p}]} a(n) ~-~ \BEul{q\in B_{2}} \, \BEu{n\in [\nicefrac{N}{q}]} a(n)\,\right| ~\leq~ 15\epsilon.
\end{equation*}
\end{Lemma}
The proof of \cref{lem_hilfslemma_2} is analogous to the proof of \cref{lem_hilfslemma_1} and therefore omitted.

\begin{proof}[Proof of \cref{lem_coprimaility_measures_3}]
First, in light of \cref{prop_coprimality_criterion}, it follows from assumption \ref{itm_c_m_c} that
\begin{eqnarray*}
\limsup_{N\to\infty}\,\left|\,\BEu{n\in[N]} g(R_{e_2}^{k_2}S_n y)~-~  \BEul{p\in B_{1}} \, \BEu{n\in [\nicefrac{N}{p}]} g(R_{e_2}^{k_2}S_{pn} y)\,\right| &\leq& \epsilon^{1/2},
\\
\limsup_{N\to\infty}\,\left|\,\BEu{n\in[N]} g(R_{e_1}^{k_1}S_n y)~-~  \BEul{q\in B_{2}} \, \BEu{n\in [\nicefrac{N}{q}]} g(R_{e_1}^{k_1}S_{qn} y)\,\right| &\leq& \epsilon^{1/2}.
\end{eqnarray*}
Note that $g(R_{e_2}^{k_2}S_{pn} y)=g(R_{e_1}^{k_1}S_{qn} y)=g(R_{e_1}^{k_1}R_{e_2}^{k_2}S_{n} y)$ for all $n\in\N$, $p\in B_1$, and $q\in B_2$, because $B_1\subset N_{e_1}^{k_1}$ and $B_2\subset N_{e_2}^{k_2}$. Therefore
\begin{eqnarray*}
\limsup_{N\to\infty}\,\left|\,\BEu{n\in[N]} g(R_{e_2}^{k_2}S_n y)~-~  \BEul{p\in B_{1}} \, \BEu{n\in [\nicefrac{N}{p}]} g(R_{e_1}^{k_1}R_{e_2}^{k_2}S_{n} y)\,\right| &\leq& \epsilon^{1/2}
\\
\limsup_{N\to\infty}\,\left|\,\BEu{n\in[N]} g(R_{e_1}^{k_1}S_n y)~-~  \BEul{q\in B_{2}} \, \BEu{n\in [\nicefrac{N}{q}]} g(R_{e_1}^{k_1}R_{e_2}^{k_2}S_{n} y)\,\right| &\leq& \epsilon^{1/2}.
\end{eqnarray*}
The claim now follows from \cref{lem_hilfslemma_2}.
\end{proof}

\begin{proof}[Proof of \cref{prop_idempotency}]
Let $\epsilon\in(0,1)$ and $\rho\in (1,1+\epsilon]$ be arbitrary and define $N_e^k\coloneqq \{n\in\N: S_n=R_e^k\}$, $P_e\coloneqq \{p\in\P: S_p=R_e\}$, and 
\[
I_e \,\coloneqq\, \left\{j\in\N: \big|P_e\cap \big[\rho^j,\rho^{j+1}\big)\big|\,\geq\, \frac{1}{d}\big|\P\cap \big[\rho^j,\rho^{j+1}\big)\big|\right\}.
\]
Take $e\in[d]$ such that $\sum_{p\in P_e}1/p=\infty$. Combining part \ref{itm_c_m_1} of \cref{lem_coprimaility_measures_2} with \cref{lem_coprimaility_measures_3} we can find $e'\in[d]$ such that $I_{e'}$ is DSR and
\begin{equation}
\label{eqn_idempotency_2a}
\limsup_{N\to\infty}\,\left|\BEu{n\in[N]} g(R_{e}S_n y)~-~\BEu{n\in[N]} g(R_{e'}S_n y) \right| ~\leq~17\epsilon^{1/2}
\end{equation}
uniformly over all $y\in Y$ and $g\in\Cont(Y)$ with $\sup_{y\in Y}|g(y)|\leq 1$. Since \eqref{eqn_idempotency_2a} holds uniformly over all $g\in\Cont(Y)$ with $\sup_{y\in Y}|g(y)|\leq 1$, we can replace $g$ with $g\circ R_e$ and $g\circ R_{e'}$ and deduce that
\begin{equation}
\label{eqn_idempotency_2b}
\limsup_{N\to\infty}\,\left|\BEu{n\in[N]} g(R_{e}^2S_n y)~-~\BEu{n\in[N]} g(R_{e'}^2S_n y) \right| ~\leq~34\epsilon^{1/2}.
\end{equation}

Next, from part \ref{itm_c_m_3} of \cref{lem_coprimaility_measures_2} and \cref{lem_coprimaility_measures_3}, we know there exists $e''\in[d]$ with $e'\approx e''$ such that
\begin{equation}
\label{eqn_idempotency_3}
\limsup_{N\to\infty}\,\left|\BEu{n\in[N]} g(R_{e'}^2S_n y)~-~\BEu{n\in[N]} g(R_{e''}S_n y) \right| ~\leq~17\epsilon^{1/2}.
\end{equation}
Since $e'\approx e''$, there exist $r\in [d]$ and $e_1,\ldots,e_r\in[d]$ such that $e'=e_1$, $e''=e_r$, and $e_{i}\sim e_{i+1}$ for all $i\in[r-1]$. It then follows from part \ref{itm_c_m_2} of \cref{lem_coprimaility_measures_2} and \cref{lem_coprimaility_measures_3} that for all $i\in[r-1]$,
\begin{equation}
\label{eqn_idempotency_4}
\limsup_{N\to\infty}\,\left|\BEu{n\in[N]} g(R_{e_i}S_n y)~-~\BEu{n\in[N]} g(R_{e_{i+1}}S_n y) \right| ~\leq~17\epsilon^{1/2}.
\end{equation}
Using \eqref{eqn_idempotency_4} and the fact that $e
_1=e'$ and $e_r=e''$ we obtain
\begin{equation}
\label{eqn_idempotency_5}
\limsup_{N\to\infty}\,\left|\BEu{n\in[N]} g(R_{e'}S_n y)~-~\BEu{n\in[N]} g(R_{e''}S_n y) \right| ~\leq~17(r-1)\epsilon^{1/2}.
\end{equation}
From \eqref{eqn_idempotency_3} and \eqref{eqn_idempotency_5} it now follows that
\begin{equation}
\label{eqn_idempotency_6}
\limsup_{N\to\infty}\,\left|\BEu{n\in[N]} g(R_{e'}^2S_n y)~-~\BEu{n\in[N]} g(R_{e'}S_n y) \right| ~\leq~17r\epsilon^{1/2}.
\end{equation}
Finally, combining \eqref{eqn_idempotency_6} with \eqref{eqn_idempotency_2a} and \eqref{eqn_idempotency_2b} we get
\begin{equation*}
\limsup_{N\to\infty}\,\left|\BEu{n\in[N]} g(R_{e}^2S_n y)~-~\BEu{n\in[N]} g(R_{e}S_n y) \right| ~\leq~17(r+3)\epsilon^{1/2}.
\end{equation*}
Since $\epsilon$ was arbitrary, this proves \eqref{eqn_idempotency_1}.
\end{proof}

\section{Applications of \cref{thm_dynamical_MVT_Omega}}
\label{sec_appl_thmA}

The purpose of this section is to give proofs of Corollaries \ref{cor_polynomoial_ud_Omega}, \ref{cor_gneralized_polynomoial_ud_Omega}, \ref{cor_odious_evil_Omega}, \ref{cor_dynamical_MVT_squarefree}, \ref{cor_dynamical_MVT_omega}.

\begin{proof}[Proof of \cref{cor_polynomoial_ud_Omega}]
Let $Q(t)=c_k t^k+\ldots c_1 t+c_0$ be a real polynomial and assume that at least one of the coefficients $c_1,\ldots,c_k$ is irrational.
Following Furstenberg's method (see \cite[pp.\ 68--69]{Furstenberg81a}) one can write, for every $h\in\Z$, the sequence $e(hQ(n))$, $n\in\N$, in the form $f(T^n x)$, $n\in\N$, using the unipotent affine transformation $T\colon \T^k\to\T^k$ defined by
$$
T(x_1,\ldots,x_k)\,=\, (x_1+c_k, x_2+x_1, x_3+x_2,\ldots, x_k+x_{k-1}),\qquad\forall (x_1,\ldots,x_k)\in\T^k.
$$
Indeed, define
$
p_k(t)=Q(t)
$
and, for $i=k-1,\ldots,1$, define inductively the polynomial $p_i$ as
$$
p_{i}(t)\,\coloneqq\, p_{i+1}(t+1)-p_{i+1}(t).
$$
Also, let $x$ denote the point $(p_1(0),\ldots,p_k(0))$ in $\T^k$. One can verify that the orbit of the point $x$ under $T$ is the sequence $(p_1(n),\ldots, p_k(n))$.
In particular, if $h\in\Z$ and $f\colon \T^k \to \C$ denotes the function $f(x_1,\ldots,x_k)= e(h x_k)$, then we have
$$
f(T^nx)\,=\,e(h Q(n)),\qquad n\in\N,
$$
as well as
$$
f(T^{\Omega(n)}x)\,=\,e(h Q(\Omega(n))),\qquad n\in\N.
$$

Next, let $X$ be the orbit closure of $x$ under $T$. Since $(X,T)$ is a transitive system, it is also uniquely ergodic, because all transitive unipotent affine transformations are uniquely ergodic (one way of seeing this is to note that any unipotent affine transformation is a niltranslation and for niltranslations this is a well established fact, see \cref{prop_dynamics-nilrotation} below).  Since one of the coefficients $c_1,\ldots,c_k$ is irrational, we have $\lim_{N\to\infty}\mathbb{E}_{n\in [N]} e(hQ(n))=0$ as long as $h$ is non-zero. This implies that the integral of $ f$ with respect to the unique $T$-invariant Borel probability measure on $X$ equals $0$. Therefore, by \cref{thm_dynamical_MVT_Omega}, we have
$$
\lim_{N\to\infty}\,\BEu{n\in[N]} f(T^{\Omega(n)}x)~=~\lim_{N\to\infty}\,\BEu{n\in[N]} e(h Q(\Omega(n)))~=~0
$$
for every non-zero $h\in\Z$. By Weyl's equidistribution criterion, this implies that $Q(\Omega(n))$, $n\in\N$, is uniformly distributed mod $1$.
\end{proof}

\begin{proof}[Proof of \cref{cor_gneralized_polynomoial_ud_Omega}]
Let $Q\colon \N\to\R$ be a generalized polynomial. By Weyl's equidistribution criterion, to prove the equivalence between the uniform distribution mod $1$ of the sequences $Q(n)$ and $Q(\Omega(n))$, it suffices to show that for every $h\in\Z$ we have
\begin{equation}
\label{eqn_generalized_poly_proof_1}
\lim_{N\to\infty}\,\BEu{n\in[N]} e(hQ(n))~=~\lim_{N\to\infty}\,\BEu{n\in[N]} e(h Q(\Omega(n))).
\end{equation}
Let $h\in\Z$ be arbitrary. According to \cite[Theorem A]{BL07} there exists a nilmanifold $X=G/\Gamma$, a niltranslation $T\colon X\to X$, a point $x\in X$, and a Riemann integrable function $\tilde{F}\colon X\to[0,1)$ such that
$$
\{Q(n)\}\,=\,\tilde{F}(T^n x),\qquad \forall n\in\N,
$$
where $\{Q(n)\}$ is the fractional part of $Q(n)$.
By replacing, if needed, $X$ with $\overline{\{T^nx: n\in\Z\}}$, we can assume without loss of generality that the orbit of $x$ under $T$ is dense in $X$. In this case the nilsystem $(X,T)$ is uniquely ergodic (see \cref{prop_dynamics-nilrotation} below).

Now define the function $F\colon X\to \C$ as $F(y)=e(h \tilde{F}(y))$ for all $y\in X$. This allows us to rewrite \eqref{eqn_generalized_poly_proof_1} as
\begin{equation}
\label{eqn_generalized_poly_proof_2}
\lim_{N\to\infty}\,\BEu{n\in[N]} F(T^n x)~=~\lim_{N\to\infty}\,\BEu{n\in[N]} F(T^{\Omega(n)} x).
\end{equation}
Since $(X,T)$ is uniquely ergodic, it follows from \cref{thm_dynamical_MVT_Omega} that \eqref{eqn_generalized_poly_proof_2} holds when $F$ is replaced by any continuous function. But if it holds for all continuous functions, then it also holds for Riemann integrable fucntions. Since $F$ is Riemann integrable, we conclude that \eqref{eqn_generalized_poly_proof_2} is true. 
\end{proof}

\begin{proof}[Proof of \cref{cor_odious_evil_Omega}]
Fix $q\geq 2$ and let $m\in\N$ with $\gcd(q-1,m)=1$.  Consider the sequence
$$
x_n~\coloneqq~ e(c s_q(n)),
$$
where $s_q$ denotes the sum of digits of $n$ in base $q$ and $c\in \R$. We can view $x=(x_n)$ as an element in the space $\mathbb{D}^\N$, where $\mathbb{D}\coloneqq\{z\in\C: |z|\leq 1\}$. Let $T$ be the left-shift on $\mathbb{D}^\N$, i.e., $T$ is the map that takes a sequence $(y_n)_{n\in\N}$ to the sequence $(y_{n+1})_{n\in\N}$. Let $X\subset\mathbb{D}^\N$ denote the orbit closure of $x$ under the transformation $T$,
$$
X=\overline{\{T^n x: n\in\N\cup\{0\}\}}.
$$
Then $(X,T)$ is an additive topological dynamical system. It is known (cf.\ \cite[p. 122]{Queffelec10}) that if $c(q-1)$ is not an integer, then $(X,T)$ is uniquely ergodic. Therefore, in light of \cref{thm_dynamical_MVT_Omega}, we have for any $f\in\Cont(X)$,
\begin{equation}
\label{eqn_odious_evil_Omega_1}
\lim_{N\to\infty}\,\BEu{n\in[N]} f(T^{\Omega(n)}x)~=~
\lim_{N\to\infty}\,\BEu{n\in[N]} f(T^{n}x).
\end{equation}
Note that if $f\colon\mathbb{D}^\N\to\mathbb{D}$ is the function that maps a sequence $(y_n)$ in $\mathbb{D}^\N$ onto its first coordinate $y_1$, then
$$
f(T^n x)~=~e(c s_q(n)),\qquad\text{and}\qquad f(T^{\Omega(n)} x)~=~e(c s_q(\Omega(n))).
$$
Therefore \eqref{eqn_odious_evil_Omega_1} implies
\begin{equation*}
\lim_{N\to\infty}\,\BEu{n\in[N]} e(c s_q(\Omega(n)))~=~
\lim_{N\to\infty}\,\BEu{n\in[N]} e(c s_q(n))
\end{equation*}
for all $c$ with $c(q-1)\notin\Z$. In particular, if we take $c=\frac{r}{m}$ for some $r\in\{1,\ldots,m-1\}$, then
\begin{equation}
\label{eqn_odious_evil_Omega_2}
\lim_{N\to\infty}\,\BEu{n\in[N]} e\left(\frac{r}{m} s_q(\Omega(n))\right)~=~
\lim_{N\to\infty}\,\BEu{n\in[N]} e\left(\frac{r}{m} s_q(n)\right).
\end{equation}
As we have mentioned in \cref{sec_DG_PNT}, Gelfond \cite{Gelfond68} showed that if $m$ and $q-1$ are coprime then for all $r\in\{0,1,\ldots,m-1\}$ the set of $n$ for which $s_q(n)\equiv {r}\bmod{m}$ has asymptotic density $1/m$. This is equivalent to 
\begin{equation*}
\lim_{N\to\infty}\,\BEu{n\in[N]} e\left(\frac{r}{m} s_q(n)\right)~=~0,\qquad\forall r\in\{1,\ldots,m-1\}.
\end{equation*}
By \eqref{eqn_odious_evil_Omega_2}, this implies 
\begin{equation*}
\lim_{N\to\infty}\,\BEu{n\in[N]} e\left(\frac{r}{m} s_q(\Omega(n))\right)~=~0,\qquad\forall r\in\{1,\ldots,m-1\},
\end{equation*}
which shows that the set of $n$ for which $s_q(\Omega(n))\equiv {r}\bmod{m}$ has asymptotic density $1/m$.
\end{proof}

\begin{proof}[Proof of \cref{cor_dynamical_MVT_squarefree}]
Let $(X,T)$ be a uniquely ergodic additive topological dynamical system and let $\mu$ denote the corresponding unique $T$-invariant Borel probability measure on $X$. Our goal is to prove
\begin{equation}
\label{eqn_ud_of_squarefree_1}
\lim_{N\to\infty}\,\frac{1}{N}\hspace{-.6 em}\sum_{1\leq n\leq N\atop n\, \mathrm{squarefree}} \hspace{-.6 em} f\big(T^{\Omega(n)}x\big)
~=~\frac{6}{\pi^2} \left(\int f\d\mu\right)
\end{equation}
for all $f\in\Cont(X)$ and $x\in X$. 
Let $\1_{\square\text{-free}}(n)$ denote the indicator function of the set of squarefree numbers. Then \eqref{eqn_ud_of_squarefree_1} can be rewritten as
\begin{equation}
\label{eqn_ud_of_squarefree_2}
\lim_{N\to\infty}\, \BEu{n\in[N]}\1_{\square\text{-free}}(n)\, f\big(T^{\Omega(n)}x\big)
~=~\frac{6}{\pi^2} \left(\int f\d\mu\right).
\end{equation}
Recall that $\mob\colon\N\to\{-1,0,1\}$ denotes the \Mobius{} function.
To prove \eqref{eqn_ud_of_squarefree_2}, we utilize the well-known fact that $\sum_{d^2\mid n} \mob(d)=\1_{\square\text{-free}}(n)$, which implies
\[
\1_{\square\text{-free}}(n)\, f\big(T^{\Omega(n)}x\big)~=~\sum_{d^2\mid n} \mob(d)f\big(T^{\Omega(n)}x\big)~=~\sum_{d}\1_{d^2\mid n}\, \mob(d)f\big(T^{\Omega(n)}x\big).
\]
Therefore, and since $\mathbb{E}_{n\in[N]} \1_{d^2\mid n}\, f(T^{\Omega(n)}x) = \frac{1}{d^2} \mathbb{E}_{n\in[N/d^2]} f(T^{\Omega(d^2 n)}x)+\oh_{N\to\infty}(1)$, we have
$$
\lim_{N\to\infty}\, \BEu{n\in[N]}\1_{\square\text{-free}}(n)\, f\big(T^{\Omega(n)}x\big)\, =\,
 \sum_{d=1}^D \frac{\mob(d)}{d^2}\left(\lim_{N\to\infty}\,\BEu{n\in[\nicefrac{N}{d^2}]} f\big(T^{\Omega(d^2n)}x\big)\right)+\oh_{D\to\infty}(1).
$$
By \cref{thm_dynamical_MVT_Omega}, we have
$$
\lim_{N\to\infty}\,\BEu{n\in[\nicefrac{N}{d}]} f\big(T^{\Omega(d^2n)}x\big)~=~\int f\d\mu
$$
for all $d\in\N$, which leaves us with
\[
\lim_{N\to\infty}\, \BEu{n\in[N]}\1_{\square\text{-free}}(n)\, f\big(T^{\Omega(n)}x\big)\, =\,
 \sum_{d=1}^D \frac{\mob(d)}{d^2}\left(\int f\d\mu\right)\,+\,\oh_{D\to\infty}(1).
\]
Since $\sum_d {\mob(d)}/{d^2}={6}/{\pi^2}$, this proves \eqref{eqn_ud_of_squarefree_2}.
\end{proof}

Recall that $a\colon\N\to\N$ is \define{additive} if $a(nm)=a(n)+a(m)$ for all $n,m\in\N$ with $\gcd(n,m)=1$ and \define{completely additive} if $a(nm)=a(n)+a(m)$ for all $n,m\in\N$. The following lemma will be useful for the proofs of \cref{cor_dynamical_MVT_omega} and \cref{cor_dynamical_MVT_additive_functions}. 

\begin{Lemma}
\label{lem_additive_fctns_transferrence}
Suppose $\mathcal{P}$ is a subset of $\P$ with the property that $\sum_{p\in\mathcal{P}}1/p<\infty$.
Let $a_1\colon\N\to\N$ be a completely additive function, $a_2\colon\N\to\N$ an additive function, and assume $a_1(p)=a_2(p)$ for all $p\in\P\setminus\mathcal{P}$. Let $\mathcal{H}$ be a shift-invariant\footnote{
A collection of functions $\mathcal{H}$ from $\N$ to $\C$ is called \define{shift-invariant} if for any $h\in\mathcal{H}$ and any $t\in \N$ the function $n\mapsto h(n+t)$ belongs to $\mathcal{H}$.} collection of bounded functions from $\N$ to $\C$. The following two statements hold.
\begin{enumerate}	
[label=(\roman*),ref=(\roman*),leftmargin=*]
\item
If $\lim_{N\to\infty}\E_{n\in[N]} h(a_1(n))$ exists for all $h\in\mathcal{H}$, then $\lim_{N\to\infty}\E_{n\in[N]} h(a_2(n))$ also exists for all $h\in\mathcal{H}$; \item
If $\lim_{N\to\infty}\E_{n\in[N]} h(a_1(n))=L$ for all $h\in\mathcal{H}$, then $\lim_{N\to\infty}\E_{n\in[N]} h(a_2(n))=L$ for all $h\in\mathcal{H}$.
\end{enumerate}
\end{Lemma}

\begin{proof}
Let $B$ denote the set of integers whose prime factors all belong to $\mathcal{P}$. Let $C$ denote the set of integers whose prime factors all belong to $\P\setminus\mathcal{P}$ and have multiplicity exactly $1$.
Finally, let $D$ denote the set of integers whose prime factors all belong to $\P\setminus\mathcal{P}$ and have multiplicity $2$ or greater.
It is straightforward to check that any $n\in\N$ can be uniquely written as $n=bcd$ for $b\in B$, $c\in C$, $d\in D$, and $\gcd(c,d)=1$.
For convenience, let us write $C_d$ for the set of all $c\in C$ with $\gcd(c,d)=1$ and $\mathcal{P}_d=\{p\in \P: \text{either}~p\in\mathcal{P}~\text{or}~p\mid d\}$.
Observe that a natural number $n$ belongs to $C_d$ if and only if it is squarefree and not divisible by any $p\in \mathcal{P}_d$. In other words, 
\[
\1_{C_d}(n)=\biggl(\prod_{p\in\mathcal{P}_d}(1-\1_{p\mid n})\biggr)\biggl(\prod_{q\in \P\setminus \mathcal{P}_d}(1-\1_{q^2\mid n})\biggr).
\]
The above expression is well-defined because for any fixed $n$ the products involved are finite.
By expanding these products we get
\begin{align*}
\prod_{p\in\mathcal{P}_d}(1-\1_{p\mid n})
=
\sum_{k=0}^\infty ~(-1)^k \sum_{p_1<\ldots<p_k\in \mathcal{P}_d} 1_{p_1\cdot\ldots\cdot p_k\mid n}
\end{align*}
and 
\begin{align*}
\prod_{q\in \P\setminus \mathcal{P}_d}(1-\1_{q^2\mid n})
=
\sum_{j=0}^\infty ~(-1)^j \sum_{q_1<\ldots<q_j\in \P\setminus \mathcal{P}_d} 1_{q_1^2\cdot\ldots\cdot q_j^2\mid n},
\end{align*}
which gives us
\begin{align}
\label{eqn_descr_Cd}
\1_{C_d}(n)
&= 
\sum_{k,j=0}^\infty ~(-1)^{k+j}
\sum_{p_1<\ldots<p_k\in \mathcal{P}_d\atop q_1<\ldots<q_j\in \P\setminus \mathcal{P}_d} 1_{p_1\cdot\ldots\cdot p_kq_1^2\cdot\ldots\cdot q_j^2\mid n}.
\end{align}

Now, given $h\in\mathcal{H}$, we have
\begin{align*}
\BEu{n\in[N]} h(a_2(n))
&=
\frac{1}{N}\sum_{b\in B}\sum_{d\in D}\sum_{n\leq N/bd} h(a_2(bdn)) \1_{C_d}(n)
\\
&=
\sum_{b\in B}\frac{1}{b}\sum_{d\in D}\frac{1}{d}\left( \frac{1}{N/bd} \sum_{n=1}^{N/bd} 1_{C_d}(n) h_{bd}(a_2(n))\right),
\end{align*}
where $h_m(n)=h(a_2(m)+n)$ for all $m\in\N$.
Since $a_2(n)=a_1(n)$ for all $n\in C_d\subset C$, we obtain
\begin{align*}
\frac{1}{N}\sum_{n=1}^N h(a_2(n))
&=
\sum_{b\in B}\frac{1}{b}\sum_{d\in D}\frac{1}{d}\left( \frac{1}{N/bd} \sum_{n=1}^{N/bd} 1_{C_d}(n) h_{bd}(a_1(n))\right).
\end{align*}
Additionally, since $\sum_{b\in B}1/b$ and $\sum_{d\in D}1/d$ are both convergent sums, we have
\begin{align*}
\frac{1}{N}\sum_{n=1}^N h(a_2(n))
&=
\sum_{b\in B\atop b\leq y}\frac{1}{b} \sum_{d\in D\atop d\leq y}\frac{1}{d}\left( \frac{1}{N/bd} \sum_{n=1}^{N/bd} 1_{C_d}(n) h_{bd}(a_1(n))\right)+ \oh_{y\to\infty}(1).
\end{align*}
Using \eqref{eqn_descr_Cd}, we can further conclude that
\begin{align*}
\sum_{b\in B\atop b\leq y}& \frac{1}{b} \sum_{d\in D\atop d\leq y}\frac{1}{d}
\sum_{k,j=0}^\infty ~(-1)^{k+j}
\sum_{p_1<\ldots<p_k\in \mathcal{P}_d\atop q_1<\ldots<q_j\in \P\setminus \mathcal{P}_d}
\left( \frac{1}{N/bd} \sum_{n=1}^{N/bd} 1_{p_1\cdot\ldots\cdot p_kq_1^2\cdot\ldots\cdot q_j^2\mid n} h_{bd}(a_1(n))\right)+ \oh_{y\to\infty}(1)
\\
&=
\sum_{b\in B\atop b\leq y}\frac{1}{b} \sum_{d\in D\atop d\leq y}\frac{1}{d}
\sum_{k,j=0}^\infty ~(-1)^{k+j}
\\
&\qquad
\sum_{p_1<\ldots<p_k\in \mathcal{P}_d\atop q_1<\ldots<q_j\in \P\setminus \mathcal{P}_d}
\frac{1}{p_1\cdot\ldots\cdot q_j^2}
\left( \frac{1}{\frac{N}{bdp_1\cdot\ldots\cdot q_j^2}} \sum_{n=1}^{\frac{N}{bdp_1\cdot\ldots\cdot q_j^2}}  h_{bd}(a_1(p_1\cdot\ldots\cdot q_j^2n))\right)+ \oh_{y\to\infty}(1).
\end{align*}
Combining the above and taking the limit as $N\to\infty$, we get
\begin{align*}
\lim_{N\to\infty}\frac{1}{N} & 
\sum_{n=1}^{N} h(a_2(n))
=
\sum_{b\in B\atop b\leq y}\frac{1}{b} \sum_{d\in D\atop d\leq y}\frac{1}{d}
\sum_{k,j=0}^\infty ~(-1)^{k+j}
\\
&
\sum_{p_1<\ldots<p_k\in \mathcal{P}_d\atop q_1<\ldots<q_j\in \P\setminus \mathcal{P}_d}
\frac{1}{p_1\cdot\ldots\cdot q_j^2}
\left(\lim_{N\to\infty}\BEu{n\in[N]}  h_{bd,p_1\cdot\ldots\cdot q_j^2 }(a_1(n))\right)+ \oh_{y\to\infty}(1),
\end{align*}
where $h_{bd,p_1\cdot\ldots\cdot q_j^2 }(n)=h_{bd}(a_1(p_1\cdot\ldots\cdot q_j^2 )+n)$.
In particular, it follows that the limit on the left hand side of the above equation exists if all the limits on the right hand side exist. Likewise, the limit on the left hand side of the above equation equals $L$ if all the limits on the right hand side equal $L$. 
\end{proof}

\begin{proof}[Proof of \cref{cor_dynamical_MVT_omega}]
We want to show
\begin{equation}
\label{eqn_ud_of_omega_1}
\lim_{N\to\infty}\BEu{n\in[N]} f\big(T^{\omega(n)}x\big)
~=~\int f\d\mu.
\end{equation}
Since $\lim_{N\to\infty}\E_{n\in[N]} f\big(T^{\Omega(n)}x\big)=\int f\d\mu$ holds for all $f\in C(X)$ and $x\in X$ and $\Omega(p)=\omega(p)$ for all $p\in\P$, \eqref{eqn_ud_of_omega_1} follows from \cref{lem_additive_fctns_transferrence}  applied with $a_1(n)=\Omega(n)$, $a_2(n)=\omega(n)$, $\mathcal{P}=\emptyset$, and $\mathcal{H}=\{n\mapsto f(T^nx): f\in C(X),~x\in X\}$.
\end{proof}

\section{Applications of \cref{thm_dynamical_MVT_fg_sue}}
\label{sec_applications_of_thm_B}

This section is dedicated to the proofs of Corollaries~\ref{cor_ortho_Omega_rational_phases}, \ref{cor_dynamical_MVT_additive_functions}, and \ref{cor_dynamical_MVT_several_additive_functions}.
Also, as promised in the introduction, we include here a derivation from \cref{thm_dynamical_MVT_fg_sue} of the fact that any bounded and finitely generated multiplicative function has a mean value (see \cref{prop_fg_mult_fctn}).

\begin{proof}[Proof of \cref{cor_ortho_Omega_rational_phases}]
Let $(X,\mu,T)$ be uniquely ergodic.
We want to show that
\begin{equation}
\lim_{N\to\infty}\BEu{n\in [N]} f\big(T^{\Omega(mn+r)}x\big)
~=~\int f\d\mu
\end{equation}
for every $x\in X$, $f\in\Cont(X)$, $m\in\N$, and $r\in\{0,1,\ldots,m-1\}$.
Fix $m\in\N$, and $r\in\{0,1,\ldots,m-1\}$. We can assume without loss of generality that $m$ and $r$ are coprime. Otherwise, there exists $s\in\N$ such that $m=sm'$ and $r=sr'$ with $\gcd(m',r')=1$ and hence
$$
\lim_{N\to\infty}\BEu{n\in [N]}f\big(T^{\Omega(mn+r)}x\big)
\,=\,\lim_{N\to\infty}\BEu{n\in [N]}f\big(T^{\Omega(s)+\Omega(m'n+r')}x\big)
\,=\,\int f\circ T^{\Omega(s)}\d\mu\,=\,\int f \d\mu.
$$
Let $(\Z/m\Z)^*=\{0\leq s\leq m-1: \gcd(s,m)=1\}$ denote the multiplicative group of primitive residue classes modulo $m$. Define  $Y\coloneqq X\times (\Z/m\Z)^*$ and
\[
S_p(x,s)=\begin{cases}
(Tx, {ps}\bmod{m}),&\text{if}~p\nmid m
\\
(Tx, {s}\bmod{m}),&\text{if}~p\mid m
\end{cases}
\]
for all  $p\in\P$ and $(x,s)\in Y$. The maps $(S_p)_{p\in \P}$ generate a multiplicative action $S=(S_n)_{n\in\N}$ on $Y$, turning $(S,Y)$ into a multiplicative topological dynamical system. Clearly, this multiplicative system is finitely generated. Additionally, using only the fact that $\sum_{ p\equiv{s}\bmod{m}}1/p=\infty$ for all $s\in (\Z/m\Z)^*$, it is straightforward to check that $(S,Y)$ is strongly uniquely ergodic with respect to the measure $\nu=\mu\otimes \kappa$, where $\kappa$ denotes the normalized counting measure on $(\Z/m\Z)^*$. This means that $(Y,S)$ meets all the requirements of \cref{thm_dynamical_MVT_fg_sue} and so
\[
\lim_{N\to\infty}\BEu{n\in[N]} g\big(S_ny\big)
~=~\int g\d\nu
\]
holds for all $g\in C(Y)$ and $y\in Y$.
To conclude the proof, note that
\begin{align*}
\lim_{N\to\infty}\BEu{n\in[N]} f\big(T^{\Omega(mn+r)}x\big)
~&=~\lim_{N\to\infty}\BEu{n\in[N]\atop n\equiv r\bmod{m}} f\big(T^{\Omega(n)}x\big) 
\\
~&=~\lim_{N\to\infty}\BEu{n\in[N]\atop \gcd(n,m)=1} g\big(S_n(x,1)\big) 
\end{align*}
where $g=f\otimes 1_{\{r\}}$. Then we have
\begin{align*}
\lim_{N\to\infty}\BEu{n\in[N]\atop \gcd(n,m)=1} g\big(S_n(x,1)\big) ~&=~\lim_{N\to\infty}\frac{m}{\varphi(m)}\BEu{n\in[N]} g\big(S_n(x,1)\big) \prod_{p\mid m} (1-1_{p\mid n})
\\
~&=~\frac{m}{\varphi(m)}\sum_{q\mid m} \frac{\mu(q)}{q} \lim_{N\to\infty}\BEu{n\in[N]} g\big(S_{qn}(x,1)\big)
\\
~&=~\left(\frac{m}{\varphi(m)}\sum_{q\mid m} \frac{\mu(q)}{q} \right)\biggl(\int f\d\mu\biggr)
\\
~&=~\int f\d\mu.
\end{align*}
This finishes the proof.
\end{proof}

\begin{proof}[Proof of \cref{cor_dynamical_MVT_additive_functions}]
It suffices to prove the corollary for the special case when $a\colon\N\to\N\cup\{0\}$ is completely additive instead of just additive, because the general case will then follow from \cref{lem_additive_fctns_transferrence}. Let $a\colon\N\to\N\cup\{0\}$ be a completely additive function such that $\{a(p):p\in\P\}$ is finite and set $P_0\coloneqq \{p\in\P: a(p)\neq 0\}$.
Suppose $\sum_{p\in P_0}1/p=\infty$ and assume $(X,\mu,T^{a(p)})$ is uniquely ergodic for all $p\in P_0$.
Let us write for short $T^a\coloneqq (T^{a(n)})_{n\in\N}$. Then $(X,T^a)$ is a multiplicative topological dynamical system.
Since $\{a(p):p\in\P\}$ is finite, the system $(X,T^a)$ is finitely generated.
Moreover, since $(X,\mu,T^{a(p)})$ is uniquely ergodic for all $p\in P_0$ and $\sum_{p\in P_0}1/p=\infty$, we conclude that $(X,T^a)$ is strongly uniquely ergodic. Therefore, we can apply \cref{thm_dynamical_MVT_fg_sue} to the system $(X,T^a)$ and \eqref{eqn_ud_additive_functions_2} follows.
\end{proof}

\begin{proof}[Proof of \cref{cor_dynamical_MVT_several_additive_functions}]
Assume $(X,\mu,T_1,\ldots,T_k)$ is totally uniquely ergodic and we are given completely additive functions $a_1,\ldots,a_k\colon\N\to\N\cup\{0\}$ satisfying properties \ref{itm_several_add_fctns_i} and \ref{itm_several_add_fctns_ii} of \cref{cor_dynamical_MVT_several_additive_functions}. Define a multiplicative system $(Y,S)$ as
\begin{align}
\label{eqn_def_YS_sev_add_fctns_1}    
Y=X\qquad\text{and}\qquad S_n=T_1^{a_1(n)}\circ \cdots \circ T_k^{a_k(n)}.
\end{align}
It follows immediately from property \ref{itm_several_add_fctns_i} that $(Y,S)$ is finitely generated. On top of that, we claim that $(Y,S)$ is strongly uniquely ergodic. To verify this claim, it suffices to show that $\mu$ is the only Borel probability measure on $Y$ that pretends to be invariant under the multiplicative action $S=(S_n)_{n\in\N}$. Indeed, let $\nu$ be an arbitrary Borel probability measure on $Y$ that pretends to be invariant under $S$. By definition, this means there exists a set $P\subset \P$ with $\sum_{p\notin P}1/p <\infty$ and $S_p\nu=\nu$ for all $p\in P$. Let $\Gamma\coloneqq \{(n_1,\ldots,n_k): T^{n_1}\cdots T^{n_k}\nu=\nu\}$. Observe that $\Gamma$ is a subgroup of $\Z^k$ containing the set $\{(a_1(p),\ldots,a_k(p)): p\in P\}$ as a subset. Also, we cannot have $\mathrm{rank}(\Gamma)<k$ because this would contradict property \ref{itm_several_add_fctns_ii}. Hence $\Gamma$ must have full rank in $\Z^k$. Since $\Gamma$ has full rank and $\nu$ is invariant under $\Gamma$, it follows that $\nu=\mu$ because $(X,\mu,T_1,\ldots,T_k)$ is totally uniquely ergodic. This completes the proof that $(Y,S)$ is strongly uniquely ergodic.
The conclusion of \cref{cor_dynamical_MVT_several_additive_functions} now follows directly from \cref{thm_dynamical_MVT_fg_sue} applied to the system $(Y,S)$ defined in \eqref{eqn_def_YS_sev_add_fctns_1}.
\end{proof}

The following variant of \cref{lem_additive_fctns_transferrence}, where additivity is replaced by multiplicativity and shift-invariance is replaced by dilation-invariance, is needed for the proof of \cref{prop_fg_mult_fctn}.

\begin{Lemma}
\label{lem_mult_fctns_transferrence}
Suppose $\mathcal{P}$ is a subset of $\P$ with the property that $\sum_{p\in\mathcal{P}}1/p<\infty$.
Let $b_1\colon\N\to\C$ be a completely multiplicative function, $b_2\colon\N\to\C$ a multiplicative function, and assume $b_1(p)=b_2(p)$ for all $p\in\P\setminus\mathcal{P}$. Let $\mathcal{H}$ be a dilation-invariant\footnote{
A collection of functions $\mathcal{H}$ from $\C$ to $\C$ is called \define{dilation-invariant} if for any $h\in\mathcal{H}$ and any $w\in \C$ the function $z\mapsto h(wz)$ belongs to $\mathcal{H}$.} collection of bounded functions from $\C$ to $\C$. The following two statements hold.
\begin{enumerate}	
[label=(\roman*),ref=(\roman*),leftmargin=*]
\item
\label{itm_mult_fctns_transferrence_i}
If $\lim_{N\to\infty}\E_{n\in[N]} h(b_1(n))$ exists for all $h\in\mathcal{H}$, then $\lim_{N\to\infty}\E_{n\in[N]} h(b_2(n))$ also exists for all $h\in\mathcal{H}$;
\item
If $\lim_{N\to\infty}\E_{n\in[N]} h(b_1(n))=L$ for all $h\in\mathcal{H}$, then $\lim_{N\to\infty}\E_{n\in[N]} h(b_2(n))=L$ for all $h\in\mathcal{H}$.
\end{enumerate}
\end{Lemma}
The proof of \cref{lem_mult_fctns_transferrence} is almost identical to the proof of \cref{lem_additive_fctns_transferrence} and is omitted.
\begin{Proposition}
\label{prop_fg_mult_fctn}
Any bounded and finitely generated multiplicative function $b\colon\N\to \C$ has a mean value, i.e., $\lim_{N\to\infty}\E_{n\in[N]}b(n)$ exists.
\end{Proposition}

\begin{proof}
Let $b\colon\N\to \C$ be a bounded and finitely generated multiplicative function. Define $Z=\{b(p): p\in\P \}$ and note that $Z$ is a finite set because $b$ is finitely generated.
Also, let $Z'=\{z\in Z: \sum_{b(p)=z}1/p =\infty\}$.
Any $z\in Z'$ must satisfy $|z|\leq 1$, because otherwise $b$ would be unbounded. Also, if there exists some $z\in Z'$ with $|z|<1$ then one can show that the mean value of $b$ exists and equals zero (see \cite[Lemma 2.9]{BKLR18} for details). Thus, it suffices to deal with the case when all $z\in Z'$ satisfy $|z|=1$. Let $\mathcal{P}=\{p\in P: b(p)\in Z'\}$ and let $b^*$ denote the completely multiplicative function uniquely determined by $b^*(p)=b(p)$ for $p\in\mathcal{P}$ and $b^*(p)=1$ otherwise. Using part \ref{itm_mult_fctns_transferrence_i} of
\cref{lem_mult_fctns_transferrence} (with $\mathcal{P}$ as we have just defined and $\mathcal{H}=\{z\mapsto wz: w\in \C\}$), we conclude that $b$ has a mean value if $b^*$ does. Finally, note that the multiplicative rotation $z\mapsto b^*(n)z$ defined in \cref{ex_2}, part \ref{remark_itm_2}, is finitely generated, because $b^*$ is a finitely generated, and strongly uniquely ergodic, because the support of this multiplicative rotation is the closure of the group generated by $Z\subset S^1$, whose Haar measure is the only Borel probability measure that pretends to be invariant under the action induced by $b^*$ (cf.\ \cref{def_fg_sue}).
Therefore, it follows form \cref{thm_dynamical_MVT_fg_sue} that $b^*$ has a mean value, as desired.
\end{proof}

\section{Disjointness of additive and multiplicative semigroup actions}
\label{sec_disjointness_add_mult_actions}

In this section we prove \cref{thm_ortho_fg_sue_nilsequences} (in \cref{sec_proof_thm_C}), \cref{thm_ortho_fg_sue_horocycle} (in \cref{sec_horocycle}), and Corollaries \ref{cor_ortho_Omega_linear_phases}, \ref{cor_ortho_Omega_Besicovitch_ap_fctns}, \ref{cor_polynomoial_joint_ud}, and \ref{cor_ortho_fg_sue_nilsequences} (in \cref{sec_applications_thm_C}).

\subsection{Nilsystems, nilsequences, and a proof of \cref{thm_ortho_fg_sue_nilsequences}}
\label{sec_proof_thm_C}

Let $G$ be a Lie group with identity $1_G$.
The \define{lower central series} of $G$ is the sequence
\[
G=G_1 \trianglerighteq G_2\trianglerighteq G_3\trianglerighteq \ldots\trianglerighteq\{1_G\}
\]
where $G_{i+1}:=[G_i,G]$ is defined as the subgroup of $G$ generated by all commutators $aba^{-1}b^{-1}$ with $a\in G_i$ and $b\in G$.
If $G_{s+1}=\{1_G\}$ for some finite $s\in\N$ then $G$ is called \define{($s$-step) nilpotent}.
Each $G_i$ is a closed and normal subgroup of $G$ (cf. \cite[Section 2.11]{Leibman05a}).

Given a ($s$-step) nilpotent Lie group $G$ and a \define{uniform}\footnote{A closed subgroup
$\Gamma$ of $G$ is called \define{uniform} if $G/\Gamma$ is compact or, equivalently,
if there exists a compact set $K$ such that $K\Gamma = G$.}
and \define{discrete} subgroup $\Gamma$ of $G$, the quotient space
$X\coloneqq G/\Gamma$ is called a \define{($s$-step) nilmanifold}.
The group $G$ acts continuously and transitively on $X$ via
left-multiplication, i.e., for any $x \in X$ and $a\in G$ we have $a\cdot x= (ab)\Gamma$ where $b$ is any element of $G$ such that $x=b\Gamma$. 
Any map $T\colon X\to X$ of the from $T(x)= g\cdot x$, $x\in X$, where $g$ is a fixed element of $G$, is called a \define{niltranslation}. The pair $(X,T)$ is an additive topological dynamical system called a \define{nilsystem}.
Any nilmanifold $X=G/\Gamma$ possesses a unique
$G$-invariant probability measure $\mu$ called the \define{Haar measure on X}
(see \cite[Lemma 1.4]{Raghunathan72}).

Let us state some classical results regarding the dynamics of niltranslation.
\begin{Proposition}
[{see {\cite[Th{\'e}or{\`e}me 2]{Lesigne91}} in the case of connected $G$ and {\cite[Theorem 2.19]{Leibman05a}} in the general case, cf.~also \cite{AGH63, Parry69}}]
\label{prop_dynamics-nilrotation}
Suppose $(X,T)$ is a nilsystem and $\mu$ is the Haar measure on $X$. Then the following are equivalent:
\begin{enumerate}	
[label=(\roman*),ref=(\roman*),leftmargin=*]
\item\label{item:nilrotation-i}
$(X,T)$ is transitive\footnote{An additive topological dynamical system $(X,T)$ is called \define{transitive} if there exists at least one point with a dense orbit.};
\item\label{item:nilrotation-ii}
$(X,\mu,T)$ is uniquely ergodic;
\end{enumerate}
Moreover, the following are equivalent:
\begin{enumerate}	
[label=(\roman*),ref=(\roman*),leftmargin=*]
\setcounter{enumi}{2}
\item
$X$ is connected and $(X,\mu,T)$ is uniquely ergodic.
\item
$(X,\mu,T)$ is totally uniquely ergodic\footnote{An additive topological dynamical system $(X,\mu,T)$ is called \define{totally uniquely ergodic} if $(X,\mu,T^m)$ is uniquely ergodic for every $m\in\N$}.
\end{enumerate}
\end{Proposition}

We will also make use of vertical characters:
Let $G=G_1 \trianglerighteq G_2\trianglerighteq
\ldots\trianglerighteq G_{s}\trianglerighteq\{1_G\}$ be the lower central series of $G$. 
A function $f\in \Cont(X)$ is called a \define{vertical character} if there exists a continuous group homomorphism $\chi\colon G_s\to \{z\in\C: |z|=1\}$ satisfying $\chi(\gamma)=1$ for all $\gamma\in G_s\cap \Gamma$ and such that $f(tx)=\chi(t)f(x)$ for all $t\in G_s$ and $x\in X$.

For the proof of \cref{thm_ortho_fg_sue_nilsequences} we will make use of the following number-theoretic orthogonality criterion.

\begin{Proposition}[see {\cite[Proposition 4]{BKLR19b}} and {\cite[Theorem 2]{BSZ13}}; cf.\ also \cite{Katai86,Daboussi75}]
\label{prop_katai}
Let $a\colon \N\to\C$ be bounded and $P\subset\P$ with $\sum_{p\in P}1/p=\infty$. If
\begin{equation*}
\label{eqn_katai_criterion}
\lim_{N\to\infty}\, \BEu{n\in [N]} a(pn)\,\overline{a(qn)}~=~0
\end{equation*}
for all $p,q\in P$ with $p\neq q$ then
\begin{equation*}
\label{eqn_katai_conclusion}
\lim_{N\to\infty} \BEu{n\in [N]} a(n)~=~0.
\end{equation*}
\end{Proposition}

\begin{proof}[Proof of \cref{thm_ortho_fg_sue_nilsequences}]
Let $(Y,S)$ be a finitely generated multiplicative topological dynamical system and let $(X,T)$ be nilsystem where $X=G/\Gamma$ for some $s$-step nilpotent Lie group $G$ and uniform and discrete subgroup $\Gamma\subset G$.
Our goal is to show that
\begin{equation}
\label{eqn_ortho_fg_sue_nilsequences_1}
\lim_{N\to\infty}\left(\BEu{n\in [N]} f(T^n x)g(S_n y)-
\left(\BEu{n\in [N]} f(T^n x)\right)
\left(\BEu{n\in [N]} g(S_ny)\right)\right)=0
\end{equation}
for all $x\in X$, $f\in\Cont(X)$, $y\in Y$, and $g\in\Cont(Y)$.
By replacing $X$ with the orbit closure of $x$ if necessary, we can assume without loss of generality that the orbit of $x$ is dense in $X$. According to \cref{prop_dynamics-nilrotation}, this implies that $(X,\mu,T)$ is uniquely ergodic, where $\mu$ is the Haar measure on $X$.

By assumption, either $(X,T)$ or $(Y,S)$ is aperiodic. Let us first deal with the case where $(X,T)$ is aperiodic. Note that in this case $(X,\mu,T)$ is totally uniquely ergodic because it is both uniquely ergodic and aperiodic.

We will prove \eqref{eqn_ortho_fg_sue_nilsequences_1} by induction on the nilpotency step $s$. If $s$ is zero then $(X,T)$ is the trivial system and \eqref{eqn_ortho_fg_sue_nilsequences_1} holds trivially. Let us therefore assume $s\geq 1$ and \eqref{eqn_ortho_fg_sue_nilsequences_1} has already been proven for all nilsystems of step $s-1$.
Since functions of the form $\{g\in\Cont(Y): |g|=1\}$ separate points on $Y$ (for example the function $z\mapsto e\big(\frac{d_Y(z,y)}{2d_Y(x,y)}\big)$ separates the points $x$ and $y$, where $d_Y$ is a metric on $Y$), we have by the Stone-Weierstrass theorem that the algebra of functions generated by $\{g\in\Cont(Y): |g|=1\}$ is uniformly dense in $\Cont(Y)$. Hence, in order to prove \eqref{eqn_ortho_fg_sue_nilsequences_1}, we can assume without loss of generality that $|g|=1$.
It is also not hard to show that the class of vertical characters separate points on $X$ and so the algebra generated by them is uniformly dense in $\Cont(X)$. This allows us to assume that $f$ is a vertical character, which means there exists a continuous group homomorphism $\chi\colon G_s\to \{z\in\C: |z|=1\}$ satisfying $\chi(\gamma)=1$ for all $\gamma\in G_s\cap \Gamma$ such that $f(tx)=\chi(t)f(x)$ for all $t\in G_s$ and $x\in X$.

If $\chi$ is trivial (meaning $\chi(t)=1$ for all $t\in G_s$) then $f$ is $G_s$-invariant and can be viewed as a continuous function on the quotient space $X/G_s$. Since $X/G_s$ is a $(s-1)$-step nilmanifold, it follows from the induction hypothesis that \eqref{eqn_ortho_fg_sue_nilsequences_1} holds.
Therefore we only have to deal with the case when $\chi$ is non-trivial.

It was shown in \cite[p.\ 102]{FH17} that if $T$ is totally uniquely ergodic and $f$ is a vertical character with non-trivial vertical frequency $\chi$ then for all pairs of distinct primes $p$ and $q$ and all $x\in X$ one has
\begin{equation}
\label{eqn_ortho_fg_sue_nilsequences_2}
\lim_{N\to\infty} \BEu{n\in [N]} f(T^{pn}x)\overline{f(T^{qn}x)}~=~0.
\end{equation}
In fact, the very same argument used in \cite[p.\ 102]{FH17} gives a little bit more. Indeed, the proof in \cite[p.\ 102]{FH17} utilizes the fact that the sequence $(T^{pn}x,T^{qn}x)$ distributes uniformly in a sub-nilmanifold $Y\subset X\times X$ for which it is shown that $f\otimes f$ is a vertical character. For the same reasons, $f_1\otimes f_2$ is a vertical character of $Y$ whenever $f_1$ and $f_2$ are vertical characters of $X$. This implies the following slight strengthening of \eqref{eqn_ortho_fg_sue_nilsequences_2}: If $T$ is totally uniquely ergodic and $f_1,f_2$ are two vertical characters with the same non-trivial vertical frequency $\chi$ then for all pairs of distinct primes $p$ and $q$ and all $x\in X$ one has
\begin{equation}
\label{eqn_ortho_fg_sue_nilsequences_2_improved}
\lim_{N\to\infty} \BEu{n\in [N]} f_1(T^{pn}x)\overline{f_2(T^{qn}x)}~=~0.
\end{equation}
(Although we do not need \eqref{eqn_ortho_fg_sue_nilsequences_2_improved} right away, it will be needed later on in the proof when dealing with the case when $(Y,S)$ is aperiodic.)

Let $R_1,\ldots,R_d$ denote the generators of $S$ and define $P_e\coloneqq \{p\in\P: S_p=R_e\}$ for all $e\in[d]$. For at least one $e\in[d]$ we must have $\sum_{p\in P_e}1/p=\infty$. Note that for $p,q\in P_e$,
\begin{equation*}
\begin{split}
f(T^{pn}x)g(S_{pn}y)\overline{f(T^{qn}x)g(S_{qn}y)} & ~=~
f(T^{pn}x)g(R_eS_{n}y)\overline{f(T^{qn}x)g(R_eS_{n}y)}
\\
&~=~f(T^{pn}x)\overline{f(T^{qn}x)},
\end{split}
\end{equation*}
where we have used $g(R_eS_{n}y)\overline{g(S_{R_en}y)}=1$ because $|g|=1$. Therefore \eqref{eqn_ortho_fg_sue_nilsequences_2} implies 
\begin{equation}
\label{eqn_ortho_fg_sue_nilsequences_3}
\lim_{N\to\infty} \BEu{n\in [N]} f(T^{pn}x)g(S_{pn}y)\overline{f(T^{qn}x)g(S_{qn}y)}~=~0
\end{equation} 
for all $p,q\in P_e$ with $p\neq q$. In light of \cref{prop_katai}, we have
\begin{equation*}
\lim_{N\to\infty} \BEu{n\in [N]} f(T^{n}x)g(S_{n}y)~=~0.
\end{equation*}
Since $\lim_{N\to\infty}\mathbb{E}_{n\in[N]} f(T^nx)=0$ (because $f$ is a vertical character with non-trivial $\chi$ and hence $\int f\d\mu=0$), \eqref{eqn_ortho_fg_sue_nilsequences_1} follows.

Next, let us deal with the case when $(Y,S)$ is aperiodic. By compactness, any nilmanifold has only finitely many connected components, say $X_0, X_1,\ldots,X_{m-1}$. Since $T$ acts ergodically on $X$, it cyclically permutes these connected components. In particular, after a potential reordering of $X_0,X_1,\ldots, X_{m-1}$, we have $T^n (X_i)= X_{i+n\bmod m}$ for all $n\in\N$.

Let $\pi\colon G\to X$ denote the natural projection map from $G$ onto $X=G/\Gamma$ and choose $a\in G$ such that $Tx=ax$ for all $x\in X$. Let $G^\circ$ denote the identity component of $G$. Then $\pi(G^\circ)=X_0$ (cf.\ \cite[Subsection 2.1]{Leibman05a}).
Let $\langle G^\circ, a \rangle$ denote the smallest Lie group generated by $G^\circ$ and $a$. Since translation by $a$ acts ergodically on the connected components, we have $\pi(\langle G^\circ, a \rangle)=X$. Therefore we can assume without loss of generality that $G=\langle G^\circ, a \rangle$.

Let
\[
G=G_1 \trianglerighteq G_2\trianglerighteq G_3\trianglerighteq \ldots G_s\trianglerighteq G_{s+1}=\{1_G\}
\]
be the lower central series of $G$.
If $s=1$ then $(X,T)$ is merely a group rotation, in which case \eqref{eqn_ortho_fg_sue_nilsequences_1} follows from
\begin{equation}
\label{eqn_ortho_fg_sue_nilsequences_1_2}
\lim_{N\to\infty}\, \BEu{n\in [N]} e(n\alpha)g(S_n y)~=~0,\qquad\forall\alpha\in\R\setminus\Z.
\end{equation}
Since we have already proved \eqref{eqn_ortho_fg_sue_nilsequences_1} for the case where $(X,T)$ is aperiodic, it follows that $\lim_{N\to\infty}\mathbb{E}_{n\in [N]} e(n\alpha)g(S_n y)=0$ whenever $\alpha$ is irrational. If $\alpha$ is rational then $\lim_{N\to\infty}\mathbb{E}_{n\in [N]} e(n\alpha)g(S_n y)=0$ follows directly from the aperiodicity assumption on $(Y,S)$. Therefore \eqref{eqn_ortho_fg_sue_nilsequences_1_2} holds. 

It remains to deal with the case $s\geq 2$, for which we will use induction on $s$.
Since $G^\circ$ is a normal subgroup of $G$ and $G=\langle G^\circ, a \rangle$, it follows that $G_2$, the second element in the lower central series of $G$, is generated by $[G^\circ, G^\circ]\cup [a^\Z, G^\circ]$. Since $[G^\circ, G^\circ]\cup [a^\Z, G^\circ]$ is connected and since any group generated by a connected set is connected, we conclude that $G_2$ is connected. A similar argument can be used to show that $G_i$ is connected for all $i\in\{2,\ldots,s\}$. 

To prove \eqref{eqn_ortho_fg_sue_nilsequences_1} we can once again assume that $f$ is a vertical character with a non-trivial vertical frequency, i.e., there exists a non-trivial continuous group homomorphism $\chi\colon G_s\to \{z\in\C: |z|=1\}$ satisfying $\chi(\gamma)=1$ for all $\gamma\in G_s\cap \Gamma$ such that $f(tx)=\chi(t)f(x)$ for all $t\in G_s$ and $x\in X$.
Since $G_s$ is connected, the action of $G_s$ on $X$ preserves the connected components $X_0,\ldots, X_{m-1}$. Therefore the restriction of $f$ onto the $r$-th connected component is a vertical character of the sub-nilmanifold $X_r$ with a non-trivial vertical frequency. Since $T^m$ is a totally uniquely ergodic niltranslation on $X_r$, it follows from \eqref{eqn_ortho_fg_sue_nilsequences_2_improved} with $f_1= f\circ T^{pr}$ and $f_2= f\circ T^{qr}$ that
\begin{equation}
\label{eqn_ortho_fg_sue_nilsequences_2_2}
\lim_{N\to\infty} \BEu{n\in [N]} f(T^{p(mn+r)}x)\overline{f(T^{q(mn+r)}x)}~=~0.
\end{equation} 
for all $r\in\{0,1,\ldots,m-1\}$ and all distinct primes $p$ and $q$. Since \eqref{eqn_ortho_fg_sue_nilsequences_2_2} holds for all $r\in\{0,1,\ldots,m-1\}$, it implies that
\begin{equation*}
\label{eqn_ortho_fg_sue_nilsequences_2_3}
\lim_{N\to\infty} \BEu{n\in [N]} f(T^{pn}x)\overline{f(T^{qn}x)}~=~0.
\end{equation*} 
One can now use the same argument as above to conclude that
\begin{equation*}
\lim_{N\to\infty} \BEu{n\in [N]} f(T^{n}x)g(S_{n}y)~=~0,
\end{equation*}
from which \eqref{eqn_ortho_fg_sue_nilsequences_1} follows.
\end{proof}

\subsection{Horocyclic flow and a proof of \cref{thm_ortho_fg_sue_horocycle}}
\label{sec_horocycle}

The proof of \cref{thm_ortho_fg_sue_horocycle} relies on some facts proved in \cite{BSZ13} as well as \cref{thm_ortho_fg_sue_nilsequences}.

\begin{proof}[Proof of \cref{thm_ortho_fg_sue_horocycle}]
Let $G\coloneqq SL_2(\R)$, let $\Gamma$ be a lattice in $G$, and let $u\coloneqq \left[\begin{smallmatrix} 1&1\\ 0&1 \end{smallmatrix}\right]$. Let $X\coloneqq G/\Gamma$.
Let $(Y,S)$ be a finitely generated and aperiodic multiplicative topological dynamical system. We want to show that
\begin{equation}
\label{eqn_ortho_fg_sue_horocycle_1}
\lim_{N\to\infty}\left(\BEu{n\in [N]} f(u^n x)g(S_n y)-
\left(\BEu{n\in [N]} f(u^n x)\right)
\left(\BEu{n\in [N]} g(S_ny)\right)\right)=0
\end{equation}
for all $x\in X$, $f\in\Cont(X)$, $y\in Y$, and $g\in\Cont(Y)$.

If $x$ is not generic\footnote{In our context, we call a point $x\in X$ \define{generic for the Haar measure $\mu$} if for all $f\in \Cont(X)$ we have $\lim_{N\to\infty}\mathbb{E}_{n\in[N]} f(u^nx)=\int f\d\mu$.} for the Haar measure $\mu$ on $X$ then the orbit closure of $x$ under $u$, $\overline{\{u^nx:n\in\N\}}$, is either finite or a circle.
Therefore, the action of $u$ on $\overline{\{u^nx:n\in\N\}}$ is either a finite cyclic rotation or an irrational circle rotation (cf.\ \cite[p.\ 14]{BSZ13}). In this case, \eqref{eqn_ortho_fg_sue_horocycle_1} follows from \cref{thm_ortho_fg_sue_nilsequences}.
If $x$ is generic for the Haar measure $\mu$ then it follows from Corollaries 5 and  6 in \cite{BSZ13} that for all $f\in\Cont(X)$ and all but finitely many pairs of distinct primes $p$ and $q$ we have
\begin{equation}
\label{eqn_ortho_fg_sue_horocycle_2}
\lim_{N\to\infty}\BEu{n\in [N]} f(u^{pn}x)\overline{f}(u^{qn}x) ~=~\int\int f(x)\overline{f}(y)\d\mu(x)\d\mu(y)~=~\left|\int f\d\mu\right|^2.
\end{equation}

Arguing as in the proof of \cref{thm_ortho_fg_sue_nilsequences}, we see that it suffices to prove \eqref{eqn_ortho_fg_sue_horocycle_1} for the special case where $|g(y)|=1$ for all $y\in Y$.
Also, by replacing $f$ with $f-\int f\d\mu$, we can assume without loss of generality that $\int f\d\mu=0$.

Let $R_1,\ldots,R_d$ denote the generators of $S$ and define $P_e\coloneqq \{p\in\P: S_p=R_e\}$, $e\in[d]$. For at least one $e\in[d]$ we must have $\sum_{p\in P_e}1/p=\infty$. Note that for $p,q\in P_e$,
\begin{equation*}
\begin{split}
f(u^{pn}x)g(S_{pn}y)\overline{f(u^{qn}x)g(S_{qn}y)} & ~=~
f(u^{pn}x)g(R_eS_{n}y)\overline{f(u^{qn}x)g(R_eS_{n}y)}
\\
&~=~f(u^{pn}x)\overline{f(u^{qn}x)},
\end{split}
\end{equation*}
where we have used $g(R_eS_{n}y)\overline{g(R_eS_{n}y)}=1$ because $|g|=1$. Therefore \eqref{eqn_ortho_fg_sue_horocycle_2} implies 
\begin{equation*}
\lim_{N\to\infty} \BEu{n\in [N]} f(T^{pn}x)g(S_{pn}y)\overline{f(T^{qn}x)g(S_{qn}y)}~=~0
\end{equation*} 
for all but finitely many $p,q\in P_e$ with $p\neq q$. In light of \cref{prop_katai}, we thus have
\begin{equation*}
\lim_{N\to\infty} \BEu{n\in [N]} f(T^{n}x)g(S_{n}y)~=~0.
\end{equation*}
Since $x$ is generic and $\int f\d\mu=0$ we have $\lim_{N\to\infty}\mathbb{E}_{n\in[N]} f(T^nx)=0$ and so \eqref{eqn_ortho_fg_sue_horocycle_1} follows.
\end{proof}

\subsection{Proofs of Corollaries \ref{cor_ortho_Omega_linear_phases}, \ref{cor_ortho_Omega_Besicovitch_ap_fctns}, \ref{cor_polynomoial_joint_ud}, and \ref{cor_ortho_fg_sue_nilsequences}}
\label{sec_applications_thm_C}

For the proofs of Corollaries \ref{cor_ortho_Omega_linear_phases} and \ref{cor_ortho_Omega_Besicovitch_ap_fctns} we need the following lemma:

\begin{Lemma}
\label{lem_Omega_derived_systems_are_aperiodic}
Let $(X,T)$ be a uniquely ergodic topological dynamical system. Then the multiplicative topological dynamical system $(X,T^\Omega)$ is aperiodic.
\end{Lemma}

\begin{proof}
To verify that $(X,T^\Omega)$ is aperiodic we have to show that
\begin{equation}
\label{eqn_Omega_derived_systems_are_aperiodic}
\lim_{N\to\infty}\BEu{n\in [N]} e(n\alpha)\, f\big(T^{\Omega(n)}x\big)
~=~0
\end{equation}
for every $f\in\Cont(X)$, $x\in X$, and $\alpha\in\Q\setminus\Z$.
Pick $m\in\N$ such that $m \alpha\in\Z$.
Then
\begin{align*}
\lim_{N\to\infty}\BEu{n\in [N]} e(n\alpha)\, f\big(T^{\Omega(n)}x\big)
&=
\BEu{r\in [m]} e(r\alpha) \left(\lim_{N\to\infty}\BEu{n\in [N]} f\big(T^{\Omega(mn+r)}x\big)\right)
\\&=
\left(\BEu{r\in [m]} e(r\alpha)\right) \left(\int f\d\mu\right)
\\&=0,
\end{align*}
where the second to last equality follows from
\cref{cor_ortho_Omega_rational_phases}.
\end{proof}

\begin{proof}[Proof of \cref{cor_ortho_Omega_linear_phases}]
Let $(X,T)$ be a uniquely ergodic additive topological dynamical system and let $\mu$ denote the corresponding unique $T$-invariant Borel probability measure on $X$. Our goal is to show that
\begin{equation}
\label{eqn_ortho_Omega_linear_phases_1}
\lim_{N\to\infty}\BEu{n\in [N]} e(n\alpha)\, f\big(T^{\Omega(n)}x\big)
~=~
\begin{cases}
\int f\d\mu,&\text{if}~\alpha\in\Z
\\
0,&\text{if}~\alpha\in\R\setminus\Z
\end{cases}
\end{equation}
for every $f\in\Cont(X)$ and every $x\in X$.
Since rotation by $\alpha$ is a ($1$-step) nilsystem and $(X,T^\Omega)$ is finitely generated and aperiodic (due to \cref{lem_Omega_derived_systems_are_aperiodic}), it follows from \cref{thm_ortho_fg_sue_nilsequences} that
\begin{equation}
\label{eqn_ortho_Omega_linear_phases_2}
\lim_{N\to\infty}\left(\BEu{n\in [N]} e(n\alpha) f\big(T^{\Omega(n)}x\big)-
\left(\BEu{n\in [N]} e(n\alpha)\right)
\left(\BEu{n\in [N]} f\big(T^{\Omega(n)}x\big)\right)\right)=0.
\end{equation}
By \cref{thm_dynamical_MVT_Omega} we have 
\begin{equation}
\label{eqn_ortho_Omega_linear_phases_3}
\lim_{N\to\infty} \BEu{n\in [N]}f\big(T^{\Omega(n)}x\big)~=~\int f\d\mu,
\end{equation}whereas
\begin{equation}
\label{eqn_ortho_Omega_linear_phases_4}
\lim_{N\to\infty} \BEu{n\in [N]} e(n\alpha)~=~
\begin{cases}
1,&\text{if}~\alpha\in\Z
\\
0,&\text{if}~\alpha\in\R\setminus\Z
\end{cases}.
\end{equation}
Putting together \eqref{eqn_ortho_Omega_linear_phases_2}, \eqref{eqn_ortho_Omega_linear_phases_3}, and \eqref{eqn_ortho_Omega_linear_phases_4} proves \eqref{eqn_ortho_Omega_linear_phases_1}.
\end{proof}

\begin{proof}[Proof of \cref{cor_ortho_Omega_Besicovitch_ap_fctns}]
Let $a\colon\N\to\C$ be a Besicovitch almost periodic function.
In view of \cref{cor_ortho_Omega_linear_phases}, for any trigonometric polynomial $P(n)\coloneqq c_1 e(n\alpha_1)+\ldots+c_L e(n\alpha_L)$ we have
\begin{equation*}
\lim_{N\to\infty}\BEu{n\in [N]} P(n)\,f\big(T^{\Omega(n)}x\big)
~=~ M(P)\left(
\int f\d\mu\right)
\end{equation*}
for every $f\in\Cont(X)$ and every $x\in X$.
Since $a$ can be approximated by trigonometric polynomials, \eqref{eqn_ortho_Omega_Bes_ap} follows.
\end{proof}

Before proving \cref{cor_polynomoial_joint_ud}, we turn our attention to \cref{cor_ortho_fg_sue_nilsequences}.

\begin{proof}[Proof of \cref{cor_ortho_fg_sue_nilsequences}]
From \cref{thm_ortho_fg_sue_nilsequences} it follows that
\begin{equation*}
\lim_{N\to\infty}\left(\BEu{n\in [N]} \eta(n) g(S_ny) -
\left(\BEu{n\in [N]} \eta(n)\right)
\left(\BEu{n\in [N]} g(S_ny)\right)\right)=0
\end{equation*}
for all $y\in Y$, $g\in\Cont(Y)$, and all nilsequences $\eta\colon\N\to\C$. Since
$$
\lim_{N\to\infty}\BEu{n\in [N]} g(S_ny)~=~\int g\d\nu
$$
by \cref{thm_dynamical_MVT_fg_sue}, we conclude that
\begin{equation*}
\lim_{N\to\infty}\frac{1}{N}\sum_{n=1}^N \eta(n)\,g\big(S_n y\big)
~=~
M(\eta)\left(\int g\d\nu\right).
\end{equation*}
This finishes the proof.
\end{proof}

\begin{proof}[Proof of \cref{cor_polynomoial_joint_ud}]
Let $p(x)=c_k x^k + \ldots + c_1 x + c_0$ and $q(x)=d_\ell x^\ell+\ldots+d_1 x+d_0$ be two polynomials with real coefficients and suppose at least one of the coefficients $c_1,\ldots,c_k$ is irrational and at least one of the coefficients $d_1,\ldots,d_\ell$ is irrational.
It follows from \cref{cor_ortho_fg_sue_nilsequences} that for any uniquely ergodic $(X,\mu,T)$, $x\in X$, and $f\in\Cont(X)$ one has
\begin{equation}
\label{eqn_nil_AM_n_1}
\lim_{N\to\infty}\frac{1}{N}\sum_{n=1}^N e(h p(n))\,f\big(T^{\Omega(n)} x\big)
~=~0,\qquad\forall h\in\Z\setminus\{0\}.
\end{equation}
Let $T(x_1,\ldots,x_\ell)=(x_1+d_\ell, x_2+x_1, x_3+x_2,\ldots, x_\ell+x_{\ell-1})$, $f(x_1,\ldots,x_\ell)=e(m x_\ell)$ for $m\in\N$, $q_\ell(t)=q(t)$, $q_i(t)=q_{i+1}(t+1)-q_i(t)$ for $i=1,2,\ldots,\ell-1$, $x=(q_1(0),\ldots,q_\ell(0))$, and $X=\overline{\{T^n x: n\in\N\}}$. As explained in the proof of \cref{cor_polynomoial_ud_Omega} in \cref{sec_appl_thmA}, with this choice of $X$, $T$, $f$, and $x$, we have that $(X,T)$ is uniquely ergodic and  
\[
f\big(T^{n} x\big)=e(mq(n)).
\]
Hence \eqref{eqn_nil_AM_n_1} implies
\begin{equation*}
\lim_{N\to\infty}\frac{1}{N}\sum_{n=1}^N e(h p(n))\,e(m q(\Omega(n)))
~=~0,\qquad\forall h,m\in\Z\setminus\{0\}.
\end{equation*}
By Weyl's equidistribution criterion, it follows that $\big(p(n),q(\Omega(n))\big)_{n\in\N}$ is uniformly distributed in $\T^2$.
\end{proof}

\section{Entropy of finitely generated multiplicative topological dynamical systems equals zero}
\label{sec_entropy}

Let $(X,T)$ be an additive topological dynamical system. By a \define{finite open cover} of $X$ we mean a finite collection $\mathcal{C}$ of open non-empty sets such that $\bigcup_{C\in \mathcal{C}} C=X$. A \define{subcover} of a finite open cover $\mathcal{C}$ is any subset $\mathcal{D}\subset \mathcal{C}$ that is itself a finite open cover of $X$. Also, given a finite collection $\mathcal{C}_1,\ldots,\mathcal{C}_t$ of finite open covers of $X$, we denote by $\bigvee_{i=1}^t\mathcal{C}_i$ the finite open cover of $X$ given by
$$
\bigvee_{i=1}^t\mathcal{C}_i~\coloneqq~ \{C_1\cap \ldots\cap C_t: C_1\in \mathcal{C}_1,\ldots, C_t\in\mathcal{C}_t\}.
$$
Let $H(\mathcal{C})$ be defined as
$$
H(\mathcal{C})\,\coloneqq\,\min\left\{\frac{\log(|\mathcal{D}|)}{\log(2)}: \text{$\mathcal{D}$ is a subcover of $\mathcal{C}$}\right\}.
$$
One can show that the limit
$$
H(T,\mathcal{C})~\coloneqq~ \lim_{N\to\infty}~\frac{1}{N}\, H\left(\bigvee_{n=1}^N T^{-n}\mathcal{C}\right)
$$
always exists. The \define{(topological) entropy} of the system $(X,T)$ is then defined as
$$
h(T)~\coloneqq~\sup_{\mathcal{C}} H(T,\mathcal{C}),
$$ 
where the supremum is taken over all finite open covers $\mathcal{C}$ of $X$.

For multiplicative topological dynamical systems entropy is defined similarly. A sequence $\Phi=(\Phi_{N})_{N\in\N}$ of finite non-empty subsets of $\N$ is called a \define{\Folner{} sequence} for the semigroup $(\N,\cdot)$ if for every $m\in\N$ one has
$$
\lim_{N\to\infty} \frac{|\Phi_N\,\triangle\, \Phi_N/m|}{|\Phi_N|} ~=~0,
$$
where $\Phi_N/m\coloneqq \{n: mn\in\Phi_N\}$.
Given a multiplicative topological dynamical system $(Y,S)$, an open cover $\mathcal{C}$ of $Y$, and a \Folner{} sequence $\Phi=(\Phi_N)_{N\in\N}$ for $(\N,\cdot)$, consider the quantity
$$
H(S,\mathcal{C},\Phi)~\coloneqq~ \lim_{N\to\infty}~\frac{1}{|\Phi_N|}\, H\left(\bigvee_{n\in \Phi_N} S_n^{-1}\mathcal{C}\right).
$$
The \define{(topological) entropy} of the system $(Y,S)$ is
$$
h(S)~\coloneqq~\sup_{\mathcal{C}} H(S,\mathcal{C},\Phi),
$$ 
where the supremum is taken over all finite open covers $\mathcal{C}$ of $Y$. We remark that the quantity $h(S)$ does not depend on the choice of \Folner{} sequence $\Phi=(\Phi_N)_{N\in\N}$ (cf.\ \cite[Theorem 6.8]{DFR16}).

\begin{Proposition}
\label{prop_fg_implies_zero_entropy}
Any finitely generated multiplicative topological dynamical system has zero entropy.
\end{Proposition}

\begin{proof}
Let $p_n$ denote the $n$-th prime number. Let $(Y,S)$ be a finitely generated multiplicative topological dynamical system and $R_1,\ldots, R_d$ its generators.
Since $h(S)=\sup_{\mathcal{C}} H(S,\mathcal{C})$ is independent of the choice of $\Phi$, let us stick for convenience to the ``standard'' \Folner{} sequence
$$
\Phi_N=\{p_1^{e_1}\cdot\ldots\cdot p_N^{e_N}: 0\leq e_1,\ldots,e_N\leq N-1\},\qquad \forall N\in\N.
$$
Let $\mathcal{C}$ be an arbitrary finite open cover of $Y$. Note that for $n=p_1^{e_1}\cdot\ldots\cdot p_N^{e_N}\in \Phi_N$ we have $S_n=S_{p_1}^{e_1}\circ \cdots\circ S_{p_N}^{e_N}\in \{R_1^{g_1}\circ \cdots \circ R_d^{g_d}: 0\leq g_1,\ldots,g_d\leq (N-1)^2\}$. This implies that
$$
\bigvee_{n\in \Phi_N} S_n^{-1}\mathcal{C}~\subset~ \bigvee_{
0\leq g_1,\ldots,g_d\leq (N-1)^2
}R_1^{-g_1}\circ \cdots \circ R_d^{-g_d}\mathcal{C}.
$$
Since the size of the cover $\bigvee_{
0\leq g_1,\ldots,g_d\leq (N-1)^2
}R_1^{-g_1}\circ \cdots \circ R_d^{-g_d}\mathcal{C}$ is at most $|\mathcal{C}|^{{N^{2d}}}$, we can estimate
$$
H\left(\bigvee_{n\in \Phi_N} S_n^{-1}\mathcal{C}\right)~\leq~
\frac{\log\left(|\mathcal{C}|^{N^{2d}}\right)}{\log 2}~=~N^{2d}\, \frac{\log |\mathcal{C}|}{\log 2}.
$$
Since $|\Phi_N|=N^N$, we conclude that
$$
\lim_{N\to\infty}~\frac{1}{|\Phi_N|}\, H\left(\bigvee_{n\in \Phi_N} S_n^{-1}\mathcal{C}\right)~\leq~ \lim_{N\to\infty} \frac{N^{2d}\, \log|\mathcal{C}|}{N^N\, \log 2}~=~0.
$$
This shows that the entropy of $(Y,S)$ is zero.
\end{proof}


\bibliographystyle{aomalphanomr}
\bibliography{mynewlibrary}


\bigskip
\footnotesize
\noindent
Vitaly Bergelson\\
\textsc{The Ohio State University}\par\nopagebreak
\noindent
\href{mailto:bergelson.1@osu.edu}
{\texttt{bergelson.1@osu.edu}}

\bigskip
\footnotesize
\noindent
Florian K.\ Richter\\
\textsc{{\'E}cole Polytechnique F{\'e}d{\'e}rale de Lausanne} (EPFL)\par\nopagebreak
\noindent
\href{mailto:f.richter@epfl.ch}
{\texttt{f.richter@epfl.ch}}

\end{document}